\documentclass[reqno,11pt]{amsart}
\usepackage{amsmath,physics}
\usepackage{amsxtra}
\usepackage{amscd}
\usepackage{amsthm}
\usepackage{amsfonts}
\usepackage{amssymb}
\usepackage{scalerel}
\usepackage{verbatim}
\usepackage[matrix,arrow,curve]{xy}
\usepackage{a4wide}
\usepackage[T1]{fontenc}

\allowdisplaybreaks[4]
\usepackage{indentfirst}
\usepackage{stmaryrd}
\usepackage{leftidx}
\usepackage{accents}

\usepackage{xcolor}
\usepackage{color}
\usepackage{dynkin-diagrams}
\usepackage{cite}
\usepackage{enumitem}

\usepackage{mathtools,extarrows}
\usepackage{tikz}
\usetikzlibrary{chains}

\tikzset{node distance=2em, ch/.style={circle,draw,on chain,inner sep=2pt},chj/.style={ch,join},every path/.style={shorten >=4pt,shorten <=4pt},line width=1pt,baseline=-1ex}

\let\dlabel=\alabel

\newcommand{\dnode}[2][chj]{%
\node[#1,label={below:\dlabel{#2}}] {};
}

\newcommand{\dnodenj}[1]{%
\dnode[ch]{#1}
}

\newcommand{\dnodebr}[1]{%
\node[chj,label={below right:\dlabel{#1}}] {};
}

\newcommand{\dydots}{%
\node[chj,draw=none,inner sep=1pt] {\dots};
}

\usepackage{xr-hyper}
\usepackage[
pdftex,
bookmarks=false,
colorlinks=true,
debug=true,
pdfnewwindow=true]{hyperref}

\newtheorem{Thm}{Theorem}[section]

\newtheorem{Def}[Thm]{Definition}
\newtheorem{Prop}[Thm]{Proposition}
\newtheorem{Lem}[Thm]{Lemma}
\newtheorem{Cor}[Thm]{Corollary}

\newtheorem{Rk}[Thm]{Remark}
\newtheorem{Not}[Thm]{Notation}

\numberwithin{equation}{section}

\def\QC{U_{v}^{>}(L\mathfrak{g})}

\def\Yangian{Y_{\hbar}^{>}(\mathfrak{g})}
\def\dgdual{\dot{\mathbf{Y}}^{>}_{\hbar}(\mathfrak{g})}

\def\qlc{U_{v}^{>}(L\mathfrak{sp}_{2n})}
\def\qld{U_{v}^{>}(L\mathfrak{o}_{2n})}
\def\integraldl{\mathbf{U}_{v}^{>}(L\mathfrak{o}_{2n})}
\def\integralcl{\mathbf{U}_{v}^{>}(L\mathfrak{sp}_{2n})}
\def\integrald{\mathcal{U}_{v}^{>}(L\mathfrak{o}_{2n})}
\def\integralc{\mathcal{U}_{v}^{>}(L\mathfrak{sp}_{2n})}
\def\integrall{\mathbf{U}_{v}^{>}(L\mathfrak{g})}
\def\integral{\mathcal{U}_{v}^{>}(L\mathfrak{g})}

\newif\ifShowLabels
\ShowLabelstrue
\newdimen\theight
\def\TeXref#1{%
    \leavevmode\vadjust{\setbox0=\hbox{{\tt
        \quad\quad  {\small \rm #1}}}%
    \theight=\ht0
    \advance\theight by \lineskip
    \kern -\theight \vbox to
    \theight{\rightline{\rlap{\box0}}%
    \vss}%
    }}%

\ShowLabelsfalse

\newcommand\nc{\newcommand}
\nc{\unl}{\underline}
\nc{\ol}{\overline}
\nc{\on}{\operatorname}

\nc{\BA}{{\mathbb{A}}}
\nc{\BC}{{\mathbb{C}}}
\nc{\BD}{{\mathbb{D}}}
\nc{\BF}{{\mathbb{F}}}
\nc{\BG}{{\mathbb{G}}}
\nc{\BM}{{\mathbb{M}}}
\nc{\BN}{{\mathbb{N}}}
\nc{\BO}{{\mathbb{O}}}
\nc{\BQ}{{\mathbb{Q}}}
\nc{\BP}{{\mathbb{P}}}
\nc{\BR}{{\mathbb{R}}}
\nc{\BZ}{{\mathbb{Z}}}
\nc{\BS}{{\mathbb{S}}}
\nc{\BK}{{\mathbb{K}}}
\nc{\BW}{W}

\nc{\CA}{{\mathcal{A}}} \nc{\CB}{{\mathcal{B}}} \nc{\CalC}{{\mathcal
C}} \nc{\CalD}{{\mathcal D}} \nc{\CE}{{\mathcal{E}}}
\nc{\CF}{{\mathcal{F}}} \nc{\CG}{{\mathcal{G}}}
\nc{\CH}{{\mathcal{H}}} \nc{\CI}{{\mathcal{I}}}
\nc{\CK}{{\mathcal{K}}} \nc{\CL}{{\mathcal{L}}}
\nc{\CM}{{\mathcal{M}}} \nc{\CN}{{\mathcal{N}}}
\nc{\CO}{{\mathcal{O}}} \nc{\CP}{{\mathcal{P}}}
\nc{\CQ}{{\mathcal{Q}}} \nc{\CR}{{\mathcal{R}}}
\nc{\CS}{{\mathcal{S}}} \nc{\CT}{{\mathcal{T}}}
\nc{\CU}{{\mathcal{U}}} \nc{\CV}{{\mathcal{V}}}
\nc{\CW}{{\mathcal{W}}} \nc{\CX}{{\mathcal{X}}}
\nc{\CY}{{\mathcal{Y}}} \nc{\CZ}{{\mathcal{Z}}}

\nc{\fa}{{\mathfrak{a}}}
\nc{\fb}{{\mathfrak{b}}}
\nc{\fg}{{\mathfrak{g}}}
\nc{\fgl}{{\mathfrak{gl}}}
\nc{\fh}{{\mathfrak{h}}}
\nc{\fj}{{\mathfrak{j}}}
\nc{\fl}{{\mathfrak{l}}}
\nc{\fm}{{\mathfrak{m}}}
\nc{\fn}{{\mathfrak{n}}}
\nc{\fu}{{\mathfrak{u}}}
\nc{\fp}{{\mathfrak{p}}}
\nc{\frr}{{\mathfrak{r}}}
\nc{\fs}{{\mathfrak{s}}}
\nc{\ft}{{\mathfrak{t}}}
\nc{\fw}{{\mathfrak{w}}}
\nc{\fz}{{\mathfrak{z}}}

\nc{\fA}{{\mathfrak{A}}}
\nc{\fB}{{\mathfrak{B}}}
\nc{\fD}{{\mathfrak{D}}}
\nc{\fE}{{\mathfrak{E}}}
\nc{\fF}{{\mathfrak{F}}}
\nc{\fG}{{\mathfrak{G}}}
\nc{\fI}{{\mathfrak{I}}}
\nc{\fJ}{{\mathfrak{J}}}
\nc{\fK}{{\mathfrak{K}}}
\nc{\fL}{{\mathfrak{L}}}
\nc{\fM}{{\mathfrak{M}}}
\nc{\fN}{{\mathfrak{N}}}
\nc{\frP}{{\mathfrak{P}}}
\nc{\fQ}{{\mathfrak Q}}
\nc{\fR}{{\mathfrak R}}
\nc{\fS}{{\mathfrak S}}
\nc{\fT}{{\mathfrak{T}}}
\nc{\fU}{{\mathfrak{U}}}
\nc{\fW}{{\mathfrak{W}}}
\nc{\fY}{{\mathfrak{Y}}}
\nc{\fZ}{{\mathfrak{Z}}}

\nc{\ba}{{\mathbf{a}}}
\nc{\bb}{{\mathbf{b}}}
\nc{\bc}{{\mathbf{c}}}
\nc{\bd}{{\mathbf{d}}}
\nc{\be}{{\mathbf{e}}}
\nc{\bi}{{\mathbf{i}}}
\nc{\bj}{{\mathbf{j}}}
\nc{\bn}{{\mathbf{n}}}
\nc{\bp}{{\mathbf{p}}}
\nc{\bq}{{\mathbf{q}}}
\nc{\bu}{{\mathbf{u}}}
\nc{\bv}{{\mathbf{v}}}
\nc{\bw}{{\mathbf{w}}}
\nc{\bx}{{\mathbf{x}}}
\nc{\by}{{\mathbf{y}}}
\nc{\bz}{{\mathbf{z}}}

\nc{\bA}{{\mathbf{A}}}
\nc{\bB}{{\mathbf{B}}}
\nc{\bC}{{\mathbf{C}}}
\nc{\bD}{{\mathbf{D}}}
\nc{\bE}{{\mathbf{E}}}
\nc{\bI}{{\mathbf{I}}}
\nc{\bK}{{\mathbf{K}}}
\nc{\bH}{{\mathbf{H}}}
\nc{\bM}{{\mathbf{M}}}
\nc{\bN}{{\mathbf{N}}}
\nc{\bO}{{\mathbf{O}}}
\nc{\bQ}{{\mathbf Q}}
\nc{\bS}{{\mathbf{S}}}
\nc{\bT}{{\mathbf{T}}}
\nc{\bV}{{\mathbf{V}}}
\nc{\bW}{{\mathbf{W}}}
\nc{\bX}{{\mathbf{X}}}
\nc{\bP}{{\mathbf{P}}}
\nc{\bY}{{\mathbf{Y}}}
\nc{\bZ}{{\mathbf{Z}}}

\nc{\sA}{{\mathsf{A}}}
\nc{\sB}{{\mathsf{B}}}
\nc{\sC}{{\mathsf{C}}}
\nc{\sD}{{\mathsf{D}}}
\nc{\sF}{{\mathsf{F}}}
\nc{\sK}{{\mathsf{K}}}
\nc{\sM}{{\mathsf{M}}}
\nc{\sO}{{\mathsf{O}}}
\nc{\sQ}{{\mathsf{Q}}}
\nc{\sP}{{\mathsf{P}}}
\nc{\sT}{{\mathsf{T}}}
\nc{\sV}{{\mathsf{V}}}
\nc{\sW}{{\mathsf{W}}}
\nc{\sX}{{\mathsf{X}}}
\nc{\sZ}{{\mathsf{Z}}}
\nc{\sU}{{\mathsf{U}}}
\nc{\sS}{{\mathsf{S}}}
\nc{\sH}{{\mathsf{H}}}

\nc{\sfb}{{\mathsf{b}}}
\nc{\sfc}{{\mathsf{c}}}
\nc{\sd}{{\mathsf{d}}}
\nc{\sg}{{\mathsf{g}}}
\nc{\sk}{{\mathsf{k}}}
\nc{\sfl}{{\mathsf{l}}}
\nc{\sfp}{{\mathsf{p}}}
\nc{\sr}{{\mathsf{r}}}
\nc{\st}{{\mathsf{t}}}
\nc{\sfu}{{\mathsf{u}}}
\nc{\sw}{{\mathsf{w}}}
\nc{\sz}{{\mathsf{z}}}
\nc{\sx}{{\mathsf{x}}}
\nc{\se}{{\mathsf{e}}}
\nc{\sfv}{{\mathsf{v}}}

\nc{\bLambda}{{\boldsymbol{\Lambda}}}
\nc{\vv}{{\boldsymbol{v}}}
\nc{\Fl}{{{\mathcal F}\ell}}
\nc{\Gr}{{\on{Gr}}}
\nc{\CHH}{{\CH\!\!\CH}}
\nc{\lambdavee}{{\lambda^{\!\scriptscriptstyle\vee}}}
\nc{\alphavee}{\alpha^{\!\scriptscriptstyle\vee}}
\nc{\rhovee}{{\rho^{\!\scriptscriptstyle\vee}}}
\newcommand\iso{\,\vphantom{j^{X^2}}\smash{\overset{\sim}{\vphantom{\rule{0pt}{0.20em}}\smash{\longrightarrow}}}\,}
\nc{\oQM}{\vphantom{j^{X^2}}\smash{\overset{\circ}{\vphantom{\vstretch{0.7}{A}}\smash{\QM}}}}
\nc{\oZ}{{}^\dagger\!\vphantom{j^{X^2}}\smash{\overset{\circ}{\vphantom{\vstretch{0.7}{A}}\smash{Z}}}}
\nc{\odZ}{{}^\dagger\!\vphantom{j^{X^2}}\smash{\overset{\circ}{\vphantom{\vstretch{0.7}{A}}\smash{\mathfrak Z}}}^{c',c}}
\nc{\bdZ}{{}^\dagger\!\vphantom{j^{X^2}}\smash{\overset{\bullet}{\vphantom{\vstretch{0.7}{A}}\smash{\mathfrak Z}}}^{c',c}}
\nc{\oS}{\vphantom{j^{X^2}}\smash{\overset{\circ}{\vphantom{\vstretch{0.7}{A}}\smash{S}}}}
\nc{\buM}{\vphantom{j^{X^2}}\smash{\overset{\bullet}{\vphantom{\vstretch{0.7}{A}}\smash{M}}}}
\nc{\dW}{{}^\dagger\ol\CW{}}
\nc{\hW}{{}^\dagger\hat\CW{}}
\nc{\wW}{{}^\dagger\widetilde\CW{}}
\nc{\dZ}{{}^\dagger\!\fZ^{c',c}}
\nc{\dZc}{{}^\dagger\!\fZ^{c,c}}
\nc{\tZ}{{}^\dagger\!\tilde{Z}{}}
\nc{\hZ}{{}^\dagger\!\hat{Z}{}}

\nc{\ssl}{\mathfrak{sl}} \nc{\gl}{\mathfrak{gl}}
\nc{\wt}{\widetilde} \nc{\Sym}{\mathrm{Sym}}
\nc{\sE}{{\mathsf{E}}} \nc{\bs}{{\mathbf{s}}}
\nc{\trig}{\mathrm{trig}} \nc{\rat}{\mathrm{rat}}
\nc{\sign}{\mathrm{sign}} \nc{\sL}{{\mathsf{L}}}
\nc{\fv}{{\mathfrak{v}}} \nc{\ad}{\mathrm{ad}}
\nc{\spsi}{{\mathsf{\psi}}} \nc{\sh}{{\mathsf{h}}}
\nc{\rtt}{\mathrm{rtt}} \nc{\qdet}{\mathrm{qdet}} \nc{\pt}{{\operatorname{pt}}}
\nc{\M}{\mathrm{M}} \nc{\Ker}{\mathrm{Ker}} \nc{\ssc}{\mathrm{sc}}
\nc{\loc}{\mathrm{loc}} \nc{\fra}{\mathrm{frac}}
\nc{\ddj}{\mathrm{DJ}} \nc{\End}{\mathrm{End}}
\nc{\GL}{\mathrm{GL}}

\nc{\hzeta}{\hat{\zeta}}

\newcommand{\sso}{\mathfrak{o}}
\newcommand{\rtU}{\mathcal{U}}
\newcommand{\Lf}{\mathcal{L}}

\begin{document}

\title[]{Shuffle algebras and their integral forms:\\ specialization map approach in types $C_n$ and $D_n$}

\author[]{Yue Hu}
\address[]{Y.H.: Beijing University of Posts and Telecommunications, School of Mathematical Sciences, Beijing, China}
\email[]{ldkhtys@gmail.com}

\author[]{Alexander Tsymbaliuk}
\address[]{A.T.: Purdue University, Department of Mathematics, West Lafayette, IN 47907, USA}
\email[]{sashikts@gmail.com}

\begin{abstract}
We construct a family of PBWD (Poincar{\'e}-Birkhoff-Witt-Drinfeld) bases for the positive subalgebras of quantum
loop algebras of type $C_n$ and $D_n$, as well as their Lusztig and RTT integral forms, in the new Drinfeld realization.
We also establish a shuffle algebra realization of these $\BQ(v)$-algebras (proved earlier in~\cite{NT21} by completely
different tools) and generalize the latter to the above $\BZ[v,v^{-1}]$-forms. The rational counterparts provide
shuffle algebra realizations of positive subalgebras of type $C_n$ and $D_n$ Yangians and their Drinfeld-Gavarini duals.
While this naturally generalizes our earlier treatment of the classical type $B_n$ in~\cite{HT24} and $A_n$ in~\cite{Tsy18},
the specialization maps in the present setup are more compelling.
\end{abstract}

\maketitle


\section{Introduction}


\subsection{Summary.}

The quantum loop algebras associated to simple $\fg$ admit two presentations: the original Drinfeld-Jimbo
realization $U^{\ddj}_v(L\fg)$ and the new Drinfeld realization $U_\vv(L\fg)$. The explicit isomorphism
can be upgraded to that of quantum affine algebras, cf.~\cite[Theorem 3]{Dri88}:
\begin{equation}
\label{eq:drinfeld-iso-aff}
  U^{\ddj}_v(\widehat{\fg}) \simeq U_v(\widehat{\fg}) .
\end{equation}
Many internal algebraic properties are developed in the Drinfeld-Jimbo realization using a
\begin{equation}
\label{triang DJ}
  \mathrm{triangular\ decomposition} \ \
  U^{\ddj}_v(\widehat{\fg})\simeq
  U^{\ddj,>}_v(\widehat{\fg})\otimes U^{\ddj,0}_v(\widehat{\fg})\otimes U^{\ddj,<}_v(\widehat{\fg}).
\end{equation}
For example, Beck~\cite{b} constructed the PBW-type bases of each of these subalgebras.

\medskip
\noindent
On the other hand, the new Drinfeld realization $U_v(\widehat{\fg})$ is key to the representation theory of these algebras.
In this realization, the infinite set of generators is nicely packed into the currents $e_i(z),f_i(z),\varphi^\pm_i(z)$ (which
bore fruits in CFT already in the classical case). It is thus natural to develop algebraic aspects of $U_v(\widehat{\fg})$
intrinsic to the loop realization. We note that
\begin{equation*}
  \mathrm{triangular\ decomposition} \ \
  U_v(\widehat{\fg})\simeq U^>_v(\widehat{\fg})\otimes U^0_v(\widehat{\fg})\otimes U^<_v(\widehat{\fg})
\end{equation*}
is not intertwined with that of~\eqref{triang DJ} through the aforementioned isomorphism~\eqref{eq:drinfeld-iso-aff}.

\medskip
\noindent
Besides the standard generators-and-relations presentation, quantum groups (or rather their positive subalgebras)
admit a more elegant combinatorial (dual) realization. For finite quantum groups, this manifests in the algebra
embedding (cf.~\cite{Gre97}):
\begin{equation}
\label{eq:shuffle-finite}
  U^>_v(\fg) \hookrightarrow \mathcal{F}=\bigoplus_{i_1,\ldots,i_k\in I}^{k\in \BN}\ \BQ(v)\cdot [i_1\ldots i_k],
\end{equation}
where $I$ is the set of simple roots of $\fg$ and $\mathcal{F}$ is endowed with the \emph{quantum shuffle} product.
As shown by Lalonde-Ram in~\cite{LR95}, there is a bijection between the set $\Delta^+$ of positive roots of $\fg$ and
the so-called \emph{standard Lyndon} words in~$I$, such that the order on $\Delta^+$ induced from the lexicographical
order of words is \emph{convex}. As a consequence, Lusztig's PBW basis of $U^>_v(\fg)$ can be constructed purely
combinatorially via iterated $v$-commutators, see details in~\cite{Lec04,NT21}.

\medskip
\noindent
Using similar ideas, Feigin-Odesskii introduced the elliptic shuffle algebras in~\cite{FO97,FO98}, whose
trigonometric counterpart (in the formal setup with $\BQ[[\hbar]]$ instead of $\BQ(q)$) was further studied by
Enriquez in~\cite{Enr00, Enr03}. Explicitly, this manifests in the algebra embedding
\begin{equation}
\label{eq:shuffle-loop}
  \Psi\colon U^>_v(L\fg) \hookrightarrow S,
\end{equation}
where $S$ consists of symmetric rational functions in $\{x_{i,r}\}_{i\in I}^{r\in \BZ}$ subject to so-called
\emph{pole} and \emph{wheel} conditions, endowed with the shuffle product. Thus, it is a \emph{functional version}
of~\eqref{eq:shuffle-finite}.

\medskip
\noindent
The key benefit of \eqref{eq:shuffle-loop} is that it provides tools to treat the elements of $U_v(L\fg)$ given by
high degree non-commutative polynomials in the original generators. Within the last decade, this approach has found
novel applications in the geometric representation theory, quantum integrable systems, and knot invariants. To make
this approach self-contained, it is important to have a description of the image $\mathrm{Im}(\Psi)$. In fact,
Enriquez conjectured~\cite[Remark~3.16]{Enr03}:
\begin{equation}
\label{eq:shuffle-iso}
  \Psi\colon U^>_v(L\fg) \, \iso\, S.
\end{equation}

\noindent
To prove~\eqref{eq:shuffle-iso}, one has to ``compare the size'' of $\QC$ and $S$. For types $A_{1}$ and $\hat{A}_{1}$,
this was accomplished in~\cite{Neg14} by utilizing \emph{specialization maps} analogous to those from~\cite{FS94, FHHSY09}.
A similar approach was used later in \cite{Neg13} to prove~\eqref{eq:shuffle-iso} for types $A_{n}$ and $\hat{A}_{n}$;
for two-parameter and super counterparts of type $A_n$ in~\cite{Tsy18}; for type $\mathfrak{D}(2,1;\theta)$
in \cite{FH21}; for types $G_{2}$ and $B_{n}$ in the authors' earlier work~\cite{HT24}. In the present note we generalize
this treatment to the remaining classical types $C_n$ and $D_n$. We should emphasize right away that unlike the aforementioned
cases, the specialization maps have to be properly \underline{normalized} in the present setup, since they now require a two-step
process in which certain vanishing factors arising due to wheel conditions must be first canceled before further specialization 
(as not to produce $0$). The main technical aspect of this note is to show that these normalized specialization maps still 
exhibit the same key properties as those crucially used in~\cite{Tsy18, HT24} for types $A_n$, $B_n$, $G_2$.

\medskip
\noindent
We conclude the summary by noting that while Enriquez's conjecture \eqref{eq:shuffle-iso} was recently proved for all
finite $\fg$ in \cite{NT21} using a very different approach, the present exposition has its own benefits as it allows
to upgrade our results to important integral $\BZ[v,v^{-1}]$-forms of $U^>_v(L\fg)$ as well as to the
Yangian counterpart, none of which was possible through the technique of~\cite{NT21}.


\subsection{Outline of the paper.}

The structure of the present paper is as follows:

\medskip
\noindent
$\bullet$
In Section~\ref{pre}, we recall the notion of quantum loop algebras $U^>_v(L\fg)$ in the new Drinfeld realization and
shuffle algebras $S$, introduce certain families of quantum root vectors (associated to specific convex orders on the
set of positive roots), and state the key results (PBWD bases and shuffle algebra isomorphism) for $U^>_v(L\fg)$ of
types $C_n,D_n$. We also introduce two integral forms  and state the PBWD bases for those. We conclude this section with
introducing the main tool, the \emph{specialization maps}, and summarize their key properties in Lemmas \ref{vanish},~\ref{span}.

\medskip
\noindent
$\bullet$
In Section \ref{type C}, we establish the key properties of specialization maps for type $C_n$, and use these to prove
Theorems~\ref{mainthm} and~\ref{pbwtheorem} for type $C_n$, see Theorem~\ref{shufflePBWD-C}. We upgrade both results to
Lusztig form $\integralcl$ and RTT form $\integralc$ in Theorems~\ref{lusthm-C} and~\ref{rttthm-C}, respectively.

\medskip
\noindent
$\bullet$
In Section \ref{type D}, we establish the key properties of specialization maps for type $D_n$, and use these to prove
Theorems~\ref{mainthm} and~\ref{pbwtheorem} for type $D_n$, see Theorem~\ref{shufflePBWD-D}. We upgrade both results to
Lusztig form $\integraldl$ and RTT form $\integrald$ in Theorems~\ref{lusthm-D} and~\ref{rttthm-D}, respectively.

\medskip
\noindent
$\bullet$
In Section \ref{yangian}, we generalize the results of Sections~\ref{type C}--\ref{type D} to the rational setup by providing
the shuffle realization and constructing PBWD bases for the positive subalgebras of the Yangians and their Drinfeld-Gavarini
duals in types $C_n$ and $D_n$, see Theorems \ref{yangshuffle} and \ref{dgdualshuffle}.

\medskip
\noindent
$\bullet$
In Appendix~\ref{sec:app}, we use the RTT realization of $U_v(L{\mathfrak{sp}}_{2n}),U_v(L{\mathfrak{o}}_{2n})$
from~\cite{JLM21,JLM20} to explain the natural origin and the name of the RTT integral forms $\integralc$, $\integrald$.


\subsection{Acknowledgements}
Both authors are grateful to B.~Feigin for numerous discussions.
We are also grateful to the anonymous referees for very useful suggestions.
A.T.\ is deeply indebted to A.~Negu\c{t} for enlightening discussions over the years.
A.T.\ is grateful to IHES (Bures-sur-Yvette, France) for the hospitality and
wonderful working conditions in the Spring 2023, when the preliminary version of this note was prepared.
The work of A.T.\ was partially supported by NSF Grant DMS-$2302661$.
The work of Y.H.\ was  supported by National Natural Science Foundation of China (No. $12401028$).


\medskip

\section{Preliminaries}\label{pre}


\subsection{Quantum loop algebras and shuffle algebras in types $C_n$ and $D_n$.}\label{ssec:qlas}

Let $\mathfrak{g}$ be a finite dimensional simple Lie algebra with simple positive roots $\{\alpha_{i}\}_{i\in I}$.
We denote the set of positive roots by $\Delta^{+}$. Each $\beta\in\Delta^{+}$ can be uniquely expressed as a sum
of simple roots: $\beta=\sum_{i\in I}\nu_{\beta,i}\alpha_{i}$ with $\nu_{\beta,i}\in\BN$. We shall refer to
$\nu_{\beta,i}$ as the \emph{coefficient of $\alpha_{i}$ in $\beta$}, and we shall use the following notation:
\begin{equation*}
  i\in\beta  \Longleftrightarrow \nu_{\beta,i}\neq 0.
\end{equation*}
The height of a root $\beta\in \Delta^+$ is:
\begin{equation}
\label{eq:root-height}
  |\beta|:=\sum_{i\in I} \nu_{\beta,i}.
\end{equation}

We fix a nondegenerate invariant bilinear form on the Cartan subalgebra $\fh$ of $\fg$. This gives rise to a
nondegenerate form on the dual $\fh^*$, and we set $d_{i}\coloneqq \frac{(\alpha_{i},\alpha_{i})}{2}$. The choice of
the form is such that $d_i=1$ for short roots $\alpha_i$. Let $A=(a_{ij})_{i,j\in I}$ be the Cartan matrix of $\fg$,
so that $d_{i}a_{ij}=(\alpha_{i},\alpha_{j})=d_{j}a_{ji}$. In this paper, we consider simple Lie algebras of types
$C_{n}$ and $D_{n}$. The corresponding Dynkin diagrams look as follows:
\begin{align}
  C_{n}\ (n\geq 3) &\qquad
  \begin{tikzpicture}[start chain]
  \dnode{1}
  \dnode{2}
  \dydots
  \dnode{n-1}
  \dnodenj{n}
  \path (chain-4) -- node[anchor=mid] {\(\Leftarrow\)} (chain-5);
  \end{tikzpicture}
\label{Dynkin-C} \\
  D_{n}\ (n\geq 4) &\qquad
  \begin{tikzpicture}
  \begin{scope}[start chain]
  \dnode{1}
  \dnode{2}
  \node[chj,draw=none] {\dots};
  \dnode{n-2}
  \dnode{n}
  \end{scope}
  \begin{scope}[start chain=br going above]
  \chainin(chain-4);
  \dnodebr{n-1}
  \end{scope}
  \end{tikzpicture}
\label{Dynkin-D}
\end{align}
For these types, we have
\begin{align*}
  C_{n}\text{-type}\ (n\geq 2) \colon & \quad d_{i}=1\ (1\leq i\leq n-1),\  d_{n}=2,\\
  D_{n}\text{-type}\ (n\geq 4) \colon & \quad d_{i}=1\ (1\leq i\leq n).
\end{align*}

Let $v$ be a formal variable. We define $v_{\alpha}=v^{(\alpha,\alpha)/2}$ for any $\alpha\in\Delta^{+}$, and denote
$v_{\alpha_{i}}=v^{d_i}$ simply by $v_{i}$ for any $i\in I$. Let $\fS_{m}$ denote the symmetric group of degree $m$.
Let $U_{v}^{>}(L\mathfrak{g})$ be the \textbf{``positive subalgebra'' of the quantum loop algebra} $U_{v}(L\mathfrak{g})$
associated to $\fg$ in the new Drinfeld realization. Explicitly, $U_{v}^{>}(L\mathfrak{g})$ is the $\BQ(v)$-algebra
generated by $\{e_{i,r}\}_{i\in I}^{r\in\mathbb{Z}}$ subject to the following defining relations:
\begin{equation*}
  (z-v_{i}^{a_{ij}}w)e_{i}(z)e_{j}(w)=(v_{i}^{a_{ij}}z-w)e_{j}(w)e_{i}(z) \qquad \forall\, i,j \in I,
\end{equation*}
\begin{equation*}
  \mathop{Sym}_{z_{1},\dots,z_{1-a_{ij}}}
  \sum_{k=0}^{1-a_{ij}} (-1)^{k}\left[\begin{matrix} 1-a_{ij}\\k\end{matrix}\right]_{v_{i}}
  e_{i}(z_{1})\cdots e_{i}(z_{k})e_{j}(w)e_{i}(z_{k+1})\cdots e_{i}(z_{1-a_{ij}})=0\qquad \forall\, i\neq j.
\end{equation*}
Here, we use the following notations:
\begin{equation*}
\begin{aligned}
  & [\ell]_{u}\coloneqq\frac{u^{\ell}-u^{-\ell}}{u-u^{-1}},\quad  [\ell]_{u}!\coloneqq\prod_{k=1}^{\ell}[k]_{u},\quad
    \left[\begin{matrix} \ell\\m\end{matrix}\right]_{u}\coloneqq\frac{[\ell]_{u}!}{[\ell-m]_{u}![m]_{u}!},\\
  & e_{i}(z)\coloneqq\sum_{r\in\mathbb{Z}} e_{i,r}z^{-r},\quad
    \mathop{Sym}_{z_{1},\dots,z_{m}}V(z_{1},\dots,z_{m})\coloneqq\sum_{\sigma\in\mathfrak{S}_{m}}
    V(z_{\sigma(1)},\dots,z_{\sigma(m)}).
\end{aligned}
\end{equation*}
We shall also need the following notation later:
\begin{equation}
\label{anglev}
  \langle m \rangle_{u}\coloneqq u^m-u^{-m}\qquad \forall\, m\in\BN.
\end{equation}

We define $\mathfrak{S}_{\underline{k}}\coloneqq\prod_{i\in I}\mathfrak{S}_{k_{i}}$ for any
$\underline{k}=(k_{1},\dots,k_{|I|})\in\mathbb{N}^{I}$. Associated to the Cartan matrix $A=(a_{ij})_{i,j\in I}$,
we also have the (trigonometric version of the) \textbf{Feigin-Odesskii shuffle algebra $S$}. To this end,
consider the following $\BN^{I}$-graded $\BQ(v)$-vector space
\[
  S=\bigoplus_{\underline{k}\in \mathbb{N}^{I}}S_{\underline{k}},
\]
where $S_{\underline{k}}$ consists of rational functions $F$ in the variables
$\{x_{i,r}\}_{i\in I}^{1\leq r\leq k_{i}}$ such that:
\begin{itemize}[leftmargin=0.7cm]

\item
$F$ is $\fS_{\unl{k}}$-symmetric, that is, symmetric in $\{x_{i,r}\}_{r=1}^{k_{i}}$ for each $i\in I$,

\medskip
\item
(\emph{pole conditions}) $F$ has the form
\begin{equation}
\label{polecondition}
  F=\frac{f(\{x_{i,r}\}_{i\in I}^{1\leq r\leq k_{i}})}
         {\prod_{i<j}^{a_{ij}\neq 0}\prod_{1\leq r\leq k_{i}}^{1\leq s\leq k_{j}}(x_{i,r}-x_{j,s})},
\end{equation}
where $f\in \BQ(v)[\{x_{i,r}^{\pm 1}\}_{i\in I}^{1\leq r\leq k_{i}}]^{\mathfrak{S}_{\underline{k}}}$ and an arbitrary
order $<$ is chosen on $I$ to make sense of $i<j$ (though the space $S_{\unl{k}}$ is clearly independent of this order),

\medskip
\item
(\emph{wheel conditions}) for any $F\in S_{\unl{k}}$, its numerator $f$ from~\eqref{polecondition} satisfies:
\begin{equation}
\label{wheelcon}
  f(\{x_{i,r}\}_{i\in I}^{1\leq r\leq k_{i}})=0 \quad \text{once} \quad
  x_{i,s_{1}}=v_{i}^{2}x_{i,s_{2}}=\cdots=v_{i}^{-2a_{ij}}x_{i,s_{1-a_{ij}}}=v_{i}^{-a_{ij}}x_{j,r}
\end{equation}
for any $i\neq j$ such that $a_{ij}\neq 0$, $1\leq s_{1},\dots,s_{1-a_{ij}}\leq k_{i}$, and $1\leq r\leq k_{j}$.

\end{itemize}

Let $(\zeta_{i,j}(z))_{i,j\in I}$ be the matrix of rational functions in $z$ given by
\begin{equation}
\label{eq:zeta}
  \zeta_{i,j}(z)=\frac{z-v^{-(\alpha_{i},\alpha_{j})}}{z-1}.
\end{equation}
For $\unl{k},\unl{\ell}\in\BN^{I}$, let
\begin{equation*}
  \unl{k}+\unl{\ell}=(k_{i}+\ell_{i})_{i\in I}\in\BN^{I}.
\end{equation*}
Let us introduce the bilinear {\em shuffle product} $\star$ on $S$ as follows: for $F\in S_{\underline{k}}$
and $G\in S_{\underline{\ell}}$, we set
\begin{equation}
\begin{aligned}
  & F\star G \big(\{x_{i,r}\}_{i\in I}^{1\leq r\leq k_{i}+\ell_i}\big)=\\
  & \frac{1}{\unl{k}!\cdot \unl{\ell}!}\cdot
    {\mathop{Sym}}_{\mathfrak{S}_{\underline{k}+\underline{\ell}}}
    \bigg(F\big(\{x_{i,r}\}_{i\in I}^{1\leq r\leq k_{i}}\big)\cdot
          G\big(\{x_{j,s}\}_{j\in I}^{k_{j}<s\leq k_{j}+\ell_{j}}\big)
          \prod_{i,j\in I}\prod_{r\leq k_{i}}^{s>k_{j}}\zeta_{i,j}\Big(\frac{x_{i,r}}{x_{j,s}}\Big)\bigg).
\label{shuffleproduct}
\end{aligned}
\end{equation}
Here, for $\unl{k}\in \BN^I$, we set $\unl{k}!=\prod_{i\in I} k_{i}!$, and define the \emph{symmetrization}
\begin{equation}
\label{eq:symmetrization}
  {\mathop{Sym}}_{\mathfrak{S}_{\unl{k}}}\big(F(\{x_{i,r}\}_{i\in I}^{1\leq r\leq k_{i}})\big)\, \coloneqq
  \sum_{(\sigma_{1},\dots,\sigma_{|I|})\in \mathfrak{S}_{\unl{k}}}F(\{x_{i,\sigma_{i}(r)}\}_{i\in I}^{1\leq r\leq k_{i}}).
\end{equation}
This endows $S$ with a structure of an associative unital algebra.

\begin{Not}
To simplify our formulas below, we shall often use $\zeta\Big(\frac{x_{i,r}}{x_{j,s}}\Big)$ instead of
$\zeta_{i,j}\Big(\frac{x_{i,r}}{x_{j,s}}\Big)$.
\end{Not}

This algebra $(S,\star)$ is related to $\QC$ via the following result of~\cite{NT21} (conjectured in~\cite{Enr03}):

\begin{Thm}\label{mainthm}
The assignment $e_{i,r}\mapsto x_{i,1}^{r}\in S_{\mathbf{1}_i} \ (i\in I, r\in\BZ)$, where
$\mathbf{1}_i=(0,\ldots,1,\ldots,0)$ with $1$ at the $i$-th coordinate, gives rise to a $\BQ(v)$-algebra isomorphism
\begin{equation}
\label{eq:Psi-homom}
  \Psi\colon U_{v}^{>}(L\mathfrak{g}) \, \iso \, S.
\end{equation}
\end{Thm}

The key objective of the present paper is to extend the method used in \cite{HT24} to the remaining classical types
$C_{n}$ and $D_{n}$. This provides a new proof of Theorem \ref{mainthm} in these types, different from~\cite{NT21},
but more importantly also yields tools to treat integral forms along the same lines.


\subsection{Root vectors and PBWD bases in types $C_n,D_n$}

Our construction of the specialization maps and PBWD bases is based on the specific choice of a convex order
on $\Delta^{+}$. The one that is best suited for our purposes is arising through the lexicographical order on
standard Lyndon words, see~\cite{LR95,Lec04}, as we recall next. The labeling of simple roots in the corresponding
Dynkin diagrams (\ref{Dynkin-C},~\ref{Dynkin-D}) provides a total order on the set $I$ of those, and hence the
lexicographical order on the set of words in the alphabet $I$. According to \cite[Proposition~3.2]{LR95}, there
is a natural bijection between the sets of positive roots $\Delta^{+}$ and so-called {\em standard Lyndon words}.
Thus, the lexicographical order on the latter gives rise to an order $<$ on $\Delta^{+}$, which is convex
by~\cite[Proposition 26]{Lec04} (cf.~\cite[Proposition 2.34]{NT21}). Henceforth, we fix this convex order
on $\Delta^{+}$ and use standard Lyndon words to parametrize positive roots.

Let us work this out explicitly for types $C_{n}$ and $D_{n}$ with the specific order on $I$ as in
(\ref{Dynkin-C},~\ref{Dynkin-D}). Applying~\cite[Proposition 25]{Lec04}, we find the set of all standard Lyndon words:
\begin{equation*}
\begin{aligned}
  C_{n}\text{-type}\ (n\geq 3) \colon & \quad  \Delta^{+}=\
    \big\{[i\dots j] \, | \, 1\leq i\leq j\leq n\big\}\\
  & \qquad \qquad \cup \big\{[i\dots (n-1)n(n-1)\dots j] \, | \, 1\leq i<j\leq n-1\big\} \\
  & \qquad \qquad \cup \big\{[i\dots (n-1) i\dots (n-1)n] \, | \, 1\leq i\leq n-1\big\}.\\
  D_{n}\text{-type}\ (n\geq 4) \colon & \quad  \Delta^{+}=\
    \big\{[i\dots j] \, | \, 1\leq i\leq j\leq n-1\big\}\cup \big\{[n]\big\}\\
  & \qquad \qquad \cup \big\{[i\dots(n-2)n] \, | \, 1\leq i\leq n-2\big\} \\
  & \qquad \qquad \cup \big\{[i \dots (n-2) n(n-1)\dots j] \, | \, 1\leq i<j\leq n-1\big\}.
\end{aligned}
\end{equation*}
For convenience, we shall use the following notations for positive roots in types $C_{n}$ and $D_{n}$:
\begin{itemize}[leftmargin=0.7cm]

\item
Type $C_{n}\colon$
\begin{equation}
\label{eq:abbr-prC}
\begin{aligned}
  & [i,j]\coloneqq [i\dots j] \qquad \mathrm{for} \quad 1\leq i\leq j\leq n,\\
  & [i,n,j]\coloneqq [i\dots (n-1)n(n-1)\dots j] \qquad \mathrm{for} \quad 1\leq i< j < n,\\
  & [i,n,i]\coloneqq [i\dots(n-1)i\dots(n-1)n] \qquad \mathrm{for} \quad 1\leq i < n.
\end{aligned}
\end{equation}

\item
Type $D_{n}\colon$
\begin{equation}
\label{eq:abbr-prD}
\begin{aligned}
  & [i,j]\coloneqq [i\dots j] \qquad \mathrm{for} \quad 1\leq i\leq j < n \ \text{or}\ i=j=n,\\
  & [i,n]\coloneqq [i\dots (n-2)n] \qquad \mathrm{for} \quad 1\leq i\leq n-2,\\
  & [i,n,j]\coloneqq [i \dots(n-2)n(n-1)\dots j] \qquad \mathrm{for} \quad 1\leq i < j < n.
\end{aligned}
\end{equation}

\end{itemize}
The aforementioned specific convex order on $\Delta^{+}$ in types $C_{n}$, $D_{n}$ looks as follows:
\begin{itemize}[leftmargin=0.7cm]

\item
Type $C_{n}\colon$
\begin{equation}
\label{lynorderC}
\begin{aligned}
  [1]<[1,2]< \cdots<&[1,n-1]<[1,n,1]<[1,n]<[1,n,n-1]<\cdots<[1,n,2] \\
  &<[2]<\cdots<[n-1,n,n-1]<[n].
\end{aligned}
\end{equation}

\item
Type $D_{n}\colon$
\begin{equation}
\label{lynorderD}
\begin{aligned}
  &[1]<[1,2]< \cdots<[1,n-1]<[1,n]<[1,n,n-1]<[1,n,n-2]<\cdots<[1,n,2]\\
  &<[2]<\cdots<[n-2,n-1]<[n-2,n]<[n-2,n,n-1]<[n-1]<[n].
\end{aligned}
\end{equation}

\end{itemize}
We define the \emph{quantum root vectors} $\{E_{\beta,s}\}_{\beta\in\Delta^{+}}^{s\in\BZ}$ of $\QC$ in type $C_{n},D_n$
via iterated $v$-commutators. Here, for $x,y\in\QC$ and $u\in\BQ(v)$, the \emph{$u$-commutator} $[x,y]_{u}$ is
\begin{equation*}
  [x,y]_{u}\coloneqq xy-u\cdot yx.
\end{equation*}
\begin{itemize}[leftmargin=0.7cm]

\item
Type $C_{n}\colon$

If $\beta=[i_{1},\dots,i_{\ell}]\neq [i,n,i]$, we choose a collection $\lambda_{1},\dots,\lambda_{\ell-1}\in v^{\BZ}$
and a decomposition $s=s_{1}+\cdots+s_{\ell}$ with $s_{1}, \dots,s_{\ell}\in\BZ$. Then, we define
\begin{equation}
\label{rootvector1}
  E_{\beta,s}\coloneqq
  [\cdots[[e_{i_{1},s_{1}},e_{i_{2},s_{2}}]_{\lambda_{1}},e_{i_{3},s_{3}}]_{\lambda_{2}},\cdots,
  e_{i_{\ell},s_{\ell}}]_{\lambda_{\ell-1}}.
\end{equation}
If $\beta=[i,n,i]$, we choose $\lambda\in v^{\BZ}$, a decomposition $s=s_{1}+s_{2}$ with $s_{1},s_{2}\in\BZ$,
and any quantum root vector $E_{[i,n-1],s_{1}}, E_{[i,n],s_{2}}$ defined by \eqref{rootvector1}, and then define
\begin{equation}
\label{rootvector2}
  E_{\beta,s}\coloneqq [E_{[i,n-1],s_{1}},E_{[i,n],s_{2}}]_{\lambda}.
\end{equation}

\item
Type $D_{n}\colon$

For any $\beta=[i_{1},\dots,i_{\ell}]\in\Delta^{+}$, we choose a collection $\lambda_{1},\dots,\lambda_{\ell-1}\in v^{\BZ}$
and a decomposition $s=s_{1}+\cdots+s_{\ell}$ with $s_{1}, \dots,s_{\ell}\in\BZ$. Then, we define
\begin{equation}
\label{rootvector3}
  E_{\beta,s}\coloneqq
  [\cdots[[e_{i_{1},s_{1}},e_{i_{2},s_{2}}]_{\lambda_{1}},e_{i_{3},s_{3}}]_{\lambda_{2}},\cdots,
   e_{i_{\ell},s_{\ell}}]_{\lambda_{\ell-1}}.
\end{equation}
\end{itemize}

In particular, we have the following specific choices $\{\tilde{E}^{\pm}_{\beta,s}\}_{\beta\in\Delta^{+}}^{s\in \BZ}$
which will be used to construct PBWD bases of the integral forms in Theorems~\ref{pbwtheorem-Lus} and~\ref{PBWDintegralrtt}:
\begin{itemize}[leftmargin=0.7cm]

\item
Type $C_{n}\colon$

For $\beta=[i,j]$ with $1\leq i\leq j<n$ and $s\in \BZ$, we choose any decomposition $s=s_{i}+\cdots+s_{j}$,
fix a sign $\pm$, and define
\begin{equation}
\label{rvc1}
  \tilde{E}^{\pm}_{[i,j],s}\coloneqq
  [\cdots[[e_{i,s_{i}},e_{i+1,s_{i+1}}]_{v^{\pm 1}},e_{i+2,s_{i+2}}]_{v^{\pm 1}},\cdots,e_{j,s_{j}}]_{v^{\pm 1}}.
\end{equation}
For $\beta=[i,n]$ with $1\leq i\leq n$ and $s\in \BZ$, we choose any decomposition $s=s_{i}+\cdots+s_{n}$,
fix a sign $\pm$, and define
\begin{equation}
\label{rvc2}
  \tilde{E}^{\pm}_{[i,n],s}\coloneqq
  [[\cdots[e_{i,s_{i}},e_{i+1,s_{i+1}}]_{v^{\pm 1}},\cdots,e_{n-1,s_{n-1}}]_{v^{\pm 1}},e_{n,s_{n}}]_{v^{\pm 2}}.
\end{equation}
For $\beta=[i,n,j]$ with $1\leq i<j<n$ and $s\in \BZ$, we choose any decomposition
$s=s_{i}+\cdots+s_{j-1}+2s_{j}+\cdots +2s_{n-1}+s_{n}$, fix a sign~$\pm$, and define
\begin{equation}
\label{rvc3}
\begin{aligned}
  \tilde{E}^{\pm}_{[i,n,j],s}\coloneqq
  [\cdots[[[\cdots[e_{i,s_{i}},e_{i+1,s_{i+1}}]_{v^{\pm 1}},&\cdots,e_{n-1,s_{n-1}}]_{v^{\pm 1}},\\
  &e_{n,s_{n}}]_{v^{\pm 2}},e_{n-1,s_{n-1}}]_{v^{\pm 1}},\cdots,e_{j,s_{j}}]_{v^{\pm 1}}.
\end{aligned}
\end{equation}
For $\beta=[i,n,i]$ with $1\leq i\leq n-1$ and $s\in \BZ$, we choose any decomposition $s=2s_{i}+\cdots+2s_{n-1}+s_{n}$,
fix a sign~$\pm$, and define
\begin{equation}
\label{rvc4}
\begin{aligned}
  \tilde{E}^{\pm}_{[i,n,i],s}\coloneqq
    [[\cdots&[e_{i,s_{i}},e_{i+1,s_{i+1}}]_{v^{\pm 1}},\cdots,e_{n-1,s_{n-1}}]_{v^{\pm 1}},\\
  & [[\cdots[e_{i,s_{i}},e_{i+1,s_{i+1}}]_{v^{\pm 1}},\cdots,e_{n-1,s_{n-1}}]_{v^{\pm 1}},e_{n,s_{n}}]_{v^{\pm 2}}] .
\end{aligned}
\end{equation}

\item
Type $D_{n}\colon$

For $\beta=[i,j]$ with $1\leq i\leq j<n$ (or $i=j=n$) and  $s\in \BZ$, we choose any decomposition $s=s_{i}+\cdots+s_{j}$,
fix a sign $\pm$, and define
\begin{equation}
\label{rvd1}
  \tilde{E}^{\pm}_{[i,j],s}\coloneqq
  [\cdots[[e_{i,s_{i}},e_{i+1,s_{i+1}}]_{v^{\pm 1}},e_{i+2,s_{i+2}}]_{v^{\pm 1}},\cdots,e_{j,s_{j}}]_{v^{\pm 1}}.
\end{equation}
For $\beta=[i,n]$ with $1\leq i\leq n-2$ and $s\in \BZ$, we choose any decomposition $s=s_{i}+\cdots+s_{n-2}+s_{n}$,
fix a sign $\pm$, and define
\begin{equation}
\label{rvd2}
  \tilde{E}^{\pm}_{[i,n],s}\coloneqq
  [[\cdots[e_{i,s_{i}},e_{i+1,s_{i+1}}]_{v^{\pm 1}},\cdots,e_{n-2,s_{n-2}}]_{v^{\pm 1}},e_{n,s_{n}}]_{v^{\pm 1}}.
\end{equation}
For $\beta=[i,n,n-1]$ with $1\leq i\leq n-2$ and $s\in \BZ$, we choose any decomposition $s=s_{i}+\cdots+s_{n-2}+s_{n-1}+s_{n}$,
fix a sign $\pm$, and define
\begin{equation}
\label{rvd3}
  \tilde{E}^{\pm}_{[i,n,n-1],s}\coloneqq
  [[[\cdots[e_{i,s_{i}},e_{i+1,s_{i+1}}]_{v^{\pm 1}},\cdots,e_{n-2,s_{n-2}}]_{v^{\pm 1}},e_{n,s_{n}}]_{v^{\pm 1}},e_{n-1,s_{n-1}}]_{v^{\pm 1}}.
\end{equation}
For $\beta=[i,n,j]$ with $1\leq i<j\leq n-2$ and $s\in \BZ$, we choose any decomposition
$s=s_{i}+\cdots+s_{j-1}+2s_{j}+\cdots +2s_{n-2}+s_{n-1}+s_{n}$, fix a sign~$\pm$, and define
\begin{equation}
\label{rvd4}
\begin{aligned}
  \tilde{E}^{\pm}_{[i,n,j],s}\coloneqq
  [\cdots[[[\cdots[e_{i,s_{i}},e_{i+1,s_{i+1}}]_{v^{\pm 1}},&\cdots,e_{n-2,s_{n-2}}]_{v^{\pm 1}},\\
  &e_{n,s_{n}}]_{v^{\pm 1}},e_{n-1,s_{n-1}}]_{v^{\pm 1}},\cdots,e_{j,s_{j}}]_{v^{\pm 1}}.
\end{aligned}
\end{equation}

\end{itemize}

Evoking the specific convex orders $<$ on $\Delta^+$ from~\eqref{lynorderC}--\eqref{lynorderD},
let us consider the following order $<$ on the set $\Delta^{+}\times \mathbb{Z}$:
\begin{equation}
\label{orderbetas}
  (\alpha,s) < (\beta,t) \quad  \text{iff} \quad \alpha<\beta \ \text{ or }\  \alpha=\beta, s < t.
\end{equation}
Let $H$ denote the set of all functions $h\colon \Delta^{+}\times\BZ\rightarrow \BN$ with finite support.
The monomials
\begin{equation}
\label{PBWDbases}
  E_{h}\ :=\prod_{(\beta,s)\in\Delta^{+}\times\mathbb{Z}}\limits^{\rightarrow}E_{\beta,s}^{h(\beta,s)}
  \qquad \forall\, h\in H
\end{equation}
will be called the \emph{ordered PBWD monomials} of $\QC$. Here, the arrow $\rightarrow$ over the product sign
refers to the total order \eqref{orderbetas}. Our second key result generalizes~\cite[Theorem 2.16]{Tsy18} and
\cite[Theorem 2.5]{HT24} from types $A_n,B_n,G_2$ to types $C_{n}$ and $D_{n}$:

\begin{Thm}\label{pbwtheorem}
The ordered PBWD monomials $\{E_{h}\}_{h\in H}$ of~\eqref{PBWDbases} form $\BQ(v)$-bases of~$\QC$
for $\fg$ of type $C_{n}$ and $D_{n}$.
\end{Thm}


\subsection{Two integral forms in types $C_n$ and $D_n$}
Following \cite{Tsy18,HT24}, we shall also use shuffle approach to study integral forms of $\QC$ in types $C_n$ and $D_n$.
Consider the \emph{divided powers}
\begin{equation*}
  \mathbf{E}_{i,r}^{(k)}\coloneqq \frac{e_{i,r}^{k}}{[k]_{v_{i}}!} \qquad \forall\, i\in I,\  r\in\mathbb{Z},\ k\in\BN.
\end{equation*}
Following~\cite[\S7.7]{Gro94}, we define:

\begin{Def}\label{def:lusintegral}
For $\fg$ of type $C_n$ and $D_n$, the \textbf{Lusztig integral form} $\integrall$ is  the $\BZ[v,v^{-1}]$-subalgebra
of $\QC$ generated by $\{\mathbf{E}_{i,r}^{(k)}\}_{i\in I, r\in\BZ}^{k\in\BN}$.
\end{Def}

To construct PBWD bases of $\integrall$, we define the following \emph{normalized divided powers} of the
quantum root vectors in types $C_n$ and $D_n$ (cf.\ \eqref{rvc1}--\eqref{rvc4} and \eqref{rvd1}--\eqref{rvd4}):
\begin{align}
  & \label{eq:nrv-C}
  C_n-\mathrm{type}\colon \qquad \qquad  \tilde{\mathbf{E}}_{\beta,s}^{\pm,(k)}\coloneqq
  \begin{cases}
    \frac{(\tilde{E}^{\pm}_{\beta,s})^{k}}{[2]^k_{v}\cdot [k]_{v_{\beta}}!} & \text{if}\ \beta=[i,n,i] \ \mathrm{with} \ 1\leq i<n \\
    \frac{(\tilde{E}^{\pm}_{\beta,s})^{k}}{[k]_{v_{\beta}}!} &\mathrm{other\  cases}
  \end{cases}, \\
  & \label{eq:nrv-D}
  D_n-\mathrm{type}\colon \qquad \qquad
  \tilde{\mathbf{E}}_{\beta,s}^{\pm,(k)}\coloneqq \frac{(\tilde{E}^{\pm}_{\beta,s})^{k}}{[k]_{v}!} \qquad \forall\, \beta\in\Delta^{+}.
\end{align}

Completely analogously to \cite[Propositions 3.8, 4.15]{HT24}, we have\footnote{This relies on~\cite[Theorem 4.2]{BKM14}
that identifies $\tilde{\mathbf{E}}_{\beta,0}^{\pm,(1)}$ with Lusztig's quantum root vectors $\hat{E}^\pm_\beta$ of $U_v(\fg)$.}:

\begin{Prop}\label{integralrv}
In types $C_n$ and $D_n$, for any $\beta\in\Delta^{+}, s\in\BZ, k\in\BN$, the normalized divided powers of quantum root vectors
$\{\tilde{\mathbf{E}}_{\beta,s}^{\pm,(k)}\}^{k\in\BN}_{\beta\in\Delta^{+},s\in\BZ}$ defined in \eqref{eq:nrv-C}--\eqref{eq:nrv-D}
belong to $\integrall$.
\end{Prop}

For $\epsilon\in\{\pm\}$, define the ordered monomials (cf.~\eqref{PBWDbases})
\begin{equation}
\label{eq:Lus-pbwd}
  \tilde{\mathbf{E}}^{\epsilon}_{h}\ =\prod_{(\beta,s)\in\Delta^{+}\times\mathbb{Z}}\limits^{\rightarrow}
  \tilde{\mathbf{E}}_{\beta,s}^{\epsilon,(h(\beta,s))}\qquad \forall\, h\in H.
\end{equation}
Our third key result upgrades Theorem \ref{pbwtheorem} to the Lusztig integral form $\integrall$:

\begin{Thm}\label{pbwtheorem-Lus}
For $\epsilon\in\{\pm\}$, the ordered monomials $\{\tilde{\mathbf{E}}^{\epsilon}_{h}\}_{h\in H}$ of~\eqref{eq:Lus-pbwd}
form a $\BZ[v,v^{-1}]$-basis of $\integrall$ for $\fg$ of type $C_n$ and $D_n$.
\end{Thm}

Let us now introduce another integral form of $\QC$. For $\epsilon\in\{\pm \}$, define the following
\emph{normalized quantum root vectors} in types $C_n$ and $D_n$ (cf.\ \eqref{rvc1}--\eqref{rvc4} and
\eqref{rvd1}--\eqref{rvd4}):
\begin{align}
  & \label{eq:rtt-vectors-C}
  C_n-\mathrm{type}\colon \qquad \qquad
  \tilde{\mathcal{E}}^{\epsilon}_{\beta,s}\coloneqq
  \begin{cases}
    \langle 2\rangle_{v}\cdot \tilde{E}^{\epsilon}_{\beta,s} & \text{if}\ \beta=[n]\\
    \langle 1\rangle_{v}\cdot \tilde{E}^{\epsilon}_{\beta,s} & \text{other cases}
  \end{cases},\\
  & \label{eq:rtt-vectors-D}
  D_n-\mathrm{type}\colon \qquad \qquad
  \tilde{\mathcal{E}}^{\epsilon}_{\beta,s}\coloneqq \langle 1\rangle_{v}\cdot \tilde{E}^{\epsilon}_{\beta,s}
  \qquad \forall\, (\beta,s)\in\Delta^{+}\times \BZ.
\end{align}
The origin of these normalization factors (as well as the terminology ``RTT'' below) is explained in
Appendix~\ref{sec:app}.\footnote{We also note that $\tilde{\mathcal{E}}^{\pm}_{\beta,0}=(v_\beta-v_\beta^{-1})\hat{E}^\pm_\beta$,
with $\hat{E}^\pm_\beta$ being the Lusztig's quantum root vector of $U_v(\fg)$.}
Similarly to \eqref{PBWDbases}, we  consider the ordered monomials
\begin{equation}
\label{pbwdbasesrtt}
  \tilde{\mathcal{E}}^{\epsilon}_{h} \ =
  \prod_{(\beta,s)\in\Delta^{+}\times\mathbb{Z}}\limits^{\rightarrow}
  (\tilde{\mathcal{E}}^{\epsilon}_{\beta,s})^{h(\beta,s)}\qquad \forall\, h\in H.
\end{equation}

\begin{Def}\label{def:rttintegral}
For $\fg$ of type $C_n$ and $D_n$, and fixed $\epsilon\in\{\pm\}$, the \textbf{RTT integral form} $\integral$
is the $\BZ[v,v^{-1}]$-subalgebra of $\QC$ generated by
  $\{\tilde{\mathcal{E}}^{\epsilon}_{\beta,s}\}_{\beta\in\Delta^{+}}^{s\in\BZ}$.
\end{Def}

We note that the above definition depends on the choices of quantum root vectors in \eqref{rvc1}--\eqref{rvc4} or
\eqref{rvd1}--\eqref{rvd4}, as well as of $\epsilon\in\{\pm\}$. Our fourth key result shows that Definition
\ref{def:rttintegral} is well-defined and upgrades Theorem~\ref{pbwtheorem} to the RTT integral form $\integral$:

\begin{Thm}\label{PBWDintegralrtt}
Let $\fg$ be of type $C_n$ or $D_n$.

\noindent
(a) $\integral$ is independent of $\epsilon\in\{\pm\}$ and the choice of
$\{\tilde{\mathcal{E}}^{\epsilon}_{\beta,s}\}_{\beta\in\Delta^{+}}^{s\in \BZ}$
from \eqref{eq:rtt-vectors-C} or \eqref{eq:rtt-vectors-D}.

\noindent
(b) For $\epsilon\in\{\pm\}$, the ordered monomials
$\{ \tilde{\mathcal{E}}^{\epsilon}_{h}\}_{h\in H}$ of~\eqref{pbwdbasesrtt} form a $\BZ[v,v^{-1}]$-basis of $\integral$.
\end{Thm}


\subsection{Specialization maps in types $C_{n}$ and $D_n$}

Following~\cite{HT24}, we shall use the technique of specialization maps to prove all theorems above.
We shall now briefly introduce those and state their key properties in the end of this subsection.

Identifying each simple root $\alpha_{i}\ (i\in I)$ with a basis element $\mathbf{1}_{i}\in\mathbb{N}^{I}$ (having
the $i$-th coordinate equal to $1$ and the rest equal to $0$), we can view $\mathbb{N}^{I}$ as the positive cone of
the root lattice of $\fg$. For any $\underline{k}\in\mathbb{N}^{I}$, let $\text{KP}(\underline{k})$ be the set of
{\em Kostant partitions}, i.e.\ unordered vector partitions of $\underline{k}$ into a sum of positive roots.
Explicitly, a Kostant partition of $\underline{k}$ is the same as a tuple
$\unl{d}=\{d_{\beta}\}_{\beta\in\Delta^{+}}\in \BN^{\Delta^{+}}$ satisfying
  $\sum_{i\in I} k_i\alpha_i \, = \sum_{\beta\in\Delta^{+}}d_{\beta}\beta$.
Our specific convex order~\eqref{lynorderC}--\eqref{lynorderD} on $\Delta^{+}$ induces a total order on
$\text{KP}(\underline{k})$:
\begin{equation}
\label{eq:KP-order}
  \{d'_\beta\}_{\beta\in \Delta^+}<\{d_\beta\}_{\beta\in \Delta^+}\Longleftrightarrow \exists\  \gamma\in \Delta^+\
  \mathrm{s.t.}\ d'_\gamma<d_\gamma\ \mathrm{and}\
  d'_\beta=d_\beta\ \mathrm{for\ all}\ \beta<\gamma.
\end{equation}

Let us now define the specialization maps in types $C_{n}$ and $D_n$. For any $F\in S_{\unl{k}}$ and
$\underline{d}\in\text{KP}(\underline{k})$, we split the variables $\{x_{i,\ell}\}_{i\in I}^{1\leq \ell\leq k_{i}}$
into the disjoint union of $\sum_{\beta\in\Delta^{+}}d_{\beta}$ groups
\begin{equation}
\label{splitvariable}
  \bigsqcup^{1\leq s\leq d_{\beta}}_{\beta\in\Delta^{+}}
  \Big\{x^{(\beta,s)}_{i,t} \, \Big| \, i\in I, 1\leq t\leq \nu_{\beta,i}\Big\} \,,
\end{equation}
where the integer $\nu_{\beta,i}$ is the coefficient of $\alpha_i$ in $\beta$ as defined in the beginning of
Subsection~\ref{ssec:qlas}. For $F\in S_{\unl{k}}$, let $f$ denote its numerator from~\eqref{polecondition}.
Then, the specialization map $\phi_{\underline{d}}(F)$ is defined  by successive specializations $\phi_{\beta,s}$
of the variables  \eqref{splitvariable} in $f$ as follows:
\begin{itemize}[leftmargin=0.7cm]

\item
$C_n$-type.

For $\beta\neq [i,n,i]$, we define $\phi_{\beta,s}(F)$ by specializing the variables
$\{x^{(\beta,s)}_{i,t}\}_{1\leq t\leq \nu_{\beta,i}}^{i\in I}$ of $f$ as:
\begin{equation}
\label{spe-C-1}
  x^{(\beta,s)}_{\ell\neq n,1}\mapsto v^{1-\ell}w_{\beta,s},\quad
  x^{(\beta,s)}_{\ell\neq n,2}\mapsto v^{-2n+\ell-1}w_{\beta,s}, \quad x^{(\beta,s)}_{n,1}\mapsto v^{-n}w_{\beta,s}.
\end{equation}
For $\beta= [i,n,i]$, the specialization $\phi_{\beta,s}$ is more complicated and is constructed in two steps.
First, we define $\phi^{(1)}_{\beta,s}(F)$ by specializing the variables
$\{x^{(\beta,s)}_{i,t}\}_{1\leq t\leq \nu_{\beta,i}}^{i\in I}$ of $f$ as:
\begin{align}
\label{spe-C-2}
  x^{(\beta,s)}_{\ell\neq n,1}\mapsto v^{1-\ell}w_{\beta,s},\quad x^{(\beta,s)}_{\ell\neq n,2}\mapsto v^{1-\ell}w'_{\beta,s},
  \quad x^{(\beta,s)}_{n,1}\mapsto v^{-n}w'_{\beta,s}.
\end{align}
According to wheel conditions \eqref{wheelcon}, $\phi^{(1)}_{\beta,s}(F)$ is divisible by
\begin{equation}
\label{Bfactor-C}
  B_{\beta}=\big\{(w_{\beta,s}-v^{-2}w'_{\beta,s})(w_{\beta,s}-v^{2}w'_{\beta,s})\big\}^{n-i-1}.
\end{equation}
Then, the second step of the specialization, denoted $\phi^{(2)}_{\beta,s}$, is defined by first dividing
$\phi^{(1)}_{\beta,s}(F)$ by $B_{\beta}$  and then specializing the variable $w'_{\beta,s}$ in
$\frac{\phi^{(1)}_{\beta,s}(F)}{B_{\beta}}$ to $v^{2}w_{\beta,s}$. In this way, we get the overall
specialization $\phi_{\beta,s}(F)$:
\begin{equation}
\label{spe-C-3}
  \phi_{\beta,s}(F)\coloneqq \phi^{(2)}_{\beta,s}\left(\phi^{(1)}_{\beta,s}(F)\right) =
  \eval{\frac{\phi^{(1)}_{\beta,s}(F)}{B_{\beta}}}_{w'_{\beta,s}\mapsto v^{2}w_{\beta,s}}.
\end{equation}

\item
$D_n$-type.

For $\beta \neq [i,n,j]$ with $i<j\leq n-2$, we define $\phi_{\beta,s}(F)$ by specializing the variables
$\{x^{(\beta,s)}_{i,t}\}_{1\leq t\leq \nu_{\beta,i}}^{i\in I}$ of $f$ as:
\begin{equation}
\label{spe-D-1}
  x^{(\beta,s)}_{\ell\neq n,1}\mapsto v^{1-\ell}w_{\beta,s},\quad x^{(\beta,s)}_{n,1}\mapsto v^{2-n}w_{\beta,s}.
\end{equation}
For $\beta= [i,n,j]$ with $1\leq i<j\leq n-2$, the specialization $\phi_{\beta,s}$ is again defined in two steps.
First, we define $\phi^{(1)}_{\beta,s}(F)$ by  specializing the variables
$\{x^{(\beta,s)}_{i,t}\}_{1\leq t\leq \nu_{\beta,i}}^{i\in I}$ of $f$ as:
\begin{align}
\label{spe-D-2}
  x^{(\beta,s)}_{\ell\neq n,1}\mapsto v^{1-\ell}w_{\beta,s} ,\quad x^{(\beta,s)}_{n,1}\mapsto v^{2-n}w_{\beta,s},
  \quad x^{(\beta,s)}_{\ell\neq n-1\& n,2}\mapsto v^{\ell+3-2n}w'_{\beta,s}.
\end{align}
According to wheel conditions \eqref{wheelcon}, $\phi^{(1)}_{\beta,s}(F)$ is divisible by
\begin{equation}
\label{Bfactor-D}
  B_{\beta}=\prod^{n-2}_{\ell=j}(w_{\beta,s}-v^{2\ell+4-2n}w'_{\beta,s})(w_{\beta,s}-v^{2\ell-2n}w'_{\beta,s}).
\end{equation}
Then, the second step of the specialization, denoted $\phi^{(2)}_{\beta,s}$, is defined by first dividing $\phi^{(1)}_{\beta,s}(F)$
by $B_{\beta}$  and then specializing the variable $w'_{\beta,s}$ in $\frac{\phi^{(1)}_{\beta,s}(F)}{B_{\beta}}$ to $w_{\beta,s}$.
In this way, we get the overall specialization $\phi_{\beta,s}(F)$:
\begin{equation}
\label{spe-D-3}
  \phi_{\beta,s}(F)\coloneqq \phi^{(2)}_{\beta,s}\left(\phi^{(1)}_{\beta,s}(F)\right) =
  \eval{\frac{\phi^{(1)}_{\beta,s}(F)}{B_{\beta}}}_{w'_{\beta,s}\mapsto w_{\beta,s}}.
\end{equation}

\end{itemize}
Finally, the specialization map  $\phi_{\underline{d}}(F)$ is defined by applying those separate maps $\phi_{\beta,s}$
in each group $\{x^{(\beta,s)}_{i,t}\}_{1\leq t\leq \nu_{\beta,i}}^{i\in I}$ of variables (the result is independent
of splitting as $F$ is symmetric):
\begin{equation*}
  \phi_{\unl{d}}(F)\coloneqq \prod^{1\leq s\leq d_{\beta}}_{\beta\in\Delta^{+}} \phi_{\beta,s}(F).
\end{equation*}
We note that $\phi_{\unl{d}}(F)$ is symmetric in $\{w_{\beta,s}\}_{s=1}^{d_{\beta}}$ for any $\beta\in\Delta^{+}$.
This gives rise to the
\begin{equation*}
  \mathrm{\textbf{specialization\ map}} \quad
  \phi_{\underline{d}}\colon S_{\underline{k}}\longrightarrow
  \BQ(v)[\{w_{\beta,s}^{\pm 1}\}_{\beta\in\Delta^{+}}^{1\leq s\leq d_{\beta}}]^{\mathfrak{S}_{\unl{d}}}.
\end{equation*}
We shall further extend it to the specialization map $\phi_{\unl{d}}$ on the entire shuffle algebra $S$:
\begin{equation*}
  \phi_{\underline{d}}\colon S\longrightarrow
  \BQ(v)[\{w_{\beta,s}^{\pm 1}\}_{\beta\in\Delta^{+}}^{1\leq s\leq d_{\beta}}]^{\mathfrak{S}_{\unl{d}}}
\end{equation*}
by declaring $\phi_{\unl{d}}(F')=0$ for any $F'\in S_{\unl{\ell}}$ with $\unl{\ell}\neq \unl{k}$.

Let us state the key properties of specialization maps $\phi_{\unl{d}}$ defined above: their proofs constitute
the key technical part of this note and will imply our main results similarly to~\cite{HT24}
(see the paragraph following Lemma 2.8 in~\cite{HT24}). For any $h\in H$,
we define its \emph{degree} $\text{deg}(h)\in \BN^{\Delta^{+}}$ as the Kostant partition
$\unl{d}=\{d_{\beta}\}_{\beta\in\Delta^{+}}$ with $d_{\beta}=\sum_{s\in \mathbb{Z}} h(\beta,s)\in \BN$ for all
$\beta\in \Delta^{+}$, and the \emph{grading} $\text{gr}(h)\in \BN^I$ so that
$\text{deg}(h)\in\text{KP}(\text{gr}(h))$. For any $\unl{k}\in\BN^{I}$ and $\unl{d}\in\text{KP}(\unl{k})$,
we define the following subsets of $H$:
\begin{equation}
\label{hunlkunld}
  H_{\underline{k}}\coloneqq \big\{h\in H \, \big| \,  \text{gr}(h)=\underline{k}\big\}, \qquad
  H_{\underline{k},\underline{d}}\coloneqq \big\{h\in H \, \big| \, \text{deg}(h)=\underline{d}\big\}.
\end{equation}
Then we have the following ``dominance property'' of $\phi_{\unl{d}}$:

\begin{Lem}\label{vanish}
For any $h\in H_{\unl{k},\unl{d}}$ and $\unl{d}'<\unl{d}$, cf.~\eqref{eq:KP-order}, we have $\phi_{\unl{d}'}(\Psi(E_{h}))=0$.
\end{Lem}

Let $S'_{\unl{k}}$ be the $\BQ(v)$-subspace of $S_{\unl{k}}$ spanned by $\{\Psi(E_{h})\}_{h\in H_{\unl{k}}}$.
Then, we have:

\begin{Lem}\label{span}
For any $F\in S_{\unl{k}}$ and $\unl{d}\in \text{\rm KP}(\unl{k})$, if $\phi_{\unl{d}'}(F)=0$ for all
$\unl{d}'\in \text{\rm KP}(\unl{k})$ such that $\unl{d}'<\unl{d}$, then there exists $F_{\unl{d}}\in S'_{\unl{k}}$
such that $\phi_{\unl{d}}(F)=\phi_{\unl{d}}(F_{\unl{d}})$ and $\phi_{\unl{d}'}(F_{\unl{d}})=0$ for all $\unl{d}'<\unl{d}$.
\end{Lem}


\medskip

\section{Shuffle algebra and its integral forms in type $C_n$}\label{type C}

In this section, we establish the key properties of the specialization maps for the shuffle algebras of type $C_{n}$.
This implies the shuffle algebra realization and PBWD-type theorems for $\qlc$ and its integral forms.


\subsection{$\qlc$ and its shuffle algebra realization}

In type $C_{n}$, for any $F\in S_{\unl{k}}$ with $\unl{k}\in\mathbb{N}^{n}$, the wheel conditions are:
\begin{equation}
\label{eq:Wheel-Cn}
\begin{aligned}
  F(\{x_{i,r}\}_{1\leq i\leq n}^{1\leq r\leq k_{i}})=0 \quad \text{once} \quad
  & x_{i,1}=v^{2}x_{i,2}=vx_{i+1,1} \quad \text{for some} \quad 1\leq i\leq n-2,\\
  \text{or} \quad
  & x_{i,1}=v^{2}x_{i,2}=vx_{i-1,1} \quad \text{for some} \quad 2\leq i\leq n-1,\\
   \text{or} \quad
   &x_{n,1}=v^{4}x_{n,2}=v^{2}x_{n-1,1},\\
  \text{or} \quad
  & x_{n-1,1}=v^{2}x_{n-1,2}=v^{4}x_{n-1,3}=v^{2}x_{n,1}.
\end{aligned}
\end{equation}
We also recall the notations \eqref{eq:abbr-prC} for positive roots in type $C_n$.
Henceforth, we shall use the notation $\doteq$ as in~\cite[(2.44)]{HT24}:
\begin{equation}
\label{eq:trig-const}
  A\doteq B \quad \text{if} \quad  A=c\cdot B \quad \text{for some} \ c\in \BQ^{\times}\cdot v^{\BZ}.
\end{equation}
First, let us compute the images of the quantum root vectors $\{\tilde{E}^{\pm}_{\beta,s}\}_{\beta\in\Delta^{+}}^{s\in\BZ}$
of \eqref{rvc1}--\eqref{rvc4}. We shall use $\mathrm{denom}_\beta$ to denote the denominator in~\eqref{polecondition} for
any $F\in S_\beta$, e.g.\ for $F=\Psi(\tilde{E}^{\pm}_{\beta,s})$.

\begin{Lem}\label{lem:psirv-C}
Consider the quantum root vectors
$\{\tilde{E}^{\pm}_{\beta,s}\}_{\beta\in\Delta^{+}}^{s\in\BZ}$ of \eqref{rvc1}--\eqref{rvc4}. Their images under
$\Psi$ of~\eqref{eq:Psi-homom} in the shuffle algebra $S$ of type $C_{n}$ are as follows, cf.~\eqref{anglev}:
\begin{itemize}[leftmargin=0.7cm]

\item
If $\beta=[i,j]$ with $1\leq i\leq j<n$ or $i=j=n$, then for any $s=s_{i}+\cdots+s_{j}$ used in \eqref{rvc1}:
\begin{align*}
  \Psi(\tilde{E}^{+}_{[i,j],s})\doteq
    \frac{\langle 1\rangle_{v}^{j-i}}{\mathrm{denom}_{[i,j]}} \cdot x_{i,1}^{s_{i}+1}\cdots x_{j-1,1}^{s_{j-1}+1}x_{j,1}^{s_{j}},\
  \Psi(\tilde{E}^{-}_{[i,j],s})\doteq
    \frac{\langle 1\rangle_{v}^{j-i}}{\mathrm{denom}_{[i,j]}} \cdot x_{i,1}^{s_{i}}x_{i+1,1}^{s_{i+1}+1}\cdots x_{j,1}^{s_{j}+1}.
\end{align*}

\item
If $\beta=[i,n]$ with $1\leq i<n$, then for any decomposition $s=s_{i}+\cdots+s_{n}$ used in \eqref{rvc2}:
\begin{align*}
  \Psi(\tilde{E}^{+}_{[i,n],s})\doteq
   \frac{\langle 1\rangle_{v}^{n-i-1}{\langle 2\rangle_{v}}}{\mathrm{denom}_{[i,n]}} \cdot x_{i,1}^{s_{i}+1}\cdots x_{n-1,1}^{s_{n-1}+1}x_{n,1}^{s_{n}},  \
  \Psi(\tilde{E}^{-}_{[i,n],s})\doteq
    \frac{\langle 1\rangle_{v}^{n-i-1}{\langle 2\rangle_{v}}}{\mathrm{denom}_{[i,n]}} \cdot x_{i,1}^{s_{i}}x_{i+1,1}^{s_{i+1}+1}\cdots x_{n,1}^{s_{n}+1}.
\end{align*}

\item
If $\beta=[i,n,j]$ with $1\leq i<j\leq n-1$, then for any
$s=s_{i}+\cdots +s_{j-1}+2s_{j}+\cdots +2s_{n-1}+s_{n}$ used in \eqref{rvc3}, we have:
\begin{equation*}
\begin{aligned}
  & \Psi(\tilde{E}^{+}_{[i,n,j],s})\doteq \frac{\langle 1\rangle_{v}^{2n-i-j-1}\langle 2\rangle_{v}}{\mathrm{denom}_{[i,n,j]}}
    \cdot g_{1}\cdot \left[(1+v^2)x_{j,1}x_{j,2}-vx_{j-1,1}(x_{j,1}+x_{j,2})\right]\\
  & \qquad \qquad \qquad \times \prod_{\ell=j}^{n-2} Q(x_{\ell,1},x_{\ell,2},x_{\ell+1,1},x_{\ell+1,2}),
\end{aligned}
\end{equation*}
\begin{equation*}
\begin{aligned}
  & \Psi(\tilde{E}^{-}_{[i,n,j],s})\doteq
    \frac{\langle 1\rangle_{v}^{2n-i-j-1}\langle 2\rangle_{v}}{\mathrm{denom}_{[i,n,j]}} \cdot g_{2}\cdot \left[(1+v^2)x_{j-1,1}-v(x_{j,1}+x_{j,2})\right]\\
  & \qquad \qquad \qquad \times \prod_{\ell=j}^{n-2}Q(x_{\ell,1},x_{\ell,2},x_{\ell+1,1},x_{\ell+1,2}),
\end{aligned}
\end{equation*}
where
\begin{equation}
\label{eq:Qform-C}
  Q(x_{1},x_{2},y_{1},y_{2})=(1+v^2)(x_{1}x_{2}+y_{1}y_{2})-v(x_{1}+x_{2})(y_{1}+y_{2})
\end{equation}
and
\begin{align*}
  g_{1}=\prod^{j-1}_{\ell=i}x_{\ell,1}^{s_{\ell}+1}(x_{j,1}x_{j,2})^{s_{j}}
    \prod^{n-1}_{\ell=j+1}(x_{\ell,1}x_{\ell,2})^{s_{\ell}+1}x_{n,1}^{s_{n}+1},\
  g_{2}=x_{i,1}^{s_{i}}\prod^{j-1}_{\ell=i+1}x_{\ell,1}^{s_{\ell}+1}
    \prod^{n-1}_{\ell=j}(x_{\ell,1}x_{\ell,2})^{s_{\ell}+1}x_{n,1}^{s_{n}+1}.
\end{align*}

\item
If $\beta=[i,n,i]$ with $1\leq i\leq n-1$, then for any decomposition $s=2s_{i}+\cdots+2s_{n-1}+s_{n}$ used in \eqref{rvc4},
we have (cf.~\eqref{eq:Qform-C}):
\begin{equation}\label{eq:Imtilde Eb-3}
  \Psi(\tilde{E}^{+}_{[i,n,i],s})\doteq
  \frac{\langle 1\rangle_{v}^{2n-2i-2}\langle 2\rangle_{v}^{2}}{\mathrm{denom}_{[i,n,i]}} \cdot
  \prod^{n-1}_{\ell=i}(x_{\ell,1}x_{\ell,2})^{s_{\ell}+1} x^{s_{n}}_{n,1}
  \prod_{\ell=i}^{n-2}Q(x_{\ell,1},x_{\ell,2},x_{\ell+1,1},x_{\ell+1,2}),
\end{equation}
\begin{equation*}
   \Psi(\tilde{E}^{-}_{[i,n,i],s})\doteq
   \frac{\langle 1\rangle_{v}^{2n-2i-2}\langle 2\rangle_{v}^{2}}{\mathrm{denom}_{[i,n,i]}} \cdot
   (x_{i,1}x_{i,2})^{s_{i}} \prod^{n-1}_{\ell=i+1}(x_{\ell,1}x_{\ell,2})^{s_{\ell}+1} x^{s_{n}+2}_{n,1}
   \prod_{\ell=i}^{n-2}Q(x_{\ell,1},x_{\ell,2},x_{\ell+1,1},x_{\ell+1,2}).
\end{equation*}

\end{itemize}
\end{Lem}

\begin{proof}
We shall present only the derivation of the formula for $\Psi(\tilde{E}^{+}_{[i,n,i],s})$, while the other formulas
are obtained in a similar (but simpler) way. The proof proceeds by a descending induction in $i$. The base case $i=n-1$
is derived as follows:
\begin{multline*}
  \Psi(\tilde{E}^{+}_{[n-1,n,n-1],s})
  = \Psi(e_{n-1,s_{n-1}})\star \Psi(\tilde{E}^{+}_{[n-1,n],s_{n-1}+s_{n}}) -
    \Psi(\tilde{E}^{+}_{[n-1,n],s_{n-1}+s_{n}})\star\Psi(e_{n-1,s_{n-1}}) \\
  = \frac{\langle 2\rangle_{v}(x_{n-1,1}x_{n-1,2})^{s_{n-1}}x^{s_{n}}_{n,1}}{\mathrm{denom}_{[n-1,n,n-1]}} \, \cdot
    \mathop{Sym}_{x_{n-1,1},x_{n-1,2}}\left(\frac{x_{n-1,2}(x_{n-1,1}-v^{-2}x_{n-1,2})(x_{n-1,1}-v^2x_{n,1})}{x_{n-1,1}-x_{n-1,2}} \, + \right.\\
    \left.\frac{x_{n-1,1}(x_{n-1,1}-v^{-2}x_{n-1,2})(x_{n,1}-v^2x_{n-1,2})}{x_{n-1,1}-x_{n-1,2}}\right) \doteq
    \frac{\langle 2\rangle^{2}_{v}}{\mathrm{denom}_{[n-1,n,n-1]}}\cdot (x_{n-1,1}x_{n-1,2})^{s_{n-1}+1}x^{s_{n}}_{n,1}.
\end{multline*}
As per the step of induction, let us assume that \eqref{eq:Imtilde Eb-3} holds for any $j+1\leq i\leq n-1$. Due to
\begin{equation*}
  \mathop{Sym}_{x_{1},x_{2}}\left(\frac{(x_1-v^{-2} x_{2})(x_1-vy_2)(vx_{2}-y_1)}{x_1-x_2}\right)\doteq Q(x_{1},x_{2},y_{1},y_{2})
\end{equation*}
with $Q(x_{1},x_{2},y_{1},y_{2})$ defined in \eqref{eq:Qform-C}, we obtain:
\begin{equation*}
  \Psi(\tilde{E}^{+}_{[j,n,j],s})\doteq
  \frac{\langle 1\rangle^{2}_{v} (x_{j,1}x_{j,2})^{s_{j}+1}}{\prod^{t=1,2}_{s=1,2}(x_{j,s}-x_{j+1,t})}\cdot
  Q(x_{j,1},x_{j,2},x_{j+1,1},x_{j+1,2})\cdot \Psi(\tilde{E}^{+}_{[j+1,n,j+1],s-2s_{j}}).
\end{equation*}
Using the induction hypothesis for $\Psi(\tilde{E}^{+}_{[j+1,n,j+1],s-2s_{j}})$, we derive~\eqref{eq:Imtilde Eb-3} for $i=j$.
\end{proof}

For more general quantum root vectors $\{E_{\beta,s}\}_{\beta\in\Delta^{+}}^{s\in\BZ}$ of $\qlc$ defined in
\eqref{rootvector1}--\eqref{rootvector2}, their images under $\Psi$ are not so well factorized as for the particular
choices above, but what is actually important is that they behave well under the specialization maps:

\begin{Lem}\label{phirv-C}
For any choices of $s_k$ and $\lambda_k$ in \eqref{rootvector1}--\eqref{rootvector2}, we have:
\begin{equation*}
  \phi_{\beta}(\Psi(E_{\beta,s}))\doteq c_{\beta}\cdot w_{\beta,1}^{s+\kappa_{\beta}}
  \qquad \forall\, (\beta,s)\in\Delta^{+}\times \BZ,
\end{equation*}
where $\{\kappa_{\beta}\}_{\beta\in\Delta^{+}}$ are explicitly given by
\begin{equation}
\label{kappaC}
  \kappa_{\beta}=
  \begin{cases}
    j-i &\quad \mathrm{if}\ \beta=[i,j] \ \ \mathrm{with} \ i\leq j \\
    4n-i-3j-1 &\quad \mathrm{if}\ \beta=[i,n,j]\ \ \mathrm{with} \ i<j \\
    2n-2i  &\quad \mathrm{if}\ \beta=[i,n,i]
  \end{cases}
\end{equation}
and the constants $\{c_{\beta}\}_{\beta\in \Delta^{+}}$ are explicitly given by
\begin{equation}
\label{eq:c-factor-C}
  c_{\beta}=
  \begin{cases}
    \langle 1\rangle_{v}^{|\beta|-1} &\quad \mathrm{if}\ \beta=[i,j] \ \mathrm{or}\ \beta=[n]\\
    \langle 1\rangle_{v}^{|\beta|-2}{\langle 2\rangle_{v}} &\quad \mathrm{if}\ \beta=[i,n]\\
    \langle 1\rangle_{v}^{|\beta|-3} {\langle 2\rangle_{v}}\cdot \prod_{\ell=j}^{n-1}\left\{(v^{2n-2\ell}-1)(v^{2n-2\ell+4}-1)\right\}
    &\quad \mathrm{if}\ \beta=[i,n,j]\\
      \langle 1\rangle_{v}^{|\beta|-3} \langle 2\rangle_{v}^{2}
    &\quad \mathrm{if}\ \beta=[i,n,i]
  \end{cases}
\end{equation}
where $|\beta|$ denotes the height of $\beta$, cf.~\eqref{eq:root-height}.
\end{Lem}

\begin{proof}
It suffices to consider only $\beta=[i,n,i]$, as for the other roots the proof is analogous to that of \cite[Lemma 4.2]{HT24}.
For $\beta=[i,n,i]$, recall that $E_{\beta,s}=[E_{[i,n-1],s_{1}},E_{[i,n],s_{2}}]_{\lambda}$ with $\lambda\in v^{\BZ}$, $s=s_{1}+s_{2}$,
so that $\Psi(E_{\beta,s})=\Psi(E_{[i,n-1],s_{1}})\star \Psi(E_{[i,n],s_{2}})-\lambda \Psi(E_{[i,n],s_{2}})\star \Psi(E_{[i,n-1],s_{1}})$.

First, let us prove that $\phi_{\beta}(\Psi(E_{[i,n-1],s_{1}})\star \Psi(E_{[i,n],s_{2}}))=0$. Consider
\begin{equation*}
\begin{aligned}
    F_{\beta}=\
    & \Psi(E_{[i,n-1],s_{1}})(x_{i,1},\ldots,x_{n-1,1})\Psi(E_{[i,n],s_{2}})(x_{i,2},\ldots,x_{n-1,2},x_{n,1}) \times\\
    & \zeta\left(\frac{x_{n-1,1}}{x_{n-1,2}}\right)\zeta\left(\frac{x_{n-1,1}}{x_{n,1}}\right) \cdot
      \prod_{\ell=i}^{n-2}\left\{\zeta\left(\frac{x_{\ell,1}}{x_{\ell,2}}\right)
      \zeta\left(\frac{x_{\ell,1}}{x_{\ell+1,2}}\right) \zeta\left(\frac{x_{\ell+1,1}}{x_{\ell,2}}\right)\right\}.
\end{aligned}
\end{equation*}
According to \eqref{shuffleproduct}--\eqref{eq:symmetrization}, we have
\begin{equation}
\label{eq:symmeFbeta}
\begin{aligned}
  \Psi(E_{[i,n-1],s_{1}}) & \star \Psi(E_{[i,n],s_{2}})  \doteq \\
  & \sum_{(\sigma_{i},\dots,\sigma_{n-1})\in \mathfrak{S}^{n-i}_{2}}
    F_{\beta}\left(x_{i,\sigma_{i}(1)},x_{i,\sigma_{i}(2)},\dots,x_{n-1,\sigma_{n-1}(1)},x_{n-1,\sigma_{n-1}(2)},x_{n,1}\right).
\end{aligned}
\end{equation}
Using $\sigma$ to denote $(\sigma_{i},\dots,\sigma_{n-1})\in \mathfrak{S}^{n-i}_{2}$, we can write each summand
above as $\sigma(F_{\beta})$. We note that evaluating the $\phi_{\beta}$-specialization of $\sigma(F_{\beta})$
in \eqref{eq:symmeFbeta} is equivalent to evaluating the $\phi_{\beta}$-specialization of $F_{\beta}$ with respect
to different splittings of the variables $\{x^{(\beta,1)}_{\ell,t}\}_{\ell\in\beta}^{1\leq t\leq \nu_{\beta,\ell}}$.
To this end, we shall write $o(x^{(*,*)}_{*,*})=1$ if a variable $x^{(*,*)}_{*,*}$ is plugged into $\Psi(E_{[i,n-1],s_{1}})$,
and $o(x^{(*,*)}_{*,*})=2$ if it is plugged into $\Psi(E_{[i,n],s_{2}})$. According to~\eqref{spe-C-2}, the
$\phi^{(1)}_{\beta}$-specialization of the corresponding summand vanishes unless
$$
  o(x^{(\beta,1)}_{i,1})= o(x^{(\beta,1)}_{i+1,1})=\cdots=o(x^{(\beta,1)}_{n-1,1}) \qquad \mathrm{and} \qquad
  o(x^{(\beta,1)}_{i,2})= o(x^{(\beta,1)}_{i+1,2})=\cdots=o(x^{(\beta,1)}_{n-1,2}).
$$
We still have two cases to consider:
\smallskip

\begin{itemize}[leftmargin=0.7cm]

\item
if $o(x^{(\beta,1)}_{i,1})=\cdots=o(x^{(\beta,1)}_{n-1,1})=2$ and $o(x^{(\beta,1)}_{i,2})=\cdots=o(x^{(\beta,1)}_{n-1,2})=1$,
then $o(x^{(\beta,1)}_{n,1})=2$, and the $\phi^{(1)}_{\beta}$-specialization of the corresponding summand vanishes due to
$\zeta\left(\frac{x^{(\beta,1)}_{n-1,2}}{x^{(\beta,1)}_{n,1}}\right)$;

\item
if $o(x^{(\beta,1)}_{i,1})=\cdots=o(x^{(\beta,1)}_{n-1,1})=1$ and
$o(x^{(\beta,1)}_{i,2})=\cdots=o(x^{(\beta,1)}_{n-1,2})=o(x^{(\beta,1)}_{n,1})=2$, then the product of $\zeta$-factors
\begin{equation}
\label{eq:zeta-product-C}
  \prod^{n-2}_{\ell=i} \left\{
  \zeta\left(\frac{x^{(\beta,1)}_{\ell,1}}{x^{(\beta,1)}_{\ell,2}}\right)
  \zeta\left(\frac{x^{(\beta,1)}_{\ell,1}}{x^{(\beta,1)}_{\ell+1,2}}\right)
  \zeta\left(\frac{x^{(\beta,1)}_{\ell+1,1}}{x^{(\beta,1)}_{\ell,2}}\right)\right\}
\end{equation}
contributes $B_{\beta}$ of~\eqref{Bfactor-C} towards the $\phi^{(1)}_{\beta}$-specialization of the corresponding summand,
and so the overall $\phi_\beta$-specialization vanishes due to the  $\zeta$-factors
  $\zeta\left(\frac{x^{(\beta,1)}_{n-1,1}}{x^{(\beta,1)}_{n-1,2}}\right)
   \zeta\left(\frac{x^{(\beta,1)}_{n-1,1}}{x^{(\beta,1)}_{n,1}}\right)$.

\end{itemize}
This completes the proof of $\phi_{\beta}(\Psi(E_{[i,n-1],s_{1}})\star \Psi(E_{[i,n],s_{2}}))=0$.

The evaluation of $\phi_{\beta}(\Psi(E_{[i,n],s_{2}})\star \Psi(E_{[i,n-1],s_{1}}))$ is analogous. We shall write
$o(x^{(*,*)}_{*,*})=1$ if $x^{(*,*)}_{*,*}$ is plugged into $\Psi(E_{[i,n],s_{2}})$, and $o(x^{(*,*)}_{*,*})=2$ if
it is plugged into $\Psi(E_{[i,n-1],s_{1}})$. As before, the $\phi^{(1)}_{\beta}$-specialization of the corresponding
summand vanishes unless
$$
  o(x^{(\beta,1)}_{i,1})= o(x^{(\beta,1)}_{i+1,1})=\cdots=o(x^{(\beta,1)}_{n-1,1}) \qquad \mathrm{and} \qquad
  o(x^{(\beta,1)}_{i,2})= o(x^{(\beta,1)}_{i+1,2})=\cdots=o(x^{(\beta,1)}_{n-1,2}).
$$
We have two cases to consider:
\smallskip

\begin{itemize}[leftmargin=0.7cm]

\item
if $o(x^{(\beta,1)}_{i,1})=\cdots=o(x^{(\beta,1)}_{n-1,1})=o(x^{(\beta,1)}_{n,1})=1$ and
$o(x^{(\beta,1)}_{i,2})=\cdots=o(x^{(\beta,1)}_{n-1,2})=2$, then the product~\eqref{eq:zeta-product-C}
contributes $B_{\beta}$ to the $\phi^{(1)}_{\beta}$-specialization of the corresponding summand, and so again the overall
$\phi_\beta$-specialization vanishes due to the $\zeta$-factor $\zeta\left(\frac{x^{(\beta,1)}_{n-1,1}}{x^{(\beta,1)}_{n-1,2}}\right)$;

\item
if $o(x^{(\beta,1)}_{i,1})=\cdots=o(x^{(\beta,1)}_{n-1,1})=2$ and
$o(x^{(\beta,1)}_{i,2})=\cdots=o(x^{(\beta,1)}_{n-1,2})=o(x^{(\beta,1)}_{n,1})=1$, then this is the only
summand that does not vanish under the specialization $\phi_{\beta}$, and its $\phi^{(1)}_{\beta}$-specialization is
\begin{align*}
  \doteq &
  \eval{\Psi(E_{[i,n],s_{2}})}^{x_{n,1}\mapsto v^{-n}w'_{\beta,1}}_{x_{\ell\neq n,1}\mapsto v^{1-\ell}w'_{\beta,1} }\cdot
  \eval{\Psi(E_{[i,n-1],s_{1}})}_{x_{\ell,1}\mapsto v^{1-\ell}w_{\beta,1}} \\
  & \qquad \qquad \qquad \times   B_{\beta}\cdot \frac{(w'_{\beta,1}-v^{-2}w_{\beta,1})(w'_{\beta,1}-v^{4}w_{\beta,1})}{w'_{\beta,1}-w_{\beta,1}}.
\end{align*}
Dividing by $B_\beta$ and specializing further $w'_{\beta,1}\mapsto v^{2}w_{\beta,1}$, we thus get
\[
  \phi_{\beta}(\Psi(E_{[i,n],s_{2}})\star \Psi(E_{[i,n-1],s_{1}}))\doteq
  \langle 1 \rangle_v^{2n-2i-2} \langle 2 \rangle_v^{2} \cdot w^{s+2n-2i}_{\beta,1}.
\]

\end{itemize}
This implies the desired result
  $\phi_{\beta}(\Psi(E_{\beta,s}))\doteq \langle 1 \rangle_v^{2n-2i-2} \langle 2 \rangle_v^{2} \cdot w^{s+2n-2i}_{\beta,1}$.
\end{proof}

Let us generalize the above lemma by computing $\phi_{\unl{d}}(\Psi(E_{h}))$ for any $h\in H_{\unl{k},\unl{d}}$.
Note that
\begin{equation}
\label{eq:formulapsiEh}
  \Psi(E_{h})\ =
  \prod_{\beta\in\Delta^{+}}\limits^{\rightarrow}
    \left(\Psi(E_{\beta,r_{\beta}(h,1)})\star\cdots\star\Psi( E_{\beta,r_{\beta}(h,d_\beta)})\right)
  \qquad \forall\, h\in H_{\unl{k},\unl{d}}.
\end{equation}
Here, the product refers to the shuffle product and the arrow $\rightarrow$ over the product sign refers to the
order \eqref{lynorderC}, and $r_{\beta}(h,1)\leq \dots\leq r_{\beta}(h,d_{\beta})$  is obtained by listing all
integers $r\in\mathbb{Z}$ with multiplicity $h(\beta,r)>0$ in the non-decreasing order. Denote the variables in
$\Psi(E_{\beta,r_{\beta}(h,s)})$ by $\{z^{(\beta,s)}_{i,t}\}^{1\leq t\leq \nu_{\beta,i}}_{i\in\beta}$ (as we reserve
$\{x^{(\beta,s)}_{i,t}\}^{1\leq t\leq \nu_{\beta,i}}_{i\in\beta}$ for the variables of splittings below), and let
\begin{equation*}
    F_{h}\coloneqq
    \prod_{\substack{\beta\in\Delta^{+} \\ 1\leq s\leq d_{\beta}}} \Psi(E_{\beta,r_{\beta}(h,s)})
    \prod^{(\beta,p)<(\beta',q)}_{\substack{\beta,\beta'\in\Delta^{+} \\ 1\leq p\leq d_{\beta},1\leq q\leq d_{\beta'}}}
    \prod^{j\in\beta'}_{i\in\beta}\prod_{1\leq t\leq \nu_{\beta,i}}^{1\leq r\leq \nu_{\beta',j}}
      \zeta\left(\frac{z^{(\beta,p)}_{i,t}}{z^{(\beta',q)}_{j,r}}\right),
\end{equation*}
where the order $(\beta,p)<(\beta',q)$ is as in \eqref{orderbetas}. Then we have
\begin{equation}
\label{eq:part-symm}
  \Psi(E_{h})\doteq
  \sum_{\sigma\in\mathfrak{S}_{\unl{k}}} \sigma\big(F_{h}(\{z^{(*,*)}_{*,*}\})\big)=
  \sum_{\sigma\in\mathfrak{S}_{\unl{k}}} F_{h}\big(\{\sigma(z^{(*,*)}_{*,*})\}\big).
\end{equation}
To evaluate the $\phi_{\unl{d}}$-specialization of each term $\sigma(F_{h})$ in~\eqref{eq:part-symm}, it is
equivalent to evaluate  the $\phi_{\unl{d}}$-specialization of $F_{h}$ with respect to different splittings
of the variables $x^{(*,*)}_{*,*}$. We shall write $ o(x^{(*,*)}_{*,*})=(\beta,s)$ if a variable $x^{(*,*)}_{*,*}$
is plugged into $\Psi(E_{\beta,r_{\beta}(h,s)})$. Then, we have:

\begin{Prop}\label{prop:phidEh}
For a summand $\sigma(F_{h})$ in the symmetrization \eqref{eq:part-symm}, we have $\phi_{\unl{d}}(\sigma(F_{h}))=0$ unless
for any $\beta\in\Delta^{+}$ and $1\leq s\leq d_{\beta}$, there is $s'$ with $1\leq s'\leq d_{\beta}$ so that
\begin{equation}
\label{fsunld}
  o(x^{(\beta,s')}_{i,t})=(\beta,s)\  \text{for any} \ i\in\beta\  \text{and}\  1\leq t\leq \nu_{\beta,i},
\end{equation}
that is we plug the variables $x^{(\beta,s')}_{*,*}$ into the same function $\Psi(E_{\beta,r_{\beta}(h,s)})$.
\end{Prop}

\begin{proof}
We prove this result by an induction on $n$.

\noindent
\underline{Step 1} (base of induction): Verification for type $C_{2}$.

In this case, $\Delta^{+}=\{[1]<[1,2,1]<[1,2]<[2]\}$. For $\beta=[1,2,1]$, $B_{\beta}$ of \eqref{Bfactor-C} is trivial,
and the specialization map $\phi_{\beta,s}$ is
\begin{equation}
\label{eq:spe121}
  x^{(\beta,s)}_{1,1}\mapsto w_{\beta,s},\quad
  x^{(\beta,s)}_{1,2}\mapsto v^{2}w_{\beta,s},\quad
  x^{(\beta,s)}_{2,1}\mapsto w_{\beta,s}.
\end{equation}
\begin{itemize}[leftmargin=0.7cm]

\item Case 1: $\beta=[1]$.

If \eqref{fsunld} fails for $\beta=[1]$,  then there is a variable $x^{(\eta,r)}_{1,t}$ with $\eta>[1]$ and
$o(x^{(\eta,r)}_{1,t})=([1],s)$ for some $1\leq s\leq d_{[1]}$. We can also assume that $s$ is the smallest number
with this property, which means for any $1\leq s'<s$, we already plug a variable $x^{(\beta,*)}_{1,1}$ into
$\Psi(E_{\beta,r_{\beta}(h,s')})$. If $\eta=[1,2]$ or $\eta=[1,2,1]$ and $t=2$, then $\phi_{\unl{d}}(\sigma(F_{h}))=0$
due to the $\zeta$-factors $\zeta\left(\frac{x^{(\eta,r)}_{1,1}}{x^{(\eta,r)}_{2,1}}\right)$ or
$\zeta\left(\frac{x^{(\eta,r)}_{1,2}}{x^{(\eta,r)}_{2,1}}\right)$ respectively. Otherwise $\eta=[1,2,1]$ and $t=1$,
so that $o(x^{(\eta,r)}_{1,2})>o(x^{(\eta,r)}_{1,1})$ (by the minimality of $s$), and $\phi_{\unl{d}}(\sigma(F_{h}))=0$
due to $\zeta\left(\frac{x^{(\eta,r)}_{1,1}}{x^{(\eta,r)}_{1,2}}\right)$.

\item Case 2: $\beta=[1,2,1]$.

Assuming \eqref{fsunld} holds for any $([1],s)$ with $1\leq s\leq d_{[1]}$, let us prove that $\phi_{\unl{d}}(\sigma(F_{h}))=0$
unless \eqref{fsunld} holds for any $([1,2,1],s)$ with $1\leq s\leq d_{[1,2,1]}$. Suppose $o(x^{(\eta,p)}_{1,q})=([1,2,1],1)$.
From \eqref{eq:spe121}, we see that $\phi_{\unl{d}}(\sigma(F_{h}))=0$ unless
$o(x^{([1,2,1],s)}_{1,1})\geq  o(x^{([1,2,1],s)}_{1,2})\geq o(x^{([1,2,1],s)}_{2,1})$ for any $1\leq s\leq d'_{[1,2,1]}$,
due to the $\zeta$-factors
$
  \zeta\left(\frac{x^{([1,2,1],s)}_{1,1}}{x^{([1,2,1],s)}_{1,2}}\right)
  \zeta\left(\frac{x^{([1,2,1],s)}_{1,2}}{x^{([1,2,1],s)}_{2,1}}\right).
$
Since $1\in\eta$ and $\eta\geq [1,2,1]$, we have $2\in\eta$ and we can assume that
$o(x^{(\eta,p)}_{1,q})\geq o(x^{(\eta,p)}_{2,1})$, as otherwise $\phi_{\unl{d}}(\sigma(F_{h}))=0$, so that
$o(x^{(\eta,p)}_{2,1})=([1,2,1],1)$. Yet there is another variable that satisfies $o(x^{(\eta',p')}_{1,q'})=([1,2,1],1)$.
If $(\eta',p')\neq (\eta,p)$, then we have $o(x^{(\eta',p')}_{1,q'})<o(x^{(\eta',p')}_{2,1})$ and so $\phi_{\unl{d}}(\sigma(F_{h}))=0$.
If $(\eta',p')= (\eta,p)$, then $\eta=[1,2,1]$, and so all the variables $x^{([1,2,1],p)}_{*,*}$ are plugged into
$\Psi(E_{[1,2,1],r_{[1,2,1]}(h,1)})$. Proceeding the same way, we get $\phi_{\unl{d}}(\sigma(F_{h}))=0$ unless
\eqref{fsunld} holds for any $([1,2,1],s)$ with $1\leq s\leq d_{[1,2,1]}$.

\item Case 3: $\beta=[1,2]$.

Assuming \eqref{fsunld} holds for $\beta=[1]$ and $\beta=[1,2,1]$, choose a variable satisfying $o(x^{(\eta,p)}_{1,q})=([1,2],1)$.
As $\eta\geq [1,2]$ and $1\in \eta$, it must be $\eta=[1,2]$, $q=1$. And we know $\phi_{\unl{d}}(\sigma(F_{h}))=0$ unless
$o(x^{(\eta,p)}_{1,1})\geq o(x^{(\eta,p)}_{2,1})$, so that $ o(x^{(\eta,p)}_{2,1})=([1,2],1)$. Proceeding the same way, we get
$\phi_{\unl{d}}(\sigma(F_{h}))=0$ unless \eqref{fsunld} holds for any $([1,2],s)$ with $1\leq s\leq d_{[1,2]}$.

\item Case 4:  $\beta=[2]$.

If \eqref{fsunld} holds for any $\beta<[2]$, then it must also hold for $\beta=[2]$.

\end{itemize}
This completes the verification of the result for $C_2$.

\medskip
\noindent
\underline{Step 2} (step of induction): Assuming the validity for type $C_{n-1}$, let us prove it for $C_n$.

Fix  $\gamma\in\Delta^{+}$ and $1\leq p\leq d_{\gamma}$. Then, it suffices to prove that if  for any $(\beta,s)<(\gamma,p)$,
we already chose $s'$ such that all the variables $x^{(\beta,s')}_{*,*}$ are plugged into $\Psi(E_{\beta,r_{\beta}(h,s)})$,
then $\phi_{\unl{d}}(\sigma(F_{h}))=0$ unless we choose $p'$ and plug all the variables $x^{(\gamma,p')}_{*,*}$ into
$\Psi(E_{\gamma,r_{\gamma}(h,p)})$. To this end, we present case-by-case study:
\begin{itemize}[leftmargin=0.7cm]

\item Case 1: $\gamma=[1,j]$ with $1\leq j\leq n-1$, and suppose $o(x^{(\eta,q)}_{1,t})=(\gamma,p)$ with $\eta\geq \gamma$.

If $\eta=[1,\ell]$ with $j<\ell\leq n$, then $t=1$ and $\phi_{\unl{d}}(\sigma(F))=0$ from type $A_n$ results.

If $\eta=[1,n,1]$ and $t=2$, then $\phi_{\unl{d}}(\sigma(F))=0$ unless
$o(x^{(\eta,q)}_{1,2})\geq\cdots\geq  o(x^{(\eta,q)}_{n-1,2}) \geq o(x^{(\eta,q)}_{n,1})$ due to the  $\zeta$-factors
  $\zeta\left(\frac{x^{(\eta,q)}_{1,2}}{x^{(\eta,q)}_{2,2}}\right) \cdots
   \zeta\left(\frac{x^{(\eta,q)}_{n-1,2}}{x^{(\eta,q)}_{n,1}}\right)$.
As $o(x^{(\eta,q)}_{1,2})=(\gamma,p)$ and we already plugged variables into all
$\Psi(E_{\beta,r_{\beta}(h,s)})$ with $(\beta,s)<(\gamma,p)$, we get $\phi_{\unl{d}}(\sigma(F))=0$
unless $o(x^{(\eta,q)}_{1,2})=\cdots= o(x^{(\eta,q)}_{n-1,2}) = o(x^{(\eta,q)}_{n,1})$, which is impossible
as $n\notin \gamma$. Thus $\phi_{\unl{d}}(\sigma(F))=0$.

If $\eta=[1,n,1]$ and $t=1$, then we likewise get $\phi_{\unl{d}}(\sigma(F))=0$ unless
\begin{equation}
\label{eq:1n1analysis-1}
  o(x^{(\eta,q)}_{1,2})\geq \cdots \geq o(x^{(\eta,q)}_{n-1,2})\geq o(x^{(\eta,q)}_{n,1}) >
  o(x^{(\eta,q)}_{1,1})=\cdots=o(x^{(\eta,q)}_{n-1,1})=(\gamma,p).
\end{equation}
The product of $\zeta$-factors (cf.~\eqref{eq:zeta-product-C})
$
  \prod^{n-2}_{\ell=1}\left\{\zeta\left(\frac{x^{(\eta,q)}_{\ell,1}}{x^{(\eta,q)}_{\ell,2}}\right)
  \zeta\left(\frac{x^{(\eta,q)}_{\ell,1}}{x^{(\eta,q)}_{\ell+1,2}}\right)
  \zeta\left(\frac{x^{(\eta,q)}_{\ell+1,1}}{x^{(\eta,q)}_{\ell,2}}\right)\right\}
$
contributes the $B_{\eta}$ factor of \eqref{Bfactor-C}, while the remaining $\zeta$-factors
  $\zeta\left(\frac{x^{(\eta,q)}_{n-1,1}}{x^{(\eta,q)}_{n-1,2}}\right)
   \zeta\left(\frac{x^{(\eta,q)}_{n-1,1}}{x^{(\eta,q)}_{n,1}}\right)$
contribute $0$ when specializing $w'_{\eta,q}$ to $v^{2}w_{\eta,q}$ (cf.~\eqref{spe-C-3}),
and so $\phi_{\unl{d}}(\sigma(F))=0$.

If $\eta=[1,n,j]$ with $2\leq j\leq n-1$, then we similarly get $\phi_{\unl{d}}(\sigma(F))=0$ unless
\[
  o(x^{(\eta,q)}_{1,1})=\cdots=o(x^{(\eta,q)}_{n,1})= o(x^{(\eta,q)}_{n-1,2})= \cdots= o(x^{(\eta,q)}_{j,2})=(\gamma,p),
\]
which is impossible, as $n\notin \gamma$.

Finally, if $\eta=\gamma=[1,j]$, then $\phi_{\unl{d}}(\sigma(F))=0$ unless
$o(x^{(\gamma,q)}_{1,1})=\cdots=o(x^{(\gamma,q)}_{j,1})=(\gamma,p)$, that is we plug all the variables
$x^{(\gamma,q)}_{*,*}$ into $\Psi(E_{\gamma,r_{\gamma}(h,p)})$.

\item Case 2: $\gamma=[1,n,1]$, and suppose $o(x^{(\eta,q)}_{1,t})=(\gamma,p)$ with $\eta\geq \gamma$.

Since $\nu_{\gamma,1}=2$, there is another variable $x^{(\eta',q')}_{1,t'}$ with $o(x^{(\eta',q')}_{1,t'})=(\gamma,p)$.
If $\eta,\eta'>\gamma$, then $t=t'=1$ and $\phi_{\unl{d}}(\sigma(F))=0$ unless
\begin{align}
\label{eq:1n1analysis-2}
  (\gamma,p)=o(x^{(\eta,q)}_{1,1})=\cdots= o(x^{(\eta,q)}_{n,1})\quad  \mathrm{and} \quad
  (\gamma,p)=o(x^{(\eta',q')}_{1,1})= \cdots = o(x^{(\eta',q')}_{n,1}),
\end{align}
which is impossible as $\nu_{\gamma,n}=1$.

If exactly one of  $\eta,\eta'$ is $\gamma$, then without loss of generality we can assume $\eta=\gamma$ and $\eta'>\gamma$,
so that $t'=1$. If $t=1$, then the same analysis as after \eqref{eq:1n1analysis-1} implies $\phi_{\unl{d}}(\sigma(F))=0$.
If $t=2$, then the same analysis as after \eqref{eq:1n1analysis-2} implies $\phi_{\unl{d}}(\sigma(F))=0$.

If $\eta=\eta'=\gamma$ and $q\neq q'$, then we consider three cases depending on the values of $t$, $t'$. If $t=t'=2$,
then analysis similar to that after \eqref{eq:1n1analysis-2} implies that $\phi_{\unl{d}}(\sigma(F))=0$. If exactly one
of $t,t'$ is equal to $1$, then the same analysis as after \eqref{eq:1n1analysis-1} implies $\phi_{\unl{d}}(\sigma(F))=0$ again.
Finally, if $t=t'=1$, then we know $\phi_{\unl{d}}(\sigma(F))=0$ unless
\begin{align*}
  (\gamma,p)=o(x^{(\eta,q)}_{1,1})=\cdots= o(x^{(\eta,q)}_{n-1,1})\quad  \mathrm{and} \quad
  (\gamma,p)=o(x^{(\eta',q')}_{1,1})= \cdots = o(x^{(\eta',q')}_{n-1,1}).
\end{align*}
Since $\nu_{\gamma,n}=1$ and $q\neq q'$, we have $o(x^{(\eta,q)}_{n,1})>(\gamma,p)$ or $o(x^{(\eta',q')}_{n,1})>(\gamma,p)$,
and therefore the same analysis as after \eqref{eq:1n1analysis-1} implies $\phi_{\unl{d}}(\sigma(F))=0$.

Finally if $\eta=\eta'=\gamma$ and  $q=q'$, then $\phi_{\unl{d}}(\sigma(F))=0$ unless we plug all the variables
$x^{(\gamma,q)}_{*,*}$ into $\Psi(E_{\gamma,r_{\gamma}(h,p)})$.

\item Case 3: $\gamma=[1,n]$ or $\gamma=[1,n,j]$ with $2<j\leq n-1$. We also suppose $o(x^{(\eta,q)}_{1,t})=(\gamma,p)$.

If $\eta=[1,n,k]> \gamma$, then $\phi_{\unl{d}}(\sigma(F))=0$ unless
\[
  o(x^{(\eta,q)}_{1,1})=\cdots=o(x^{(\eta,q)}_{n,1})= o(x^{(\eta,q)}_{n-1,2})= \cdots= o(x^{(\eta,q)}_{k,2})=(\gamma,p).
\]
The latter is impossible for $k<j$ as $\nu_{\gamma,k}=1$. Thus $\phi_{\unl{d}}(\sigma(F))=0$ unless $\eta=\gamma$ and
we plug all the variables $x^{(\gamma,q)}_{*,*}$ into $\Psi(E_{\gamma,r_{\gamma}(h,p)})$.

\item Case 4: $\gamma=[1,n,2]$, and suppose $o(x^{(\eta,q)}_{1,t})=(\gamma,p)$ with $\eta\geq \gamma$.

If \eqref{fsunld} holds for any $(\beta,s)<(\gamma,p)$, then we must have $\eta=\gamma$ and so $\phi_{\unl{d}}(\sigma(F))=0$
unless we plug all the variables $x^{(\gamma,q)}_{*,*}$ into $\Psi(E_{\gamma,r_{\gamma}(h,p)})$.

\item Case 5: $\gamma>[1,n,2]$.

If \eqref{fsunld} holds for any $(\beta,s)<([2],1)$, then we can use the induction assumption for $C_{n-1}$ to conclude that
$\phi_{\unl{d}}(\sigma(F))=0$ unless \eqref{fsunld} holds for all $(\gamma,p)$.

\end{itemize}

This completes the proof.
\end{proof}

Combining Lemma \ref{phirv-C} and Proposition \ref{prop:phidEh}, we obtain the formula for
$\phi_{\unl{d}}(\Psi(E_{h}))$ with $h\in H_{\unl{k},\unl{d}}$:

\begin{Prop}\label{spekpC}
For any $h\in H_{\unl{k},\unl{d}}$, we have
\begin{equation}
\label{speehC}
  \phi_{\unl{d}}(\Psi(E_{h}))\doteq
  \prod_{\beta,\beta'\in \Delta^+}^{\beta<\beta'}G_{\beta,\beta'}\cdot
  \prod_{\beta\in\Delta^{+}}(c_{\beta}^{d_{\beta}}\cdot G_{\beta})\cdot
  \prod_{\beta\in\Delta^{+}} P_{\lambda_{h,\beta}}
\end{equation}
with $\{P_{\lambda_{h,\beta}}\}_{\beta\in\Delta^{+}}$ given by
\begin{equation}
\label{hlp-C}
  P_{\lambda_{h,\beta}}=\text{\rm Sym}_{\mathfrak{S}_{d_{\beta}}}
  \left(w_{\beta,1}^{r_{\beta}(h,1)}\cdots w_{\beta,d_{\beta}}^{r_{\beta}(h,d_{\beta})}
       \prod_{1\leq i<j\leq d_{\beta}}\frac{w_{\beta,i}-v_{\beta}^{-2}w_{\beta,j}}{w_{\beta,i}-w_{\beta,j}}\right),
\end{equation}
where $\{r_{\beta}(h,s)\}^{1\leq s\leq d_{\beta}}_{\beta\in\Delta^{+}}$ are defined after \eqref{eq:formulapsiEh},
the constants $\{c_{\beta}\}_{\beta\in\Delta^{+}}$ are as in Lemma~\ref{phirv-C}, and the terms $G_{\beta,\beta'},G_{\beta}$
are products of linear factors $w_{\beta,s}$ and $w_{\beta,s}-v^{\BZ}w_{\beta',s'}$ which are independent of
$h\in H_{\unl{k},\unl{d}}$ and are $\mathfrak{S}_{\unl{d}}$-symmetric (the factors $G_\beta$ are specified in
Remark~\ref{rmk:G-Ctype}).
\end{Prop}

\begin{Rk}\label{rankreduction}
Proposition~\ref{spekpC}  (cf.~\cite[Lemma 3.17]{Tsy18}) features a ``\textbf{rank $1$ reduction}'': each $P_{\lambda_{h,\beta}}$
from~\eqref{hlp-C} can be viewed as the shuffle product $x^{r_{\beta}(h,1)}\star\cdots \star x^{r_{\beta}(h,d_\beta)}$ in
the shuffle algebra of type $A_1$, evaluated at  $\{w_{\beta,s}\}_{s=1}^{d_\beta}$.
\end{Rk}

Using the same arguments as in the proof of Proposition \ref{prop:phidEh}, we can now evaluate $\phi_{\unl{d}'}(\Psi(E_{h}))$
for any $\unl{d}'<\unl{d}\in\text{KP}(\unl{k})$ and $h\in H_{\unl{k},\unl{d}}$:

\begin{Prop}\label{vanishC}
Lemma {\rm \ref{vanish}} is valid for type $C_{n}$, with $\phi_{\unl{d}}$ of~\eqref{spe-C-1}--\eqref{spe-C-3}.
\end{Prop}

\begin{proof}
Given $\unl{d}'<\unl{d} \in \mathrm{KP}(\unl{k})$, let $\gamma\in \Delta^+$ be the smallest root such that
$d'_{\gamma}<d_{\gamma}$, and let
\begin{equation*}
  \bigsqcup^{1\leq s\leq d'_{\beta}}_{\beta\in\Delta^{+}}
  \Big\{x'^{(\beta,s)}_{i,t} \, \Big| \, i\in I, 1\leq t\leq \nu_{\beta,i}\Big\} \,
\end{equation*}
be any splitting of the variables for $\phi_{\unl{d}'}$. To evaluate the $\phi_{\unl{d}'}$-specialization of each
summand $\sigma(F_{h})$ in the symmetrization \eqref{eq:part-symm}, we write $o(x'^{(*,*)}_{*,*})=(\beta,s)$ if
a variable $x'^{(*,*)}_{*,*}$ is plugged into $\Psi(E_{\beta,r_{\beta}(h,s)})$. Arguing as in the Step 1 of the proof
of Proposition \ref{prop:phidEh}, we know that Lemma \ref{vanish} is valid for type $C_{2}$. Now assuming that
Lemma \ref{vanish} is valid for type $C_{n-1}$, let us prove its validity for type $C_n$. First, according to
the proof of Proposition \ref{prop:phidEh}, we know $\phi_{\unl{d}'}(\sigma(F_{h}))=0$ unless for any
$(\beta,s)\leq (\gamma,d'_{\gamma})$, there is some $1\leq s'\leq d'_{\beta}$ such that all the variables $x'^{(\beta,s')}_{*,*}$
are plugged into $\Psi(E_{\beta,r_{\beta}(h,s)})$. Then there is $\eta>\gamma$ and $1\leq q\leq d'_{\eta}$ with
$o(x'^{(\eta,q)}_{1,t})=(\gamma,d'_{\gamma}+1)$. Using the same analysis as in the Step 2 of the proof of
Proposition~\ref{prop:phidEh}, we then get $\phi_{\unl{d}'}(\sigma(F_{h}))=0$. This completes the proof.
\end{proof}

\begin{Rk}\label{rmk:G-Ctype}
The factors $\{G_{\beta}\}_{\beta\in\Delta^{+}}$ featuring in \eqref{speehC} are explicitly given by:
\begin{itemize}[leftmargin=0.7cm]

\item
If $\beta=[i,j]$ with $1\leq i\leq j<n$ or $i=j=n$, then
\begin{equation}
\label{eq:formulagbeta-C-1}
  G_{\beta}=
  \prod_{1\leq s\leq d_{\beta}} w_{\beta,s}^{\kappa_{\beta}}
  \prod_{1\leq s\neq s'\leq d_{\beta}} (w_{\beta,s}-v^{2}w_{\beta,s'})^{j-i}.
\end{equation}

\item
If $\beta=[i,n]$ with $1\leq i<n$, then
\begin{equation}
\label{eq:formulagbeta-C-2}
 G_{\beta}=
 \prod_{1\leq s\leq d_{\beta}} w_{\beta,s}^{\kappa_{\beta}}
 \prod_{1\leq s\neq s'\leq d_{\beta}} \big\{(w_{\beta,s}-v^{2}w_{\beta,s'})^{n-i-1}(w_{\beta,s}-v^{4}w_{\beta,s'})\big\}.
\end{equation}

\item
If $\beta=[i,n,j]$ with $1\leq i<j \leq  n-1$, then
\begin{equation}
\label{eq:formulagbeta-C-3}
\begin{aligned}
  G_{\beta}=
  &\prod_{1\leq s\leq d_{\beta}} w_{\beta,s}^{\kappa_{\beta}}
   \prod_{1\leq s\neq s'\leq d_{\beta}} \big\{(w_{\beta,s}-v^{2}w_{\beta,s'})^{2n-i-j-1}(w_{\beta,s}-v^{4}w_{\beta,s'})\big\}\times \\
  &\prod_{1\leq s\neq s'\leq d_{\beta}}
    \left\{ \prod_{\ell=j}^{n-2} (w_{\beta,s}-v^{2n-2\ell}w_{\beta,s'})
            \prod_{\ell=j}^{n-1}(w_{\beta,s}-v^{2n-2\ell+4}w_{\beta,s'}) \right\}.
\end{aligned}
\end{equation}

\item
If $\beta=[i,n,i]$ with $1\leq i < n$, then
\begin{equation}
\label{eq:formulagbeta-C-4}
\begin{aligned}
  G_{\beta}
  &=\prod_{1\leq s\leq d_{\beta}} w_{\beta,s}^{\kappa_{\beta}}
    \prod_{1\leq s\neq s'\leq d_{\beta}} (w_{\beta,s}-v^{2}w_{\beta,s'})^{2n-2i-1}\times \\
  &\quad \prod_{1\leq s\neq s'\leq d_{\beta}} \big\{ (w_{\beta,s}-w_{\beta,s'})^{n-i-1}  (w_{\beta,s}-v^{4}w_{\beta,s'})^{n-i} \big\}.
\end{aligned}
\end{equation}
\end{itemize}
\end{Rk}

The factors $G_{\beta,\beta'}$ featuring in \eqref{speehC} can be computed recursively, which  shall be used in the proof
of our next result:

\begin{Prop}\label{spanC}
Lemma {\rm \ref{span}}  is valid for type $C_{n}$, with $\phi_{\unl{d}}$ of \eqref{spe-C-1}--\eqref{spe-C-3}.
\end{Prop}

\begin{proof}
The wheel conditions~\eqref{eq:Wheel-Cn} for $F \in S_{\unl{k}}$, together with the condition $\phi_{\unl{d}'}(F)=0$
for any $\unl{d}'\in \text{\rm KP}(\unl{k})$ satisfying $\unl{d}'<\unl{d}$, guarantee that $\phi_{\unl{d}}(F)$
(which is a Laurent polynomial in the variables $\{w_{\beta,s}\}$) vanishes under specific specializations
$w_{\beta,s}=v^\#\cdot w_{\beta',s'}$. To evaluate the aforementioned powers $\#$ of $v$ and the orders of vanishing,
let us view $\phi_{\unl{d}}$ as a step-by-step specialization in each interval $[\beta]$. We note that this computation
is local with respect to any fixed pair $(\beta,s)\leq (\beta',s')$. We set $G_{\beta,\beta}=G_{\beta}$.
For any pair $\beta\leq \beta'$, consider
\begin{equation*}
  \unl{d}=
  \begin{cases}
    \big\{d_{\beta}=2,\, \mathrm{and}\  d_{\gamma}=0\ \mathrm{for\ other}\ \gamma\big\} & \mathrm{if}\ \beta=\beta'\\
    \big\{d_{\beta}=d_{\beta'}=1,\, \mathrm{and}\  d_{\gamma}=0\ \mathrm{for\ other}\ \gamma\big\} & \mathrm{if}\ \beta<\beta'
  \end{cases}
\end{equation*}
and let $\unl{d}\in\text{KP}(\unl{k})$. According to Proposition~\ref{spekpC} and Remark \ref{rankreduction}, it suffices
to show that for any $F\in S_{\unl{k}}$, the specialization $\phi_{\unl{d}}(F)$ is divisible by $G_{\beta,\beta'}$ if
$\phi_{\unl{d}'}(F)=0$ for any $\unl{d}'<\unl{d}$. Using $A_n$-type results and the induction (i.e.\ assuming the result
holds for type $C_{n-1}$), we still have the following cases to analyze (henceforth, we shall use the notation
$w-v^{\pm k}w'$ to denote the product $(w-v^{k}w')(w'-v^{k}w)$):
\begin{itemize}[leftmargin=0.7cm]

\item $\beta=\beta'=[1,n,j]$ with $1<j<n$.

If $j=n-1$, then
\begin{equation*}
  G_{\beta}={w_{\beta,1}^2w_{\beta,2}^2} (w_{\beta,1}-v^{\pm 2}w_{\beta,2})(w_{\beta,1}-v^{\pm 6}w_{\beta,2})\cdot G_{\alpha}
  \quad \mathrm{with} \quad \alpha=[1,n].
\end{equation*}
For any $F\in S_{\unl{k}}$, as we specialize all the variables but $\{x^{(\beta,1)}_{n-1,2},x^{(\beta,2)}_{n-1,2}\}$,
we know that the wheel conditions involving the specialized variables produce the factor $G_{\alpha}$ by the induction
assumption. As we specialize $x^{(\beta,1)}_{n-1,2}$, the corresponding wheel conditions
\begin{equation*}
  x^{(\beta,1)}_{n-1,2}=v^{2}x^{(\beta,2)}_{n-1,1}=vx^{(\beta,2)}_{n-2,1}, \qquad
  x^{(\beta,1)}_{n-1,1}=v^{2}x^{(\beta,2)}_{n-1,1}=v^{4}x^{(\beta,1)}_{n-1,2}=v^{2}x^{(\beta,1)}_{n,1}
\end{equation*}
contribute the new factors $w_{\beta,1}-v^{ 6}w_{\beta,2}$, $w_{\beta,1}-v^{ 2}w_{\beta,2}$ to $\phi_{\unl{d}}(F)$.
By symmetry, as we specialize the variable $x^{(\beta,2)}_{n-1,2}$, we also get new extra factors
$w_{\beta,2}-v^{6}w_{\beta,1}$ and $w_{\beta,2}-v^{2}w_{\beta,1}$. Thus $\phi_{\unl{d}}(F)$ is divisible by
$(w_{\beta,1}-v^{\pm 2}w_{\beta,2})(w_{\beta,1}-v^{\pm 6}w_{\beta,2})\cdot G_{\alpha}$, hence by $G_{\beta}$.

If $2\leq j\leq n-2$, then
\begin{equation*}
  G_{\beta}= {w_{\beta,1}^3w_{\beta,2}^3}
  (w_{\beta,1}-v^{\pm 2}w_{\beta,2})(w_{\beta,1}-v^{\pm(2n-2j)}w_{\beta,2})(w_{\beta,1}-v^{\pm(2n-2j+4)}w_{\beta,2})\cdot G_{\alpha}
\end{equation*}
with $\alpha=[1,n,j+1]$. For any $F\in S_{\unl{k}}$, as we specialize all the variables but
$\{x^{(\beta,1)}_{j,2},x^{(\beta,2)}_{j,2}\}$ we know that the wheel conditions involving the specialized
variables produce the factor $G_{\alpha}$ by the induction assumption. When we specialize $x^{(\beta,1)}_{j,2}$,
the new wheel conditions
\begin{equation*}
  x^{(\beta,1)}_{j,2}=v^{2}x^{(\beta,2)}_{j,1}=vx^{(\beta,2)}_{j-1,1},\
  x^{(\beta,2)}_{j,1}=v^{2}x^{(\beta,1)}_{j,2}=vx^{(\beta,2)}_{j+1,1},\
  x^{(\beta,1)}_{j+1,2}=v^{2}x^{(\beta,2)}_{j+1,2}=vx^{(\beta,1)}_{j,2}
\end{equation*}
contribute the factors $w_{\beta,1}-v^{2n-2j+4}w_{\beta,2}$, $w_{\beta,1}-v^{2n-2j}w_{\beta,2}$,
$w_{\beta,1}-v^{ 2}w_{\beta,2}$, respectively, into $\phi_{\unl{d}}(F)$. Then from symmetry (using $x^{(\beta,2)}_{j,2}$
instead of $x^{(\beta,1)}_{j,2}$), we see that $\phi_{\unl{d}}(F)$ is indeed divisible by $G_{\beta}$.

\item $\beta=\beta'=[1,n,1]$.

If $\alpha=[2,n,2]$, then we have
\begin{equation*}
  G_{\beta} = {w_{\beta,1}^2 w_{\beta,2}^2}(w_{\beta,1}-w_{\beta,2})^{2}
  (w_{\beta,1}-v^{\pm 2}w_{\beta,2})^{2}(w_{\beta,1}-v^{\pm 4}w_{\beta,2})\cdot G_{\alpha}.
\end{equation*}
For any $F\in S_{\unl{k}}$, as we specialize all the variables but
$\{x^{(\beta,1)}_{1,1},x^{(\beta,1)}_{1,2},x^{(\beta,2)}_{1,1},x^{(\beta,2)}_{1,2}\}$,
we know that the wheel conditions involving the specialized variables produce the factor $G_{\alpha}$ by the
induction assumption. As we specialize the variables $x^{(\beta,1)}_{1,1},x^{(\beta,1)}_{1,2}$, the wheel conditions at
\[
  x^{(\beta,1)}_{2,2}=v^{2}x^{(\beta,1)}_{2,1}=vx^{(\beta,1)}_{1,1}, \qquad
  x^{(\beta,1)}_{2,1}=v^{2}x^{(\beta,1)}_{2,2}=vx^{(\beta,1)}_{1,2}
\]
contribute the factor $B_{\beta}/B_{\alpha}= (w_{\beta,1}-v^{\pm 2}w'_{\beta,1})$ to the first step of the specialization
$\phi^{(1)}_{\beta}(F)$, cf.~\eqref{spe-C-2}. Then in the second step of the specialization, cf.~\eqref{spe-C-3}, we divide
by $B_{\beta}/B_{\alpha}$ and specialize $w'_{\beta,1}\mapsto w_{\beta,1}, w'_{\beta,2}\mapsto w_{\beta,2}$.
The  wheel conditions at
\begin{equation*}
\begin{aligned}
  x^{(\beta,2)}_{2,1}=v^{2}x^{(\beta,1)}_{2,1}=vx^{(\beta,1)}_{1,1}, \qquad
    x^{(\beta,2)}_{2,2}=v^{2}x^{(\beta,1)}_{2,1}=vx^{(\beta,1)}_{1,1},\\
  x^{(\beta,2)}_{2,1}=v^{2}x^{(\beta,1)}_{2,2}=vx^{(\beta,1)}_{1,2}, \qquad
    x^{(\beta,2)}_{2,2}=v^{2}x^{(\beta,1)}_{2,2}=vx^{(\beta,1)}_{1,2}
\end{aligned}
\end{equation*}
contribute the overall factor
$(w_{\beta,1}-w_{\beta,2})(w_{\beta,1}-v^{ -2}w_{\beta,2})^{2}(w_{\beta,1}-v^{-4}w_{\beta,2})$ to $\phi_{\unl{d}}(F)$.
Then from symmetry (using $x^{(\beta,2)}_{1,1},x^{(\beta,2)}_{1,2}$ instead of $x^{(\beta,1)}_{1,1},x^{(\beta,1)}_{1,2}$),
we see that $\phi_{\unl{d}}(F)$ is indeed divisible by
  $(w_{\beta,1}-w_{\beta,2})^{2}(w_{\beta,1}-v^{\pm 2}w_{\beta,2})^{2}(w_{\beta,1}-v^{\pm 4}w_{\beta,2})\cdot G_{\alpha}$,
hence by $G_{\beta}$.

\item $\beta=[1,i]$, $\beta'=[1,n,1]$.

If $i=1$, that is $\beta=[1]$, then
\begin{equation*}
  G_{\beta,\beta'}=(w_{\beta,1}-w_{\beta',1})(w_{\beta,1}-v^{-2}w_{\beta',1}).
\end{equation*}
The wheel conditions $F=0$ at $x^{(\beta',1)}_{1,1}=v^{2}x^{(\beta,1)}_{1,1}=vx^{(\beta',1)}_{2,1}$ and
$x^{(\beta',1)}_{1,2}=v^{2}x^{(\beta,1)}_{1,1}=vx^{(\beta',1)}_{2,2}$ imply that $\phi_{\unl{d}}(F)$ is
divisible by $G_{\beta,\beta'}$.

If $2\leq i\leq n-2$, then
\begin{equation*}
  G_{\beta,\beta'}= (w_{\beta,1}-w_{\beta',1})(w_{\beta,1}-v^{\pm 2}w_{\beta',1})(w_{\beta,1}-v^{4}w_{\beta',1})\cdot G_{\alpha,\beta'}
  \quad \mathrm{with} \quad \alpha=[1,i-1].
\end{equation*}
As we specialize all the variables but $x^{(\beta,1)}_{i,1}$, we know that the wheel conditions involving the specialized
variables produce the factor $G_{\alpha,\beta'}$ by the induction assumption. As we specialize $x^{(\beta,1)}_{i,1}$, the
wheel conditions $F=0$ at
\begin{equation}
\label{eq:betabeta'-1}
\begin{aligned}
  & x^{(\beta,1)}_{i,1}=v^{2}x^{(\beta',1)}_{i,1}=vx^{(\beta',1)}_{i-1,1}, \qquad
    x^{(\beta,1)}_{i,1}=v^{2}x^{(\beta',1)}_{i,2}=vx^{(\beta',1)}_{i-1,2}, \\
  & vx^{(\beta',1)}_{i+1,2}=x^{(\beta',1)}_{i,2}=v^2x^{(\beta,1)}_{i,1}, \qquad
    vx^{(\beta',1)}_{i+1,1}=x^{(\beta',1)}_{i,1}=v^2x^{(\beta,1)}_{i,1}
\end{aligned}
\end{equation}
contribute the factor $(w_{\beta,1}-w_{\beta',1})(w_{\beta,1}-v^{\pm 2}w_{\beta',1})(w_{\beta,1}-v^{4}w_{\beta',1})$
to $\phi_{\unl{d}}(F)$, and so $\phi_{\unl{d}}(F)$ is divisible by $G_{\beta,\beta'}$.

If $i=n-1$, that is $\beta=[1,n-1]$, then
\begin{equation*}
  G_{\beta,\beta'}=(w_{\beta,1}-v^{\pm 2}w_{\beta',1})(w_{\beta,1}-v^{4}w_{\beta',1})\cdot G_{\alpha,\beta'}
  \qquad \mathrm{with} \quad \alpha=[1,n-2].
\end{equation*}
Then the first three wheel conditions from \eqref{eq:betabeta'-1} imply that $\phi_{\unl{d}}(F)$ is divisible by
$(w_{\beta,1}-v^{\pm 2}w_{\beta',1})(w_{\beta,1}-v^{4}w_{\beta',1})$, hence by $G_{\beta,\beta'}$.

\item $\beta=[1,i]$,  $\beta'=[1,n,j]$.

If $i\leq j-2$, then $G_{\beta,\beta'}=G_{\beta,[1,j-1]}$, and so $\phi_{\unl{d}}(F)$ is divisible by $G_{\beta,\beta'}$
from type $A_n$.

If $i=j-1$, then
\begin{equation*}
  G_{\beta,\beta'}=(w_{\beta,1}-v^{\pm 2}w_{\beta',1})(w_{\beta,1}-v^{-2n+2j-2}w_{\beta',1})\cdot G_{\alpha,\beta'}
  \qquad \mathrm{with} \quad  \alpha=[1,j-2].
\end{equation*}
As we specialize all the variables but $x^{(\beta,1)}_{j-1,1}$, we know that the wheel conditions involving the
specialized variables produce the factor $G_{\alpha,\beta'}$ by the induction assumption. As we specialize
$x^{(\beta,1)}_{j-1,1}$, the wheel conditions $F=0$ at
$x^{(\beta',1)}_{j-1,1}=v^{2}x^{(\beta,1)}_{j-1,1}=vx^{(\beta',1)}_{j,1}$,
$x^{(\beta,1)}_{j-1,1}=v^{2}x^{(\beta',1)}_{j-1,1}=vx^{(\beta',1)}_{j-2,1}$
contribute the extra factor $(w_{\beta,1}-v^{\pm 2}w_{\beta',1})$ into $\phi_{\unl{d}}(F)$. Moreover, if we consider
$\unl{d}'=\{d'_{[1,j]}=d'_{[1,n,j+1]}=1,\, \text{and}\ d'_{\gamma}=0\ \text{for other}\ \gamma\}$, then
$\unl{d}'<\unl{d}$ and $\phi_{\unl{d}'}(F)=0$ implies that $\phi_{\unl{d}}(F)$ is divisible by
$w_{\beta,1}-v^{-2n+2j-2}w_{\beta',1}$. Thus, $\phi_{\unl{d}}(F)$ is divisible by $G_{\beta,\beta'}$, as claimed.

If $i=j<n-1$, then
\begin{equation*}
\begin{aligned}
  G_{\beta,\beta'}=(w_{\beta,1}-v^{-2n+2j-4}w_{\beta',1})\cdot G_{\beta,\alpha}
  \qquad \mathrm{with} \quad \alpha=[1,n,j+1].
\end{aligned}
\end{equation*}
As we specialize all the variables but $x^{(\beta',1)}_{j,2}$, we get the factor $G_{\beta,\alpha}$ by the induction
assumption. As we specialize $x^{(\beta',1)}_{j,2}$, the wheel condition $F=0$ at
$x^{(\beta',1)}_{j,2}=v^{2}x^{(\beta,1)}_{j,1}=vx^{(\beta,1)}_{j-1,1}$ implies $\phi_{\unl{d}}(F)$ is divisible by
$w_{\beta,1}-v^{-2n+2j-4}w_{\beta',1}$, hence by $G_{\beta,\beta'}$.

If $i=j=n-1$, then $G_{\beta,\beta'}=(w_{\beta,1}-v^{-6}w_{\beta',1})\cdot G_{\beta,[1,n]}$. From the wheel condition
$F=0$ at $x^{(\beta',1)}_{n-1,2}=v^{2}x^{(\beta,1)}_{n-1,1}=vx^{(\beta,1)}_{n-2,1}$ and the induction assumption,  we get
that $\phi_{\unl{d}}(F)$ is divisible by $G_{\beta,\beta'}$.

If $i=j+1=n$, then $G_{\beta,\beta'}=(w_{\beta,1}-v^{\pm 4}w_{\beta',1})\cdot G_{[1,n-1],\beta'}$.
Due to the induction assumption and the wheel conditions at
$x^{(\beta,1)}_{n,1}=v^{4}x^{(\beta',1)}_{n,1}=v^{2}x^{(\beta',1)}_{n-1,1}$ and
$x^{(\beta',1)}_{n,1}=v^{4}x^{(\beta,1)}_{n,1}=v^{2}x^{(\beta',1)}_{n-1,2}$, we see
that $\phi_{\unl{d}}(F)$ is divisible by $G_{\beta,\beta'}$.

If $i\geq j+1$ and $1<j<n-1$, then
  $G_{\beta,\beta'}=(w_{\beta,1}-v^{-2n+2j-4}w_{\beta',1})(w_{\beta,1}-v^{-2n+2j}w_{\beta',1})\cdot G_{\beta,[1,n,j+1]}$.
By the induction assumption and the wheel condition at
$x^{(\beta,1)}_{j,1}=v^{2}x^{(\beta',1)}_{j,2}=vx^{(\beta,1)}_{j+1,1}$ or
$x^{(\beta',1)}_{j,2}=v^{2}x^{(\beta,1)}_{j,1}=vx^{(\beta,1)}_{j-1,1}$, we see that
$\phi_{\unl{d}}(F)$ is divisible by $G_{\beta,\beta'}$.

\item $\beta=[1,n,1]$, $\beta'=[1,n]$.

If we set $\alpha=[2,n,2]$, $\alpha'=[2,n]$, then we have:
\begin{equation*}
  G_{\beta,\beta'} =
  (w_{\beta,1}-w_{\beta',1})(w_{\beta,1}-v^{\pm 2}w_{\beta',1})(w_{\beta,1}-v^{-4}w_{\beta',1})\cdot G_{\alpha,\alpha'}.
\end{equation*}
From the wheel conditions at
\begin{equation*}
\begin{aligned}
  & x^{(\beta,1)}_{1,1}=v^{2}x^{(\beta',1)}_{1,1}=vx^{(\beta,1)}_{2,1}, \qquad
    x^{(\beta,1)}_{1,2}=v^{2}x^{(\beta',1)}_{1,1}=vx^{(\beta,1)}_{2,2}, \\
  & vx^{(\beta,1)}_{1,1}=x^{(\beta',1)}_{2,1}=v^2x^{(\beta,1)}_{2,1}, \qquad
    vx^{(\beta,1)}_{1,2}=x^{(\beta',1)}_{2,1}=v^2x^{(\beta,1)}_{2,2}
\end{aligned}
\end{equation*}
and the induction assumption, we see that $\phi_{\unl{d}}(F)$ is divisible by $G_{\beta,\beta'}$.

\item $\beta=[1,n,1], \beta'=[1,n,j]$.

If $j>2$, then the same arguments as for the case $(\beta,\beta')=([1,n,1],[1,n])$ above apply.

If $j=2$ and $n=3$, then we have
\begin{equation*}
  G_{\beta,\beta'} =
  (w_{\beta,1}-v^{-2}w_{\beta',1})(w_{\beta,1}-v^{-6}w_{\beta',1})(w_{\beta,1}-v^{-8}w_{\beta',1})\cdot G_{[1,3,1],[1,3]}.
\end{equation*}
From the wheel conditions at
\begin{equation*}
\begin{aligned}
  & x^{(\beta',1)}_{2,2}=v^{2}x^{(\beta,1)}_{2,1}=vx^{(\beta,1)}_{1,1}, \qquad
    x^{(\beta',1)}_{2,2}=v^{2}x^{(\beta,1)}_{2,2}=vx^{(\beta,1)}_{1,2}, \\
  & x^{(\beta,1)}_{2,2}=v^2x^{(\beta,1)}_{2,1}=v^4x^{(\beta',1)}_{2,2}=v^2x^{(\beta',1)}_{3,1},
\end{aligned}
\end{equation*}
and the induction assumption, we see that $\phi_{\unl{d}}(F)$ is divisible by $G_{\beta,\beta'}$.

If $j=2$ and $n>3$, then we have
\begin{equation*}
  G_{\beta,\beta'} =
  (w_{\beta,1}-v^{-2n}w_{\beta',1})(v^{2n}w_{\beta,1}-v^{\pm 2}w_{\beta',1})(w_{\beta,1}-v^{-2n+4}w_{\beta',1})\cdot G_{\beta,[1,n,3]}.
\end{equation*}
From the wheel conditions at
\begin{equation*}
\begin{aligned}
  & x^{(\beta',1)}_{2,2}=v^{2}x^{(\beta,1)}_{2,1}=vx^{(\beta,1)}_{1,1}, \qquad
    x^{(\beta',1)}_{2,2}=v^{2}x^{(\beta,1)}_{2,2}=vx^{(\beta,1)}_{1,2}, \\
  & v^2 x^{(\beta',1)}_{2,2}=x^{(\beta,1)}_{2,1}=vx^{(\beta,1)}_{3,1}, \qquad
   v^2x^{(\beta',1)}_{2,2}=x^{(\beta,1)}_{2,2}=vx^{(\beta,1)}_{3,2},
\end{aligned}
\end{equation*}
and the induction assumption, we see that $\phi_{\unl{d}}(F)$ is divisible by $G_{\beta,\beta'}$.

\item $\beta=[1,n,k]$, $\beta'=[1,n,j]$.

If $j>2$, then $G_{\beta,\beta'}=(w_{\beta,1}-v^{\pm 2}w_{\beta',1})\cdot G_{[2,n,k],[2,n,j]}$, and so
$\phi_{\unl{d}}(F)$ is divisible by $G_{\beta,\beta'}$ due to the induction assumption and wheel condition
at $x^{(\beta,1)}_{2,1}=v^{2}x^{(\beta',1)}_{2,1}=vx^{(\beta',1)}_{1,1}$ or
$x^{(\beta',1)}_{2,1}=v^{2}x^{(\beta,1)}_{2,1}=vx^{(\beta,1)}_{1,1}$.

If $j=2$ and $k>3$, then
\[
  G_{\beta,\beta'}=(w_{\beta,1}-v^{-2n}w_{\beta',1})(w_{\beta,1}-v^{-2n+4}w_{\beta',1})\cdot G_{[1,n,k],[1,n,3]},
\]
and so $\phi_{\unl{d}}(F)$ is divisible by $G_{\beta,\beta'}$ due to the induction assumption and wheel condition at
$x^{(\beta,1)}_{2,1}=v^{2}x^{(\beta',1)}_{2,2}=vx^{(\beta,1)}_{3,1}$ or
$x^{(\beta',1)}_{2,2}=v^{2}x^{(\beta,1)}_{2,1}=vx^{(\beta,1)}_{1,1}$.

If $j=2, k=3$ and $n>4$, then
\begin{equation*}
  G_{\beta,\beta'} =
  (w_{\beta,1}-v^{\pm 2}w_{\beta',1})(w_{\beta,1}-v^{2n- 2}w_{\beta',1})(w_{\beta,1}-v^{2n-6}w_{\beta',1})\cdot G_{[1,n,4],[1,n,2]}.
\end{equation*}
From wheel conditions at
\begin{equation*}
\begin{aligned}
  & x^{(\beta,1)}_{3,2}=v^{2}x^{(\beta',1)}_{3,1}=vx^{(\beta',1)}_{2,1}, \qquad
    x^{(\beta,1)}_{3,2}=v^{2}x^{(\beta',1)}_{3,2}=vx^{(\beta',1)}_{4,2}, \\
  & v^2 x^{(\beta,1)}_{3,2}=x^{(\beta',1)}_{3,1}=vx^{(\beta',1)}_{4,1}, \qquad
   v^2x^{(\beta,1)}_{3,2}=x^{(\beta',1)}_{3,2}=vx^{(\beta',1)}_{2,2},
\end{aligned}
\end{equation*}
and the induction assumption, we see that $\phi_{\unl{d}}(F)$ is divisible by $G_{\beta,\beta'}$.

If $j=2, k=3$ and $n=4$, then
$G_{\beta,\beta'} = (w_{\beta,1}-v^{\pm 2}w_{\beta',1})(w_{\beta,1}-v^{6}w_{\beta',1})\cdot G_{[1,4],[1,4,2]}$,
and $\phi_{\unl{d}}(F)$ is divisible by $G_{\beta,\beta'}$, due to the induction assumption and wheel conditions at
$x^{(\beta,1)}_{3,2}=v^{2}x^{(\beta',1)}_{3,1}=vx^{(\beta',1)}_{2,1}$,
$x^{(\beta',1)}_{3,2}=v^{2}x^{(\beta,1)}_{3,2}=vx^{(\beta',1)}_{2,2}$,
$x^{(\beta',1)}_{3,1}=v^2x^{(\beta,1)}_{3,2}=v^4x^{(\beta',1)}_{3,2}=v^2x^{(\beta',1)}_{4,1}$.

\item $\beta'\geq [2]>\beta$.

If $\beta=[1]$ and $\beta'=[2,n,2]$, then $G_{\beta,\beta'}=(w_{\beta,1}-w_{\beta',1})(w_{\beta,1}-v^{2}w_{\beta',1})$.
Consider $\unl{d}'=\{d'_{[1,n]}=d'_{[2,n-1]}=1,\, \text{and}\ d'_{\gamma}=0\ \text{for other}\ \gamma\}$, so that $\unl{d}'<\unl{d}$.
Then $\phi_{\unl{d}}(F)$ is divisible by $G_{\beta,\beta'}$ due to the condition $\phi_{\unl{d}'}(F)=0$ and wheel condition
at $x^{(\beta',1)}_{2,2}=v^{2}x^{(\beta',1)}_{2,1}=vx^{(\beta,1)}_{1,1}$.

If $\beta=[1,i]$ and $\beta'=[2,n,j]$, then $G_{\beta,\beta'}=(w_{\beta,1}-w_{\beta',1})\cdot G_{[2,i],\beta'}$.
Consider $\unl{d}'=\{d'_{[1,n,j]}=d'_{[2,i]}=1,\, \text{and}\ d'_{\gamma}=0\ \text{for other}\ \gamma\}$.
Then $\unl{d}'<\unl{d}$ and $\phi_{\unl{d}}(F)$ is divisible by $G_{\beta,\beta'}$ due to the induction assumption
and $\phi_{\unl{d}'}(F)=0$.

If $\beta=[1,n,i]$ and $\beta'=[2,n,j]$ with $j<i$, then $G_{\beta,\beta'}=(w_{\beta,1}-w_{\beta',1})\cdot G_{[2,n,i],\beta'}$.
Consider $\unl{d}'=\{d'_{[1,n,j]}=d'_{[2,n,i]}=1,\, \text{and}\ d'_{\gamma}=0\ \text{for other}\ \gamma\}$. Then $\unl{d}'<\unl{d}$
and $\phi_{\unl{d}}(F)$ is divisible by $G_{\beta,\beta'}$ due to the induction assumption and $\phi_{\unl{d}'}(F)=0$.

For all other cases, the divisibility of $\phi_{\unl{d}}(F)$ by $G_{\beta,\beta'}$ follows from the
induction assumption and proper count of wheel conditions similarly to the cases above.

\end{itemize}
This completes our proof.
\end{proof}

Combining Propositions~\ref{vanishC} and~\ref{spanC}, we immediately obtain the shuffle algebra realization
and the PBWD theorem for $\qlc$:

\begin{Thm}\label{shufflePBWD-C}
(a) $\Psi\colon \qlc \,\iso\, S$ of~\eqref{eq:Psi-homom} is a $\BQ(v)$-algebra isomorphism.

\noindent
(b) For any choices of $s_k$ and $\lambda_k$ in the definition~\eqref{rootvector1}--\eqref{rootvector2} of quantum root vectors
$E_{\beta,s}$, the ordered PBWD monomials $\{E_{h}\}_{h\in H}$ from \eqref{PBWDbases} form a $\BQ(v)$-basis of $\qlc$.
\end{Thm}


\subsection{Shuffle algebra realization of the Lusztig integral form in type $C$}\label{lusc}

For any $\unl{k}\in \BN^n$, consider the $\BZ[v,v^{-1}]$-submodule $\mathbf{S}_{\unl{k}}$ of $S_{\unl{k}}$
consisting of rational functions $F$ satisfying the following two conditions:
\begin{enumerate}[leftmargin=1cm]

\item
If $f$ denotes the numerator of $F$ from \eqref{polecondition}, then
\begin{equation}
\label{li1-C}
  f\in \BZ[v,v^{-1}][\{x_{i,r}^{\pm 1}\}_{1\leq i\leq n}^{1\leq r\leq k_{i}}]^{\mathfrak{S}_{\underline{k}}}.
\end{equation}

\item
For any $\unl{d}\in\text{KP}(\underline{k})$, the specialization $\phi_{\unl{d}}(F)$ is divisible by the product
\begin{equation}
\label{li2-C}
  \prod_{\beta\in\Delta^{+}}{\tilde{c}_{\beta}}^{d_{\beta}},
\end{equation}
where we define $\{\tilde{c}_{\beta}\}_{\beta\in\Delta^{+}}$ via $\{c_\beta\}_{\beta\in \Delta^+}$ of \eqref{eq:c-factor-C}:
\begin{equation}
\label{ctildebeta-C}
  \tilde{c}_{\beta}=
  \begin{cases}
    \frac{c_{\beta}}{[2]_{v}} & \quad \text{if}\ \beta=[i,n,i]\ \text{with}\ 1\leq i\leq n-1 \\
    c_{\beta} & \quad \text{otherwise}
  \end{cases}.
\end{equation}
\end{enumerate}

Define $\mathbf{S}\coloneqq \bigoplus_{\unl{k}\in\BN^{n}}\mathbf{S}_{\unl{k}}$ and recall the Lusztig integral
form $\integralcl$ from Definition \ref{def:lusintegral}. Then, similarly to~\cite[Proposition 4.17]{HT24}, we have:

\begin{Prop}\label{lintegralC}
$\Psi(\integralcl) \subset \mathbf{S}$.
\end{Prop}

\begin{proof}
For any $m\in\BN$, $1\leq i_{1},\dots,i_{m}\leq n$, $r_{1},\dots,r_{m}\in\BZ$, and $\ell_{1},\dots,\ell_{m}\in\BN$, let
\[
  F\coloneqq \Psi\big(\mathbf{E}^{(\ell_{1})}_{i_{1},r_{1}}\cdots \mathbf{E}^{(\ell_{m})}_{i_{m},r_{m}}\big),
\]
and $f$ be the numerator of $F$ from \eqref{polecondition}. The validity of the condition \eqref{li1-C} for $f$
follows from the equality of~\cite[Lemma 1.3]{Tsy22}:
\begin{equation}
\label{rank1-power}
  \Psi(\mathbf{E}^{(\ell_{q})}_{i_{q},r_{q}})=
  v_{i_{q}}^{-\frac{\ell_{q}(\ell_{q}-1)}{2}}(x_{i_{q},1}\cdots x_{i_{q},\ell_{q}})^{r_{q}}
  \qquad \forall \, 1\leq q\leq m.
\end{equation}
To verify the validity of the divisibility \eqref{li2-C}, it suffices to show that for any $\beta\in \Delta^+$ and
$1\leq s\leq d_{\beta}$, the total contribution of $\phi_{\unl{d}}$-specializations of the $\zeta$-factors between
the variables $\{x^{(\beta,s)}_{i,t}\}_{i\in\beta}^{1\leq t\leq \nu_{\beta,i}}$ of $f$ is divisible by $\tilde{c}_{\beta}$.
It suffices to treat only the cases $\beta=[i,n,j]$ with $1\leq i\leq j<n$, since the cases when $\beta=[i,j]$ are treated
completely analogously to type $A_n$. Henceforth, we write $o(x^{(*,*)}_{*,*})=q$ if a variable $x^{(*,*)}_{*,*}$ is plugged
into $\Psi(\mathbf{E}^{(\ell_{q})}_{i_{q},r_{q}})$. We consider the cases $i\ne j$ and $i=j$ separately:
\begin{itemize}[leftmargin=0.7cm]

\item $\beta=[i,n,j]$ with $1\leq i<j<n$.
\smallskip

According to~\eqref{spe-C-1}, the $\phi_{\unl{d}}$-specialization of each summand in $F$  vanishes unless
\begin{gather*}
  o(x^{(\beta,s)}_{i,1})\geq  \cdots \geq o(x^{(\beta,s)}_{n-1,1})\geq o(x^{(\beta,s)}_{n,1}) \geq
  o(x^{(\beta,s)}_{n-1,2})\geq \cdots \geq o(x^{(\beta,s)}_{j,2}).
\end{gather*}
Since $o(x^{(\beta,s)}_{i,t})\ne o(x^{(\beta,s)}_{i',t'})$ for $i\ne i'$, we actually have:
\begin{gather*}
  o(x^{(\beta,s)}_{i,1})> \cdots >o(x^{(\beta,s)}_{n-1,1})> o(x^{(\beta,s)}_{n,1}) > o(x^{(\beta,s)}_{n-1,2})>
  \cdots >o(x^{(\beta,s)}_{j,2}).
\end{gather*}
The $\phi_{\unl{d}}$-specialization of the product of the following $\zeta$-factors
\[
  \left\{ \prod^{n-2}_{\ell=j} \zeta\left(\frac{x^{(\beta,s)}_{\ell,2}}{x^{(\beta,s)}_{\ell+1,2}}\right) \right\}
  \cdot \zeta\left(\frac{x^{(\beta,s)}_{n,1}}{x^{(\beta,s)}_{n-1,1}}\right) \cdot
  \left\{\prod_{\ell=i+1}^{n-1} \zeta\left(\frac{x^{(\beta,s)}_{\ell,1}}{x^{(\beta,s)}_{\ell-1,1}}\right)\right\},
\]
contributes $\langle 1\rangle^{2n-i-j-2}_{v}\langle 2\rangle_{v}$. Likewise, the $\phi_{\unl{d}}$-specialization of
\[
  \prod^{n-1}_{\ell=j}\left\{
  \zeta\left(\frac{x^{(\beta,s)}_{\ell,2}}{x^{(\beta,s)}_{\ell-1,1}}\right)
  \zeta\left(\frac{x^{(\beta,s)}_{\ell,2}}{x^{(\beta,s)}_{\ell,1}}\right)
  \zeta\left(\frac{x^{(\beta,s)}_{\ell,2}}{x^{(\beta,s)}_{\ell+1,1}}\right)
  \right\}
\]
contributes $\prod_{\ell=j}^{n-1}\left\{(v^{2n-2\ell}-1)(v^{2n-2\ell+4}-1)\right\}$.
This overall yields $\tilde{c}_{[i,n,j]}$ of~\eqref{ctildebeta-C}.

\medskip
\item

$\beta=[i,n,i]$ with $1\leq i<n$.
\smallskip

According to~\eqref{spe-C-2}, the $\phi_{\unl{d}}$-specialization of each summand in $F$  vanishes unless
\begin{gather*}
  o(x^{(\beta,s)}_{i,1}) \geq o(x^{(\beta,s)}_{i+1,1})\geq \cdots  \geq o(x^{(\beta,s)}_{n-1,1}),\quad
  o(x^{(\beta,s)}_{i,2})\geq \cdots \geq o(x^{(\beta,s)}_{n-1,2})\geq o(x^{(\beta,s)}_{n,1}).
\end{gather*}
Since $o(x^{(\beta,s)}_{i,t})\ne o(x^{(\beta,s)}_{i',t'})$ for $i\ne i'$, we again have strict inequalities:
\begin{gather*}
  o(x^{(\beta,s)}_{i,1})> o(x^{(\beta,s)}_{i+1,1})> \cdots > o(x^{(\beta,s)}_{n-1,1}),\quad
  o(x^{(\beta,s)}_{i,2})> \cdots > o(x^{(\beta,s)}_{n-1,2})>o(x^{(\beta,s)}_{n,1}).
\end{gather*}
For any $i\leq \ell\leq n-2$, let us consider the $\zeta$-factors between the variables
\[
  \big\{x^{(\beta,s)}_{\ell,1},x^{(\beta,s)}_{\ell,2},x^{(\beta,s)}_{\ell+1,1},x^{(\beta,s)}_{\ell+1,2}\big\}.
\]
With symmetry in the above variables, we may assume that $o(x^{(\beta,s)}_{\ell,1})\geq o(x^{(\beta,s)}_{\ell,2})$
in the following analysis. We have two cases to consider:
\smallskip

\begin{itemize}[leftmargin=0.7cm]

\item
if $o(x^{(\beta,s)}_{\ell,2})>o(x^{(\beta,s)}_{\ell+1,1})$, then we have
$o(x^{(\beta,s)}_{\ell,1})\geq o(x^{(\beta,s)}_{\ell,2})>o(x^{(\beta,s)}_{\ell+1,1}) \ \&\ o(x^{(\beta,s)}_{\ell+1,2})$,
and
$
  \zeta\left(\frac{x^{(\beta,s)}_{\ell+1,2}}{x^{(\beta,s)}_{\ell,1}}\right)
  \zeta\left(\frac{x^{(\beta,s)}_{\ell+1,1}}{x^{(\beta,s)}_{\ell,2}}\right)
$
contributes $(w_{\beta,s}-v^{2}w'_{\beta,s})(w_{\beta,s}-v^{-2}w'_{\beta,s})$ into the $\phi^{(1)}_{\beta}$-specialization;

\item
if $o(x^{(\beta,s)}_{\ell+1,1})>o(x^{(\beta,s)}_{\ell,2})$, then
  $o(x^{(\beta,s)}_{\ell,1})>o(x^{(\beta,s)}_{\ell+1,1})>o(x^{(\beta,s)}_{\ell,2})> o(x^{(\beta,s)}_{\ell+1,2})$,
and
$
  \zeta\left(\frac{x^{(\beta,s)}_{\ell+1,2}}{x^{(\beta,s)}_{\ell,1}}\right)
  \zeta\left(\frac{x^{(\beta,s)}_{\ell+1,2}}{x^{(\beta,s)}_{\ell+1,1}}\right)
  \zeta\left(\frac{x^{(\beta,s)}_{\ell,2}}{x^{(\beta,s)}_{\ell+1,1}}\right)
$
contributes $(w_{\beta,s}-v^{2}w'_{\beta,s})(w_{\beta,s}-v^{-2}w'_{\beta,s})$ into the $\phi^{(1)}_{\beta}$-specialization.

\end{itemize}
The above analysis shows that the $\phi^{(1)}_{\beta}$-specialization of that summand is divisible by $B_{\beta}$ of \eqref{Bfactor-C}.
Now let us  consider the $\zeta$-factors between the variables $\{x^{(\beta,s)}_{n-1,1},x^{(\beta,s)}_{n-1,2},x^{(\beta,s)}_{n,1}\}$.
If $o(x^{(\beta,s)}_{n,1})>o(x^{(\beta,s)}_{n-1,1})$, then the $\phi_{\unl{d}}$-specialization of that summand vanishes due to the
$\zeta$-factor $\zeta\left(\frac{x^{(\beta,s)}_{n-1,1}}{x^{(\beta,s)}_{n-1,2}}\right)$; if $o(x^{(\beta,s)}_{n,1})<o(x^{(\beta,s)}_{n-1,1})$,
then the $\zeta$-factors
  $\zeta\left(\frac{x^{(\beta,s)}_{n,1}}{x^{(\beta,s)}_{n-1,1}}\right)\zeta\left(\frac{x^{(\beta,s)}_{n,1}}{x^{(\beta,s)}_{n-1,2}}\right)$
contribute $\langle 1\rangle_{v}\langle 2\rangle_{v}$ into the overall $\phi_{\unl{d}}$-specialization. Along with the specialization
of the $\zeta$-factors (which have not been considered above yet)
$
  \prod^{n-2}_{\ell=i}\left\{\zeta\left(\frac{x^{(\beta,s)}_{\ell+1,1}}{x^{(\beta,s)}_{\ell,1}}\right)
  \zeta\left(\frac{x^{(\beta,s)}_{\ell+1,2}}{x^{(\beta,s)}_{\ell,2}}\right)\right\}
$
shows that $\phi_{\unl{d}}(F)$ is indeed divisible by $\langle 1\rangle^{|\beta|-2}_{v}\langle 2\rangle_{v}$,
which is precisely $\tilde{c}_{[i,n,i]}$ of \eqref{ctildebeta-C}.
\end{itemize}

This completes our proof.
\end{proof}

Recall the normalized divided powers~\eqref{eq:nrv-C} of the quantum root vectors
$\{\tilde{\mathbf{E}}_{\beta,s}^{\pm,(k)}\}^{k\in\BN}_{\beta\in\Delta^{+},s\in\BZ}$ and the ordered monomials
$\{\tilde{\mathbf{E}}^{\epsilon}_{h}\}_{h\in H}$ of  \eqref{eq:Lus-pbwd}. For $\epsilon\in\{\pm\}$,
let $\mathbf{S}^{\epsilon}_{\unl{k}}$ be the $\BZ[v,v^{-1}]$-submodule of $\mathbf{S}_{\unl{k}}$
spanned by $\{\Psi(\tilde{\mathbf{E}}^{\epsilon}_{h})\}_{h\in H_{\unl{k}}}$.
Then, the following analogue of Lemma \ref{span} holds:

\begin{Prop}\label{span-Lus-C}
For any $F\in \mathbf{S}_{\unl{k}}$ and $\unl{d}\in\mathrm{KP}(\unl{k})$, if $\phi_{\unl{d}'}(F)=0$ for all
$\unl{d}'\in \mathrm{KP}(\unl{k})$ such that $\unl{d}'<\unl{d}$, then there exists
$F_{\unl{d}}\in \mathbf{S}^{\epsilon}_{\unl{k}}$ such that $\phi_{\unl{d}}(F)=\phi_{\unl{d}}(F_{\unl{d}})$
and $\phi_{\unl{d}'}(F_{\unl{d}})=0$ for all $\unl{d}'<\unl{d}$.
\end{Prop}

\begin{proof}
Completely analogous to that of~\cite[Proposition~3.11]{HT24}.
\end{proof}

Combining Propositions \ref{lintegralC} and~\ref{span-Lus-C}, we obtain the following upgrade of Theorem \ref{shufflePBWD-C}:

\begin{Thm}\label{lusthm-C}
(a) The $\BQ(v)$-algebra isomorphism $\Psi\colon \qlc \,\iso\, S$ of Theorem~\ref{shufflePBWD-C}(a) gives rise to
a $\BZ[v,v^{-1}]$-algebra isomorphism $\Psi\colon \integralcl \,\iso\, \mathbf{S}$.

\noindent
(b) Theorem {\rm \ref{pbwtheorem-Lus}} holds for $\fg$ of type $C_n$.
\end{Thm}


\subsection{Shuffle algebra realization of the RTT integral form $\integralc$}\label{rttc}

To introduce the RTT integral form of the shuffle algebra $S$, we first recall the \textbf{vertical specialization map}
(cf.\ \cite[(1.59)]{Tsy22}):
\begin{equation}
\label{verticalspe}
  \varpi_{\unl{t}}\colon
  \BZ[v,v^{-1}][\{w_{\beta,s}^{\pm 1}\}_{\beta\in\Delta^{+}}^{1\leq s\leq d_{\beta}}]^{\mathfrak{S}_{\unl{d}}} \longrightarrow
  \BZ[v,v^{-1}][\{z_{\beta,r}^{\pm 1}\}_{\beta\in\Delta^{+}}^{1\leq r\leq \ell_{\beta}}].
\end{equation}
For $\unl{d}\in\text{KP}(\unl{k})$,  pick any collection of positive integers
$\unl{t}=\{t_{\beta,r}\}_{\beta\in\Delta^{+}}^{1\leq r\leq \ell_{\beta}}\ (\ell_{\beta}\in\BN)$ satisfying
\begin{equation}
\label{verticalpartition-1}
  d_{\beta}=\sum_{r=1}^{\ell_{\beta}}t_{\beta,r} \qquad \forall\, \beta\in\Delta^{+}.
\end{equation}
For any $\beta\in\Delta^{+}$, we split the variables $\{w_{\beta,s}\}_{s=1}^{d_{\beta}}$ into $\ell_{\beta}$ groups
of size $t_{\beta,r}$ each $(1\leq r\leq \ell_{\beta})$ and specialize the variables in the $r$-th group to
\begin{equation*}
  v_{\beta}^{-2}z_{\beta,r},\ v_{\beta}^{-4}z_{\beta,r},\ \dots\ ,\ v_{\beta}^{-2t_{\beta,r}}z_{\beta,r}.
\end{equation*}
For any
  $g\in \BZ[v,v^{-1}][\{w_{\beta,s}^{\pm 1}\}_{\beta\in\Delta^{+}}^{1\leq s\leq d_{\beta}}]^{\mathfrak{S}_{\unl{d}}}$,
we define $\varpi_{\unl{t}}(g)$ as the above specialization of~$g$.

Recall the factors $\{c_{\beta}\}_{\beta\in\Delta^{+}}$ of \eqref{eq:c-factor-C}.
When $\beta=[i,n,j]$ with $1\leq i<j<n$, we have
\begin{align*}
  c_{\beta}=\ &\langle 1\rangle^{\abs{\beta}-3}_{v} \langle 2\rangle_{v}\cdot
  \prod_{\ell=j}^{n-1}\left\{(v^{2n-2\ell}-1)(v^{2n-2\ell+4}-1)\right\} \\
  \doteq\ & \langle 1\rangle^{\abs{\beta}-2}_{v}\langle 2\rangle_{v}\cdot (v^{2n-2j+4}-1)\cdot
  \prod_{\ell=j}^{n-2}\left\{(v^{2n-2\ell}-1)(v^{2n-2\ell+2}-1)\right\}.
\end{align*}
For any $\unl{k}\in\BN^{n}$, consider the $\BZ[v,v^{-1}]$-submodule $\mathcal{S}_{\unl{k}}$ of $S_{\unl{k}}$
consisting of rational functions $F$ satisfying the following three conditions:
\begin{enumerate}[leftmargin=1cm]

\item
If $f$ denotes the numerator of $F$ from \eqref{polecondition}, then
\begin{equation}
\label{rttconstant1-C}
  f\in \langle 1\rangle_{v}^{k_1+\dots+k_{n-1}} \langle 2\rangle_{v}^{k_{n}}\cdot
  \BZ[v,v^{-1}][\{x_{i,r}^{\pm 1}\}_{1\leq i\leq n}^{1\leq r\leq k_{i}}]^{\mathfrak{S}_{\underline{k}}}.
\end{equation}

\item
For any $\unl{d}\in\text{KP}(\underline{k})$, the specialization
$\phi_{\unl{d}}(f\cdot \langle 1\rangle_{v}^{-k_1-\dots-k_{n-1}} \langle 2\rangle_{v}^{-k_{n}})$ is divisible by
\begin{equation}
\label{rttconstant2-C}
\begin{aligned}
  A_{\unl{d}}=
  \prod^{1\leq i<n}_{\beta=[i,n,i]\in \Delta^{+}}[2]^{d_{\beta}}_{v}
  \prod^{1\leq i<j<n}_{\beta=[i,n,j]\in \Delta^{+}}(v^{2n-2j+4}-1)^{d_{\beta}}\prod_{\ell=j}^{n-2}
  \big\{(v^{2n-2\ell}-1)^{d_{\beta}}(v^{2n-2\ell+2}-1)^{d_{\beta}}\big\}.
\end{aligned}
\end{equation}

\item
$F$ is \textbf{integral} in the sense of \cite[Definition 4.12]{HT24}: the {\em cross specialization}
\begin{equation}
\label{eq:crossspe-C}
  \Upsilon_{\unl{d},\unl{t}}(F) \coloneqq
  \varpi_{\unl{t}}\left(\frac{\phi_{\unl{d}}(F)}{\langle 1\rangle_{v}^{k_1+\dots+k_{n-1}} \langle 2\rangle_{v}^{k_{n}}\cdot
  A_{\unl{d}}\cdot \prod_{\beta\in\Delta^{+}}G_{\beta}}\right)
\end{equation}
is divisible by $\prod_{\beta\in\Delta^{+}}^{1\leq r\leq \ell_{\beta}}[t_{\beta,r}]_{v_{\beta}}!$ for any
$\unl{d}\in\mathrm{KP}(\unl{k})$ and $\unl{t}=\{t_{\beta,r}\}_{\beta\in\Delta^{+}}^{1\leq r\leq \ell_{\beta}}$ satisfying
\eqref{verticalpartition-1}, with $G_{\beta}$ of \eqref{eq:formulagbeta-C-1}--\eqref{eq:formulagbeta-C-4};
the divisibility of $\phi_{\unl{d}}(F)$ by $G_{\beta}$ is proved in Proposition~\ref{goodC}.
\end{enumerate}

We define $\mathcal{S}:=\bigoplus_{\unl{k}\in\BN^{n}}\mathcal{S}_{\unl{k}}$.
Recall the RTT integral form $\integralc$ from Definition~\ref{def:rttintegral}.
Then, similarly to~\cite[Proposition 4.13]{HT24}, we have:

\begin{Prop}\label{goodC}
$\Psi(\integralc) \subset \mathcal{S}$.
\end{Prop}

\begin{proof}
For any $\epsilon\in\{\pm\}$, $m\in\BN$, $\beta_{1},\dots,\beta_{m}\in\Delta^{+}$, $r_{1},\dots,r_{m}\in\BZ$, let
\[
  F\coloneqq
  \Psi\big(\tilde{\mathcal{E}}^{\epsilon}_{\beta_{1},r_{1}}\cdots \tilde{\mathcal{E}}^{\epsilon}_{\beta_{m},r_{m}}\big),
\]
and $f$ be the numerator of $F$. We set $\unl{k}=\sum_{q=1}^{m} \beta_{q}$. Henceforth, we shall use the notation
$o(x^{(*,*)}_{*,*})=q$ if a variable $x^{(*,*)}_{*,*}$ is plugged into
$\Psi(\tilde{\mathcal{E}}^{\epsilon}_{\beta_{q},r_{q}})$ for some $1\leq q\leq m$.

First, due to Lemma \ref{lem:psirv-C} and our choices of the normalized quantum root vectors of \eqref{eq:rtt-vectors-C},
$f$ is divisible by $\langle 1\rangle_{v}^{k_1+\dots+k_{n-1}} \langle 2\rangle_{v}^{k_{n}}$, thus implying~\eqref{rttconstant1-C}.

Next, for any $\unl{d}\in\text{KP}(\unl{k})$, we show that
$\phi_{\unl{d}}(f/\langle 1\rangle_{v}^{k_1+\dots+k_{n-1}} \langle 2\rangle_{v}^{k_{n}})$ is divisible by $A_{\unl{d}}$
of~\eqref{rttconstant2-C}. We consider the $\phi_{\unl{d}}$-specialization of each summand from the symmetrization featuring in $f$.
\begin{itemize}[leftmargin=0.7cm]

\item
$\beta=[i,n,j]$ with $1\leq i<j<n$ such that $d_{\beta}\neq 0$.

Fix any $1\leq s\leq d_{\beta}$. We can assume that
\begin{gather*}
  o(x^{(\beta,s)}_{i,1})\geq  \cdots \geq o(x^{(\beta,s)}_{n-1,1})\geq o(x^{(\beta,s)}_{n,1}) \geq
  o(x^{(\beta,s)}_{n-1,2})\geq \cdots \geq o(x^{(\beta,s)}_{j,2}),
\end{gather*}
as otherwise the $\phi_{\unl{d}}$-specialization of the corresponding summand vanishes. Let us now consider the
$\zeta$-factors arising from the variables $\{x^{(\beta,s)}_{j-1,1},x^{(\beta,s)}_{j,1},x^{(\beta,s)}_{j,2}\}$:
\begin{itemize}[leftmargin=0.5cm]

\item
If $o(x^{(\beta,s)}_{j-1,1})=o(x^{(\beta,s)}_{j,2})$, then $o(x^{(\beta,s)}_{j-1,1})=o(x^{(\beta,s)}_{j,1})=o(x^{(\beta,s)}_{j,2})$,
and from Lemma \ref{lem:psirv-C} we know that the corresponding summand is divisible by
\begin{equation}
\label{eq:otherfactors}
  (1+v^2)x^{(\beta,s)}_{j,1}x^{(\beta,s)}_{j,2}-vx^{(\beta,s)}_{j-1,1}(x^{(\beta,s)}_{j,1}+x^{(\beta,s)}_{j,2}) \quad \text{or} \quad
  (1+v^2)x^{(\beta,s)}_{j-1,1}-v(x^{(\beta,s)}_{j,1}+x^{(\beta,s)}_{j,2}),
\end{equation}
and so the $\phi_{\unl{d}}$-specialization is divisible by $v^{2n-2j+4}-1$.

\item
If $o(x^{(\beta,s)}_{j-1,1})>o(x^{(\beta,s)}_{j,2})$, then from the $\zeta$-factor
$\zeta\left(\frac{x^{(\beta,s)}_{j,2}}{x^{(\beta,s)}_{j-1,1}}\right)$ we know that the $\phi_{\unl{d}}$-specialization
of the corresponding summand is divisible by $v^{2n-2j+4}-1$.

\end{itemize}
Next, for each $j\leq \ell\leq n-2$, let us consider the $\zeta$-factors arising from the variables
\begin{equation}
  \big\{ x^{(\beta,s)}_{\ell,1},x^{(\beta,s)}_{\ell+1,1},x^{(\beta,s)}_{\ell+1,2},x^{(\beta,s)}_{\ell,2} \big\}.
\end{equation}
\begin{itemize}[leftmargin=0.5cm]

\item
If $o(x^{(\beta,s)}_{\ell,1})=o(x^{(\beta,s)}_{\ell+1,1})=o(x^{(\beta,s)}_{\ell+1,2})=o(x^{(\beta,s)}_{\ell,2})$,
then by Lemma \ref{lem:psirv-C} we know that the corresponding summand is divisible by
$Q(x^{(\beta,s)}_{\ell,1},x^{(\beta,s)}_{\ell,2},x^{(\beta,s)}_{\ell+1,1},x^{(\beta,s)}_{\ell+1,2})$,
cf.~\eqref{eq:Qform-C}, and so the $\phi_{\unl{d}}$-specialization is divisible by
\[
  Q(v^{1-\ell},v^{-2n+\ell-1},v^{-\ell},v^{-2n+\ell})\doteq (v^{2n-2\ell}-1)(v^{2n-2\ell+2}-1).
\]

\item
If $o(x^{(\beta,s)}_{\ell,1})>o(x^{(\beta,s)}_{\ell+1,1})=o(x^{(\beta,s)}_{\ell+1,2})=o(x^{(\beta,s)}_{\ell,2})$,
then $Q(x,v^{-2n+\ell-1},v^{-\ell},v^{-2n+\ell})\doteq (x-v^{-2n+\ell+1})(v^{2n-2\ell+2}-1)$ together with
  $\zeta\left(\frac{x^{(\beta,s)}_{\ell+1,2}}{x^{(\beta,s)}_{\ell,1}}\right)
   \zeta\left(\frac{x^{(\beta,s)}_{\ell,2}}{x^{(\beta,s)}_{\ell,1}}\right)$
contribute the same factors $(v^{2n-2\ell}-1)(v^{2n-2\ell+2}-1)$ into the $\phi_{\unl{d}}$-specialization of this summand.

\item
If $o(x^{(\beta,s)}_{\ell,1})=o(x^{(\beta,s)}_{\ell+1,1})=o(x^{(\beta,s)}_{\ell+1,2})>o(x^{(\beta,s)}_{\ell,2})$,
then $Q(v^{1-\ell},x,v^{-\ell},v^{-2n+\ell})\doteq (x-v^{-\ell-1})(v^{2n-2\ell+2}-1)$ together with
  $\zeta\left(\frac{x^{(\beta,s)}_{\ell,2}}{x^{(\beta,s)}_{\ell+1,1}}\right)
   \zeta\left(\frac{x^{(\beta,s)}_{\ell,2}}{x^{(\beta,s)}_{\ell,1}}\right)$
contribute the same factors $(v^{2n-2\ell}-1)(v^{2n-2\ell+2}-1)$ into the $\phi_{\unl{d}}$-specialization of this summand.

\item
If $o(x^{(\beta,s)}_{\ell+1,1})>o(x^{(\beta,s)}_{\ell+1,2})$ or
$o(x^{(\beta,s)}_{\ell,1})>o(x^{(\beta,s)}_{\ell+1,1})=o(x^{(\beta,s)}_{\ell+1,2})>o(x^{(\beta,s)}_{\ell,2})$,
then $\zeta$-factors
  $\zeta\left(\frac{x^{(\beta,s)}_{\ell,2}}{x^{(\beta,s)}_{\ell,1}}\right)
   \zeta\left(\frac{x^{(\beta,s)}_{\ell,2}}{x^{(\beta,s)}_{\ell+1,1}}\right)
   \zeta\left(\frac{x^{(\beta,s)}_{\ell+1,2}}{x^{(\beta,s)}_{\ell,1}}\right)$
contribute the same factor $(v^{2n-2\ell}-1)(v^{2n-2\ell+2}-1)$ into the $\phi_{\unl{d}}$-specialization of this summand.

\end{itemize}

\item $\beta=[i,n,i]$ with $1\leq i<n$ and $d_{\beta}\neq 0$.

Fix any $1\leq s\leq d_{\beta}$. We can assume that
\begin{gather*}
  o(x^{(\beta,s)}_{i,1})\geq  \cdots \geq o(x^{(\beta,s)}_{n-1,1}),\qquad
  o(x^{(\beta,s)}_{i,2}) \geq \cdots \geq o(x^{(\beta,s)}_{n-1,2})\geq o(x^{(\beta,s)}_{n,1}).
\end{gather*}
First, let us consider the $\zeta$-factors arising from the variables
\begin{equation*}
  \big\{ x^{(\beta,s)}_{i,1},x^{(\beta,s)}_{i+1,1},x^{(\beta,s)}_{i,2},x^{(\beta,s)}_{i+1,2} \big\}.
\end{equation*}
\begin{itemize}[leftmargin=0.5cm]

\item
If $o(x^{(\beta,s)}_{i,2})\ne o(x^{(\beta,s)}_{i+1,1})$, then the $\zeta$-factors
$\zeta\left(\frac{x^{(\beta,s)}_{i+1,1}}{x^{(\beta,s)}_{i,2}}\right)$ or
  $\zeta\left(\frac{x^{(\beta,s)}_{i,2}}{x^{(\beta,s)}_{i,1}}\right)
   \zeta\left(\frac{x^{(\beta,s)}_{i,2}}{x^{(\beta,s)}_{i+1,1}}\right)$
contribute $(w_{\beta,s}-v^2w'_{\beta,s})$ into the $\phi^{(1)}_{\beta}$-specialization of this summand.
Similarly, if $o(x^{(\beta,s)}_{i,1})\ne o(x^{(\beta,s)}_{i+1,2})$, then the $\zeta$-factors
$\zeta\left(\frac{x^{(\beta,s)}_{i+1,2}}{x^{(\beta,s)}_{i,1}}\right)$ or
  $\zeta\left(\frac{x^{(\beta,s)}_{i,1}}{x^{(\beta,s)}_{i,2}}\right)
   \zeta\left(\frac{x^{(\beta,s)}_{i,1}}{x^{(\beta,s)}_{i+1,2}}\right)$
contribute $(w_{\beta,s}-v^{-2}w'_{\beta,s})$ into the $\phi^{(1)}_{\beta}$-specialization of this summand.

\item
If $o(x^{(\beta,s)}_{i,2})= o(x^{(\beta,s)}_{i+1,1})$ and $o(x^{(\beta,s)}_{i,1})= o(x^{(\beta,s)}_{i+1,2})$, then
\[
  o(x^{(\beta,s)}_{i,1})=o(x^{(\beta,s)}_{i,2})=o(x^{(\beta,s)}_{i+1,1})=o(x^{(\beta,s)}_{i+1,2})=q.
\]
By Lemma \ref{lem:psirv-C}, we know that $\Psi(\tilde{\mathcal{E}}^{\epsilon}_{\beta_{q},r_{q}})$ contains the factor 
$Q(x^{(\beta,s)}_{i,1},x^{(\beta,s)}_{i,2},x^{(\beta,s)}_{i+1,1},x^{(\beta,s)}_{i+1,2})$, which contributes
$(w_{\beta,s}-v^2w'_{\beta,s})(w_{\beta,s}-v^{-2}w'_{\beta,s})$ into the $\phi^{(1)}_{\beta}$-specialization of this summand.

\item
If $o(x^{(\beta,s)}_{i,2})= o(x^{(\beta,s)}_{i+1,1})$ and $o(x^{(\beta,s)}_{i,1})\neq o(x^{(\beta,s)}_{i+1,2})$, then we have
\begin{align*}
  & q=o(x^{(\beta,s)}_{i,1})=o(x^{(\beta,s)}_{i,2})=o(x^{(\beta,s)}_{i+1,1})>o(x^{(\beta,s)}_{i+1,2}),\\
  \text{or}\quad & o(x^{(\beta,s)}_{i,1})>o(x^{(\beta,s)}_{i,2})=o(x^{(\beta,s)}_{i+1,1})=o(x^{(\beta,s)}_{i+1,2})=q,\\
  \text{or}\quad & o(x^{(\beta,s)}_{i,1})>o(x^{(\beta,s)}_{i,2})=o(x^{(\beta,s)}_{i+1,1})>o(x^{(\beta,s)}_{i+1,2}).
\end{align*}
For the first case, from $Q(w,w',v^{-1}w,x)\doteq (w-v^{2}w')(w-v^{-1}x)$ and the $\zeta$-factor
$\zeta\left(\frac{x^{(\beta,s)}_{i+1,2}}{x^{(\beta,s)}_{i,1}}\right)$ we see that the $\phi^{(1)}_{\beta}$-specialization
of this summand is divisible by $(w_{\beta,s}-v^2w'_{\beta,s})(w_{\beta,s}-v^{-2}w'_{\beta,s})$;
for the second case, from $Q(x,w',v^{-1}w,v^{-1}w')\doteq (w'-v^{2}x)(w'-v^{-2}w)$ and the $\zeta$-factor
$\zeta\left(\frac{x^{(\beta,s)}_{i+1,2}}{x^{(\beta,s)}_{i,1}}\right)$ we see that the $\phi^{(1)}_{\beta}$-specialization
of this summand is divisible by $(w_{\beta,s}-v^2w'_{\beta,s})(w_{\beta,s}-v^{-2}w'_{\beta,s})$;
finally, for the third case above, the $\phi^{(1)}_{\beta}$-specialization of the $\zeta$-factors
\[
  \zeta\left(\frac{x^{(\beta,s)}_{i+1,2}}{x^{(\beta,s)}_{i,2}}\right)
  \zeta\left(\frac{x^{(\beta,s)}_{i+1,2}}{x^{(\beta,s)}_{i,1}}\right)
  \zeta\left(\frac{x^{(\beta,s)}_{i+1,1}}{x^{(\beta,s)}_{i,1}}\right)
  \zeta\left(\frac{x^{(\beta,s)}_{i,2}}{x^{(\beta,s)}_{i,1}}\right)
\]
contributes $\frac{\langle 1\rangle^{2}_v}{w_{\beta,s}-w'_{\beta,s}}\cdot (w_{\beta,s}-v^2w'_{\beta,s})(w_{\beta,s}-v^{-2}w'_{\beta,s})$.
Thus, the $\phi^{(1)}_{\beta}$-specialization of this summand is divisible by $(w_{\beta,s}-v^2w'_{\beta,s})(w_{\beta,s}-v^{-2}w'_{\beta,s})$,
and the denominator $w_{\beta,s}-w'_{\beta,s}$ will be canceled (up to a monomial) with $\langle 1\rangle_v$ in the numerator
when specializing $w'_{\beta,s}\mapsto v^2 w_{\beta,s}$ in the second step of specialization $\phi_{\beta}$, cf.~\eqref{spe-C-3}.

\item
If $o(x^{(\beta,s)}_{i,2})\ne o(x^{(\beta,s)}_{i+1,1})$ and $o(x^{(\beta,s)}_{i,1})=o(x^{(\beta,s)}_{i+1,2})$, then
we can use the same analysis as for the above case to get that the $\phi^{(1)}_{\beta}$-specialization of this summand
is divisible by $(w_{\beta,s}-v^2w'_{\beta,s})(w_{\beta,s}-v^{-2}w'_{\beta,s})$.

\end{itemize}
Along with similar $\zeta$-factors arising from the variables
$\{x^{(\beta,s)}_{\ell,1},x^{(\beta,s)}_{\ell+1,1},x^{(\beta,s)}_{\ell,2},x^{(\beta,s)}_{\ell+1,2}\}$ for any $i<\ell<n-1$,
we see the  $\phi^{(1)}_{\beta}$-specialization of any summand is divisible by $B_{\beta}$ of \eqref{Bfactor-C}.
Now let us  consider the $\zeta$-factors between the variables $\{x^{(\beta,s)}_{n-1,1},x^{(\beta,s)}_{n-1,2},x^{(\beta,s)}_{n,1}\}$,
we can assume that $ o(x^{(\beta,s)}_{n-1,1})\geq o(x^{(\beta,s)}_{n-1,2})\geq o(x^{(\beta,s)}_{n,1})$, as otherwise
the corresponding term is specialized to zero under $\phi_{\unl{d}}$. Then:
\begin{itemize}[leftmargin=0.5cm]

\item
If $o(x^{(\beta,s)}_{n-1,2})>o(x^{(\beta,s)}_{n,1})$, then $\zeta\left(\frac{x^{(\beta,s)}_{n,1}}{x^{(\beta,s)}_{n-1,2}}\right)$
contributes a factor $\langle 2\rangle_{v}$ to the $\phi_{\unl{d}}$-specialization of that summand.

\item
If $o(x^{(\beta,s)}_{n-1,1})>o(x^{(\beta,s)}_{n-1,2})=o(x^{(\beta,s)}_{n,1})$, then
  $\zeta\left(\frac{x^{(\beta,s)}_{n,1}}{x^{(\beta,s)}_{n-1,1}}\right)
   \zeta\left(\frac{x^{(\beta,s)}_{n-1,2}}{x^{(\beta,s)}_{n-1,1}}\right)$
contribute a factor $\langle 2\rangle_{v}$ to the $\phi_{\unl{d}}$-specialization of that summand.

\item
If $o(x^{(\beta,s)}_{n-1,1})=o(x^{(\beta,s)}_{n-1,2})=o(x^{(\beta,s)}_{n,1})=q$, then we know $\beta_{q}=[i,n,j]$ with $i<j<n$
or $[i,n,i]$. According to Lemma \ref{lem:psirv-C}, if $\beta_{q}=[i,n,j]$ with $j\leq n-2$,
then $\Psi(\tilde{\mathcal{E}}^{\epsilon}_{\beta_{q},r_{q}})$ contains the factor
$Q(x,y,x^{(\beta,s)}_{n-1,1},x^{(\beta,s)}_{n-1,2})$, which contributes a factor $[2]_{v}$ into $\phi_{\unl{d}}(F)$;
if $\beta_{q}=[i,n,n-1]$ with $i<n-1$, then
$\Psi(\tilde{\mathcal{E}}^{\epsilon}_{\beta_{q},r_{q}})$ contains the factor
$(1+v^2)x^{(\beta,s)}_{n-1,1}x^{(\beta,s)}_{n-1,2}-vy(x^{(\beta,s)}_{n-1,1}+x^{(\beta,s)}_{n-1,2})$ or
$(1+v^2)y-v(x^{(\beta,s)}_{n-1,1}+x^{(\beta,s)}_{n-1,2})$, which contributes a factor $[2]_{v}$ into $\phi_{\unl{d}}(F)$;
finally, if  $\beta_{q}=[i,n,i]$, then $\Psi(\tilde{\mathcal{E}}^{\epsilon}_{\beta_{q},r_{q}})$ is divisible by
$\langle 1\rangle^{2n-2i-1}_{v}\langle 2\rangle^{2}_{v}$, which contributes a factor $[2]_{v}$ into
$\phi_{\unl{d}}(f\cdot \langle 1\rangle_{v}^{-k_1-\dots-k_{n-1}} \langle 2\rangle_{v}^{-k_{n}})$.

\end{itemize}

\end{itemize}
The above overall analysis shows that the $\phi_{\unl{d}}$-specialization of $f$ is divisible by $A_{\unl{d}}$ of~\eqref{rttconstant2-C}.

Next, let us verify that $\phi_{\unl{d}}(F)$ is divisible by $\prod_{\beta\in\Delta^{+}}G_{\beta}$, where $G_{\beta}$
are as in \eqref{eq:formulagbeta-C-1}--\eqref{eq:formulagbeta-C-4}.  We can expand
$\prod_{\ell=1}^{m}\tilde{\mathcal{E}}^{\epsilon}_{\beta_{\ell},r_{\ell}}$ as a linear combination of monomials
$\prod_{\ell=1}^{k}e_{i_{\ell},s_{\ell}}$ over $\BZ[v,v^{-1}]$, with $\unl{k}=\sum^{k}_{\ell=1}\alpha_{i_{\ell}}$.
Then it suffices to prove that each $\phi_{\unl{d}}(\Psi(e_{i_{1},s_{1}}\cdots e_{i_k,s_{k}}))$ is divisible by
$G_{\beta}$ for any $\beta\in\Delta^{+}$. For $\beta=[i,j]$ with $1\leq i\leq j\leq n$, this follows from
\cite[Lemma 3.51]{Tsy18}. It remains to treat the $\beta=[i,n,j]\ (1\leq i<j<n)$ and $\beta=[i,n,i]\ (1\leq i<n)$ cases.
Henceforth, we shall use the notation $\hat{o}(x^{(*,*)}_{*,*})=q$ if a variable $x^{(*,*)}_{*,*}$ is plugged into
$\Psi(e_{i_q,s_q})$ for some $1\leq q\leq k$.
\begin{itemize}[leftmargin=0.7cm]

\item $\beta=[i,n,j]$.
Fix any $1\leq s\neq s'\leq d_{\beta}$, we can assume that
\begin{gather*}
  \hat{o}(x^{(\beta,s)}_{i,1})> \cdots > \hat{o}(x^{(\beta,s)}_{n-1,1})> \hat{o}(x^{(\beta,s)}_{n,1}) >
  \hat{o}(x^{(\beta,s)}_{n-1,2})> \cdots > \hat{o}(x^{(\beta,s)}_{j,2}),\\
  \hat{o}(x^{(\beta,s')}_{i,1})>  \cdots > \hat{o}(x^{(\beta,s')}_{n-1,1})> \hat{o}(x^{(\beta,s')}_{n,1})>
  \hat{o}(x^{(\beta,s')}_{n-1,2})> \cdots> \hat{o}(x^{(\beta,s')}_{j,2}).
\end{gather*}
Let us first consider the variables
\begin{equation}
\label{icase-C}
  \big\{x^{(\beta,s)}_{i,1},x^{(\beta,s)}_{i+1,1},x^{(\beta,s')}_{i,1},x^{(\beta,s')}_{i+1,1}\big\}.
\end{equation}
Without loss of generality, we can assume that $\hat{o}(x^{(\beta,s)}_{i+1,1})> \hat{o}(x^{(\beta,s')}_{i+1,1})$.
\begin{itemize}[leftmargin=0.5cm]

\item
If $\hat{o}(x^{(\beta,s)}_{i+1,1})< \hat{o}(x^{(\beta,s')}_{i,1})$, then the $\phi_{\unl{d}}$-specialization of
  $\zeta\left(\frac{x^{(\beta,s')}_{i+1,1}}{x^{(\beta,s)}_{i,1}}\right)
   \zeta\left(\frac{x^{(\beta,s)}_{i+1,1}}{x^{(\beta,s')}_{i,1}}\right)$
contributes the factor $(w_{\beta,s}-v^2w_{\beta,s'})(w_{\beta,s'}-v^2w_{\beta,s})$.

\item
If $\hat{o}(x^{(\beta,s)}_{i+1,1})> \hat{o}(x^{(\beta,s')}_{i,1})$, then the $\phi_{\unl{d}}$-specialization of
  $\zeta\left(\frac{x^{(\beta,s')}_{i+1,1}}{x^{(\beta,s)}_{i,1}}\right)
   \zeta\left(\frac{x^{(\beta,s')}_{i+1,1}}{x^{(\beta,s)}_{i+1,1}}\right)
   \zeta\left(\frac{x^{(\beta,s')}_{i,1}}{x^{(\beta,s)}_{i+1,1}}\right)$
contributes the factor $(w_{\beta,s}-v^2w_{\beta,s'})(w_{\beta,s'}-v^2w_{\beta,s})$.

\end{itemize}
Similarly, the $\phi_{\unl{d}}$-specialization of the $\zeta$-factors arising from the following quadruples
\begin{align*}
  \big\{x^{(\beta,s)}_{i+1,1} \,,\, x^{(\beta,s)}_{i+2,1} \,,\, x^{(\beta,s')}_{i+1,1} \,,\, x^{(\beta,s')}_{i+2,1}\big\}
    \,,\, \dots \,,\,
  \big\{x^{(\beta,s)}_{n-2,1} \,,\, x^{(\beta,s)}_{n-1,1} \,,\, x^{(\beta,s')}_{n-2,1},x^{(\beta,s')}_{n-1,1}\big\}, \\
  \big\{x^{(\beta,s)}_{n-1,2} \,,\, x^{(\beta,s)}_{n-2,2}\,,\, x^{(\beta,s')}_{n-1,2} \,,\, x^{(\beta,s')}_{n-2,2}\big\}
    \,,\, \dots \,,\,
  \big\{x^{(\beta,s)}_{j+1,2} \,,\, x^{(\beta,s)}_{j,2} \,,\, x^{(\beta,s')}_{j+1,2},x^{(\beta,s')}_{j,2}\big\},
\end{align*}
along with the contribution of the $\zeta$-factors arising from~\eqref{icase-C} above, yields the overall
contribution of the factor $\{(w_{\beta,s}-v^{2}w_{\beta,s'})(w_{\beta,s'}-v^{2}w_{\beta,s})\}^{2n-i-j-2}$.

Next, let us consider the $\zeta$-factors arising from the variables
\begin{equation}
\label{eq:ncase}
  \big\{x^{(\beta,s)}_{n-1,1} \,,\, x^{(\beta,s)}_{n,1} \,,\, x^{(\beta,s)}_{n-1,2} \,,\,
  x^{(\beta,s')}_{n-1,1} \,,\, x^{(\beta,s')}_{n,1} \,,\, x^{(\beta,s')}_{n-1,2}\big\}.
\end{equation}
Without loss of generality, we can assume that $\hat{o}(x^{(\beta,s)}_{n,1})> \hat{o}(x^{(\beta,s')}_{n,1})$.
First, we note that
  $\zeta\left(\frac{x^{(\beta,s')}_{n,1}}{x^{(\beta,s)}_{n-1,1}}\right)
   \zeta\left(\frac{x^{(\beta,s')}_{n-1,2}}{x^{(\beta,s)}_{n,1}}\right)
   \zeta\left(\frac{x^{(\beta,s')}_{n-1,2}}{x^{(\beta,s)}_{n-1,1}}\right)$
contributes $(w_{\beta,s'}-v^{ 2}w_{\beta,s})(w_{\beta,s'}-v^{4}w_{\beta,s})$ into the $\phi_{\unl{d}}$-specialization.
Now we consider four cases.
\begin{itemize}[leftmargin=0.5cm]

\item
If $\hat{o}(x^{(\beta,s')}_{n,1})>\hat{o}(x^{(\beta,s)}_{n-1,2})\ \&\ \hat{o}(x^{(\beta,s')}_{n-1,1})>\hat{o}(x^{(\beta,s)}_{n,1})$,
then
  $\zeta\left(\frac{x^{(\beta,s)}_{n-1,2}}{x^{(\beta,s')}_{n,1}}\right)
   \zeta\left(\frac{x^{(\beta,s)}_{n-1,2}}{x^{(\beta,s')}_{n-1,1}}\right)
   \zeta\left(\frac{x^{(\beta,s)}_{n,1}}{x^{(\beta,s')}_{n-1,1}}\right)$
contributes $(w_{\beta,s}-v^{ 2}w_{\beta,s'})(w_{\beta,s}-v^{4}w_{\beta,s'})$ into the $\phi_{\unl{d}}$-specialization.

\item
If $\hat{o}(x^{(\beta,s')}_{n,1})<\hat{o}(x^{(\beta,s)}_{n-1,2})\ \&\ \hat{o}(x^{(\beta,s')}_{n-1,1})>\hat{o}(x^{(\beta,s)}_{n,1})$,
then
  $\zeta\left(\frac{x^{(\beta,s')}_{n-1,2}}{x^{(\beta,s)}_{n-1,2}}\right)
   \zeta\left(\frac{x^{(\beta,s')}_{n,1}}{x^{(\beta,s)}_{n-1,2}}\right)
   \zeta\left(\frac{x^{(\beta,s)}_{n,1}}{x^{(\beta,s')}_{n-1,1}}\right)$
contributes $(w_{\beta,s}-v^{ 2}w_{\beta,s'})(w_{\beta,s}-v^{4}w_{\beta,s'})$ into the $\phi_{\unl{d}}$-specialization.

\item
If $\hat{o}(x^{(\beta,s')}_{n,1})<\hat{o}(x^{(\beta,s)}_{n-1,2})\ \&\ \hat{o}(x^{(\beta,s')}_{n-1,1})<\hat{o}(x^{(\beta,s)}_{n,1})$,
then
\[
  \zeta\left(\frac{x^{(\beta,s')}_{n-1,2}}{x^{(\beta,s)}_{n-1,2}}\right)
  \zeta\left(\frac{x^{(\beta,s')}_{n,1}}{x^{(\beta,s)}_{n-1,2}}\right)
  \zeta\left(\frac{x^{(\beta,s')}_{n,1}}{x^{(\beta,s)}_{n,1}}\right)
  \zeta\left(\frac{x^{(\beta,s')}_{n-1,1}}{x^{(\beta,s)}_{n,1}}\right)
\]
contributes $(w_{\beta,s}-v^{ 2}w_{\beta,s'})(w_{\beta,s}-v^{4}w_{\beta,s'})$ into the $\phi_{\unl{d}}$-specialization.

\item
If $\hat{o}(x^{(\beta,s')}_{n,1})>\hat{o}(x^{(\beta,s)}_{n-1,2})\ \&\ \hat{o}(x^{(\beta,s')}_{n-1,1})<\hat{o}(x^{(\beta,s)}_{n,1})$,
then
\[
  \zeta\left(\frac{x^{(\beta,s)}_{n-1,2}}{x^{(\beta,s')}_{n,1}}\right)
  \zeta\left(\frac{x^{(\beta,s)}_{n-1,2}}{x^{(\beta,s')}_{n-1,1}}\right)
  \zeta\left(\frac{x^{(\beta,s')}_{n,1}}{x^{(\beta,s)}_{n,1}}\right)
  \zeta\left(\frac{x^{(\beta,s')}_{n-1,1}}{x^{(\beta,s)}_{n,1}}\right)
\]
contributes $(w_{\beta,s}-v^{ 2}w_{\beta,s'})(w_{\beta,s}-v^{4}w_{\beta,s'})$ into the $\phi_{\unl{d}}$-specialization.

\end{itemize}
We thus conclude that the $\phi_{\unl{d}}$-specialization of the $\zeta$-factors arising
from \eqref{eq:ncase} contributes the overall factor
\[
  (w_{\beta,s}-v^{2}w_{\beta,s'})(w_{\beta,s'}-v^{2}w_{\beta,s})(w_{\beta,s}-v^{4}w_{\beta,s'})(w_{\beta,s'}-v^{4}w_{\beta,s}).
\]
Similarly to the above analysis, the $\phi_{\unl{d}}$-specialization of the $\zeta$-factors arising from the tuples
\[
  \big\{x^{(\beta,s')}_{j-1,1},x^{(\beta,s')}_{j,1},x^{(\beta,s)}_{j,2}\big\} \qquad \mathrm{and} \qquad
  \big\{x^{(\beta,s')}_{\ell,1}, x^{(\beta,s')}_{\ell+1,1},x^{(\beta,s)}_{\ell+1,2},x^{(\beta,s)}_{\ell,2}\big\}\ (j\leq \ell\leq n-2)
\]
produces an overall factor
\[
  \prod_{\ell=j}^{n-2} (w_{\beta,s}-v^{2n-2\ell}w_{\beta,s'})
  \prod_{\ell=j}^{n-1}(w_{\beta,s}-v^{2n-2\ell+4}w_{\beta,s'}).
\]
This completes the verification of divisibility of $\phi_{\unl{d}}(F)$ by $G_{\beta}$ of~\eqref{eq:formulagbeta-C-3},
up to a monomial.

\item $\beta=[i,n,i]$.
Fix any $1\leq s\neq s'\leq d_{\beta}$. We can assume that
\begin{gather*}
  \hat{o}(x^{(\beta,t)}_{i,1})>  \cdots > \hat{o}(x^{(\beta,t)}_{n-1,1}),\quad
  \hat{o}(x^{(\beta,t)}_{i,2}) > \cdots > \hat{o}(x^{(\beta,t)}_{n-1,2})>\hat{o}(x^{(\beta,t)}_{n,1})\quad
  \mathrm{for} \ \ t=s\ \text{or}\ s'.
\end{gather*}
First, let us consider the $\zeta$-factors arising from the variables
\begin{equation*}
  \big\{x^{(\beta,s)}_{i,1},x^{(\beta,s)}_{i+1,1},x^{(\beta,s')}_{i,1},x^{(\beta,s')}_{i+1,1}\big\}.
\end{equation*}
Without loss of generality, we can assume that $\hat{o}(x^{(\beta,s)}_{i,1})> \hat{o}(x^{(\beta,s')}_{i,1})$.
\begin{itemize}[leftmargin=0.5cm]

\item
If $\hat{o}(x^{(\beta,s')}_{i,1})>\hat{o}(x^{(\beta,s)}_{i+1,1})$, then
  $\zeta\left(\frac{x^{(\beta,s')}_{i+1,1}}{x^{(\beta,s)}_{i,1}}\right)
   \zeta\left(\frac{x^{(\beta,s)}_{i+1,1}}{x^{(\beta,s')}_{i,1}}\right)$
contributes the factor $(w_{\beta,s}-v^2w_{\beta,s'})(w_{\beta,s'}-v^2w_{\beta,s})$ into the $\phi_{\unl{d}}$-specialization.

\item
If $\hat{o}(x^{(\beta,s')}_{i,1})<\hat{o}(x^{(\beta,s)}_{i+1,1})$, then
  $\zeta\left(\frac{x^{(\beta,s')}_{i+1,1}}{x^{(\beta,s)}_{i,1}}\right)
   \zeta\left(\frac{x^{(\beta,s')}_{i,1}}{x^{(\beta,s)}_{i,1}}\right)\zeta\left(\frac{x^{(\beta,s')}_{i,1}}{x^{(\beta,s)}_{i+1,1}}\right)$
contributes the factor $(w_{\beta,s}-v^2w_{\beta,s'})(w_{\beta,s'}-v^2w_{\beta,s})$ into the $\phi_{\unl{d}}$-specialization.

\end{itemize}
Likewise, we conclude that the $\phi_{\unl{d}}$-specialization of the $\zeta$-factors arising from
\[
  \big\{x^{(\beta,s)}_{\ell,t},x^{(\beta,s)}_{\ell+1,t},x^{(\beta,s')}_{\ell,t},x^{(\beta,s')}_{\ell+1,t}\big\}\
  (i\leq \ell\leq n-2,\ 1\leq t\leq 2)
\]
produces the overall factor of
\[
  \big\{ (w_{\beta,s}-v^{2}w_{\beta,s'})(w_{\beta,s'}-v^{2}w_{\beta,s}) \big\}^{2n-2i-2}.
\]
Analogously, the $\phi_{\unl{d}}$-specialization of the $\zeta$-factors arising from the quadruples
\[
  \big\{x^{(\beta,s)}_{\ell,1}, x^{(\beta,s)}_{\ell+1,1},x^{(\beta,s')}_{\ell,2},x^{(\beta,s')}_{\ell+1,2}\big\},\quad
  \big\{x^{(\beta,s)}_{\ell,2}, x^{(\beta,s)}_{\ell+1,2},x^{(\beta,s')}_{\ell,1},x^{(\beta,s')}_{\ell+1,1}\big\}, \ \
  (i\leq \ell\leq n-2)
\]
produces a total factor of
\[
  \big\{ (w_{\beta,s}-w_{\beta,s'})(w_{\beta,s'}-w_{\beta,s})(w_{\beta,s}-v^{4}w_{\beta,s'})(w_{\beta,s'}-v^{4}w_{\beta,s}) \big\}^{n-i-1}.
\]
Next, let us consider the $\phi_{\unl{d}}$-specialization of the $\zeta$-factors arising from the variables
\begin{equation}
\label{eq:inincase}
 \big\{x^{(\beta,s)}_{n-1,1} \,,\, x^{(\beta,s)}_{n-1,2}\,,\, x^{(\beta,s)}_{n,1} \,,\,
       x^{(\beta,s')}_{n-1,1} \,,\, x^{(\beta,s')}_{n-1,2}\,,\, x^{(\beta,s')}_{n,1}  \big\}.
\end{equation}
We can assume that
\[
  \hat{o}(x^{(\beta,s)}_{n-1,1}) > \hat{o}(x^{(\beta,s)}_{n-1,2}) > \hat{o}(x^{(\beta,s)}_{n,1})
  \qquad \mathrm{and} \qquad
  \hat{o}(x^{(\beta,s')}_{n-1,1}) > \hat{o}(x^{(\beta,s')}_{n-1,2}) > \hat{o}(x^{(\beta,s')}_{n,1}),
\]
as otherwise the corresponding term is specialized to zero under $\phi_{\unl{d}}$. Without loss of generality, we
can assume that $\hat{o}(x^{(\beta,s)}_{n,1})>\hat{o}(x^{(\beta,s')}_{n,1})$. Then
  $\zeta\left(\frac{x^{(\beta,s')}_{n,1}}{x^{(\beta,s)}_{n-1,2}}\right)
   \zeta\left(\frac{x^{(\beta,s')}_{n,1}}{x^{(\beta,s)}_{n-1,1}}\right)$
contributes $(w_{\beta,s'}-v^{ 2}w_{\beta,s})(w_{\beta,s'}-v^{4}w_{\beta,s})$.
\begin{itemize}[leftmargin=0.5cm]

\item
If $\hat{o}(x^{(\beta,s')}_{n-1,2})>\hat{o}(x^{(\beta,s)}_{n,1})$, then
$\zeta\left(\frac{x^{(\beta,s)}_{n,1}}{x^{(\beta,s')}_{n-1,2}}\right)$ contributes $(w_{\beta,s}-v^{4}w_{\beta,s'})$;
if $\hat{o}(x^{(\beta,s')}_{n-1,2})<\hat{o}(x^{(\beta,s)}_{n,1})$, then
  $\zeta\left(\frac{x^{(\beta,s')}_{n-1,2}}{x^{(\beta,s)}_{n,1}}\right)
   \zeta\left(\frac{x^{(\beta,s')}_{n-1,2}}{x^{(\beta,s)}_{n-1,2}}\right)
   \zeta\left(\frac{x^{(\beta,s')}_{n-1,2}}{x^{(\beta,s)}_{n-1,1}}\right)$
contributes $(w_{\beta,s}-v^{4}w_{\beta,s'})$.

\item
If $\hat{o}(x^{(\beta,s')}_{n-1,1})>\hat{o}(x^{(\beta,s)}_{n,1})$, then
$\zeta\left(\frac{x^{(\beta,s)}_{n,1}}{x^{(\beta,s')}_{n-1,1}}\right)$ contributes $(w_{\beta,s}-v^{2}w_{\beta,s'})$;
if $\hat{o}(x^{(\beta,s')}_{n-1,1})<\hat{o}(x^{(\beta,s)}_{n,1})$, then
  $\zeta\left(\frac{x^{(\beta,s')}_{n-1,1}}{x^{(\beta,s)}_{n,1}}\right)
   \zeta\left(\frac{x^{(\beta,s')}_{n-1,1}}{x^{(\beta,s)}_{n-1,2}}\right)
   \zeta\left(\frac{x^{(\beta,s')}_{n-1,1}}{x^{(\beta,s)}_{n-1,1}}\right)$
contributes $(w_{\beta,s}-v^{2}w_{\beta,s'})$.

\end{itemize}
We thus conclude that the $\phi_{\unl{d}}$-specialization of the $\zeta$-factors arising
from \eqref{eq:inincase} contributes the overall factor
\[
  (w_{\beta,s}-v^{2}w_{\beta,s'})(w_{\beta,s'}-v^{2}w_{\beta,s})(w_{\beta,s}-v^{4}w_{\beta,s'})(w_{\beta,s'}-v^{4}w_{\beta,s}).
\]
This completes the verification of divisibility of $\phi_{\unl{d}}(F)$ by $G_{\beta}$ of~\eqref{eq:formulagbeta-C-4}, up to
a monomial.
\end{itemize}

Finally, to prove that $F$ is integral, we need to show that for any $\beta\in\Delta^{+}$ and  $1\leq r\leq \ell_{\beta}$,
the contribution of the $\zeta$-factors between the variables $x^{(*,*)}_{*,*}$ that got specialized to $v^{?}z_{\beta,r}$
into $\Upsilon_{\unl{d},\unl{t}}(F)$ is divisible by $[t_{\beta,r}]_{v_{\beta}}!$. For $\beta=[i,j]$, this follows from
\cite[Lemma 3.51]{Tsy18}. For $\beta=[i,n,j]$ with $i<j<n$, we have $v_{\beta}=v_{i}=v$, and in the above analysis we never
used the $\zeta$-factors $\zeta\left(\frac{x^{(\beta,s)}_{i,1}}{x^{(\beta,s')}_{i,1}}\right)$ (with $1\leq s\neq s'\leq d_{\beta}$)
for the divisibility of $\phi_{\unl{d}}(F)$ by $G_{\beta}$ (see the analysis for the variables \eqref{icase-C}). For
$\beta=[i,n,i]$, we have $v_{\beta}=v_{n}=v^{2}$, and in the above analysis we never used the $\zeta$-factors
$\zeta\left(\frac{x^{(\beta,s)}_{n,1}}{x^{(\beta,s')}_{n,1}}\right)$ (with $1\leq s\neq s'\leq d_{\beta}$) for the divisibility
of $\phi_{\unl{d}}(F)$ by $G_{\beta}$ (see the analysis for the variables \eqref{eq:inincase}). We can thus appeal to the
``rank $1$'' computation of~\cite[Lemma 3.46]{Tsy18} to obtain the claimed divisibility by $[t_{\beta,r}]_{v_{\beta}}!$.
\end{proof}

Combining Propositions \ref{spekpC},  \ref{vanishC}, \ref{spanC}, \ref{goodC}, we obtain
the following upgrade of Theorem~\ref{shufflePBWD-C}:

\begin{Thm}\label{rttthm-C}
(a) The $\BQ(v)$-algebra isomorphism $\Psi\colon \qlc \,\iso\, S$ of Theorem {\rm \ref{shufflePBWD-C}(a)}
gives rise to a $\BZ[v,v^{-1}]$-algebra isomorphism $\Psi\colon \integralc \,\iso\, \mathcal{S}$.

\noindent
(b) Theorem {\rm \ref{PBWDintegralrtt}} holds for $\fg$ of type $C_n$.
\end{Thm}


\medskip

\section{Shuffle algebra and its integral forms in type $D_n$}\label{type D}

In this section, we establish the key properties of the specialization maps for the shuffle algebras of type $D_{n}$.
This implies the shuffle algebra realization and PBWD-type theorems for $\qld$ and its integral forms.


\subsection{$\qld$ and its shuffle algebra realization}

In type $D_{n}$, for any $F\in S_{\unl{k}}$ with $\unl{k}\in\mathbb{N}^{n}$, the wheel conditions are:
\begin{equation*}
\begin{aligned}
  F(\{x_{i,r}\}_{1\leq i\leq n}^{1\leq r\leq k_{i}})=0 \quad \text{once} \quad
  & x_{i,1}=v^{2}x_{i,2}=vx_{i+1,1} \quad \text{for some} \quad 1\leq i\leq n-2,\\
  \text{or} \quad
  & x_{i,1}=v^{2}x_{i,2}=vx_{i-1,1} \quad \text{for some} \quad 2\leq i\leq n-1,\\
   \text{or} \quad
   &x_{n-2,1}=v^{2}x_{n-2,2}=vx_{n,1},\\
  \text{or} \quad
  & x_{n,1}=v^{2}x_{n,2}=vx_{n-2,1}.
\end{aligned}
\end{equation*}
Recall the notations \eqref{eq:abbr-prD} for positive roots in type $D_n$. Similarly to type $C$, we shall use
$\mathrm{denom}_\beta$ to denote the denominator in~\eqref{polecondition} for any $F\in S_\beta$, for example for
$F=\Psi(\tilde{E}^{\pm}_{\beta,s})$.

\begin{Lem}\label{lem:psirv-D}
Consider the particular choices~\eqref{rvd1}--\eqref{rvd4} of quantum root vectors
$\{\tilde{E}^{\pm}_{\beta,s}\}_{\beta\in\Delta^{+}}^{s\in\BZ}$. Their images under $\Psi$ of~\eqref{eq:Psi-homom}
in the shuffle algebra $S$ of type $D_{n}$  are as follows:
\begin{itemize}[leftmargin=0.7cm]

\item
If $\beta=[i,j]$ with $1\leq i\leq j<n$ or $i=j=n$, then for any $s=s_{i}+\cdots+s_{j}$ used in \eqref{rvd1}:
\begin{align*}
  \Psi(\tilde{E}^{+}_{[i,j],s})\doteq
    \frac{\langle 1\rangle_{v}^{j-i}}{\mathrm{denom}_{[i,j]}} \cdot x_{i,1}^{s_{i}+1}\cdots x_{j-1,1}^{s_{j-1}+1}x_{j,1}^{s_{j}},\
  \Psi(\tilde{E}^{-}_{[i,j],s})\doteq
    \frac{\langle 1\rangle_{v}^{j-i}}{\mathrm{denom}_{[i,j]}} \cdot x_{i,1}^{s_{i}}x_{i+1,1}^{s_{i+1}+1}\cdots x_{j,1}^{s_{j}+1}.
\end{align*}

\item
If $\beta=[i,n]$ with $1\leq i\leq n-2$, then for any $s=s_{i}+\cdots+s_{n-2}+s_{n}$ used in \eqref{rvd2}:
\begin{align*}
  & \Psi(\tilde{E}^{+}_{[i,n],s})\doteq
    \frac{\langle 1\rangle_{v}^{n-i-1}}{\mathrm{denom}_{[i,n]}} \cdot x_{i,1}^{s_{i}+1}\cdots x_{n-2,1}^{s_{n-2}+1}x_{n,1}^{s_{n}},\\
  & \Psi(\tilde{E}^{-}_{[i,n],s})\doteq
    \frac{\langle 1\rangle_{v}^{n-i-1}}{\mathrm{denom}_{[i,n]}} \cdot x_{i,1}^{s_{i}}x_{i+1,1}^{s_{i+1}+1}\cdots x_{n-2,1}^{s_{n-2}+1} x_{n,1}^{s_{n}+1}.
\end{align*}

\item
If $\beta=[i,n,n-1]$ with $1\leq i\leq n-2$, then for any $s=s_{i}+\cdots+s_{n-2}+s_{n-1}+s_{n}$ used in \eqref{rvd3}:
\begin{align*}
  & \Psi(\tilde{E}^{+}_{[i,n,n-1],s})\doteq
   \frac{ \langle 1\rangle_{v}^{n-i}}{\mathrm{denom}_{[i,n,n-1]}} \cdot x_{i,1}^{s_{i}+1}\cdots x_{n-3,1}^{s_{n-3}+1}x_{n-2,1}^{s_{n-2}+2}x^{s_{n-1}}_{n-1,1}x_{n,1}^{s_{n}},\\
  & \Psi(\tilde{E}^{-}_{[i,n,n-1],s})\doteq
    \frac{\langle 1\rangle_{v}^{n-i}}{\mathrm{denom}_{[i,n,n-1]}} \cdot x_{i,1}^{s_{i}}x_{i+1,1}^{s_{i+1}+1}\cdots x_{n-2,1}^{s_{n-2}+1}x_{n-1,1}^{s_{n-1}+1} x_{n,1}^{s_{n}+1}.
\end{align*}

\item
If $\beta=[i,n,j]$ with $1\leq i<j\leq n-2$, then for any decomposition $s=s_{i}+\cdots +s_{j-1}+2s_{j}+\cdots +2s_{n-2}+s_{n-1}+s_{n}$
used in \eqref{rvd4}, we have:
\begin{equation*}
  \Psi(\tilde{E}^{+}_{[i,n,j],s})\doteq
  \frac{\langle 1\rangle_{v}^{2n-i-j-1}}{\mathrm{denom}_{[i,n,j]}}\cdot
  g_{1}\cdot \prod_{\ell=j}^{n-2}(v^2 x_{\ell,1}-x_{\ell,2})(v^{2} x_{\ell,2}-x_{\ell,1}),
\end{equation*}
\begin{equation*}
  \Psi(\tilde{E}^{-}_{[i,n,j],s})\doteq
  \frac{\langle 1\rangle_{v}^{2n-i-j-1}}{\mathrm{denom}_{[i,n,j]}} \cdot
  g_{2}\cdot \prod_{\ell=j}^{n-2}(v^2 x_{\ell,1}-x_{\ell,2})(v^{2} x_{\ell,2}-x_{\ell,1}),
\end{equation*}
where
\begin{align*}
  & g_{1}=\prod^{j-2}_{\ell=i}x_{\ell,1}^{s_{\ell}+1}x^{s_{j-1}+2}_{j-1,1}(x_{j,1}x_{j,2})^{s_{j}}
    \prod^{n-2}_{\ell=j+1}(x_{\ell,1}x_{\ell,2})^{s_{\ell}+1}x_{n-1,1}^{s_{n-1}+1}x_{n,1}^{s_{n}+1},\\
  & g_{2}=x_{i,1}^{s_{i}}\prod^{j-1}_{\ell=i+1}x_{\ell,1}^{s_{\ell}+1}
    \prod^{n-2}_{\ell=j}(x_{\ell,1}x_{\ell,2})^{s_{\ell}+1}x_{n-1,1}^{s_{n-1}+1}x_{n,1}^{s_{n}+1}.
\end{align*}

\end{itemize}
\end{Lem}

\begin{proof}
Straightforward computation.
\end{proof}

For more general quantum root vectors $E_{\beta,s}$ defined in~\eqref{rootvector3}, we have:

\begin{Lem}\label{phirv-D}
For any $s\in \BZ$ and any choices of  $s_k$ and $\lambda_k$ in \eqref{rootvector3}, we have:
\begin{equation}
\label{eq:D-general}
  \phi_{\beta}\left(\Psi\left(E_{\beta,s}\right)\right) \doteq \langle 1 \rangle_{v}^{\abs{\beta}-1} \cdot w_{\beta,1}^{s+\abs{\beta}-1}
  \qquad \forall\, (\beta,s)\in\Delta^{+}\times \BZ.
\end{equation}
\end{Lem}

\begin{proof}
It suffices to treat the cases of $\beta=[i,n,j]$ with $i<j\leq n-2$, since the other cases follow from type $A_{n-1}$ results
of~\cite[Lemma 1.4]{Tsy22}.

Let us first verify~\eqref{eq:D-general} for $\beta=[i,n,n-2]$. Recall that $E_{\beta,s}=[E_{\alpha,r},e_{n-2,s_{n-i+2}}]_{\lambda}$
with $\alpha=[i,n,n-1]$, $r=s_{1}+\cdots+s_{n-i+1}$, and $\lambda\in v^{\BZ}$, so that
\[
  \phi_{\beta}(\Psi(E_{\beta,s}))=
  \phi_{\beta}\left(\Psi(E_{\alpha,r})\star\Psi(e_{n-2,s_{n-i+2}})\right)-
  \lambda \phi_{\beta}\left(\Psi(e_{n-2,s_{n-i+2}})\star\Psi(E_{\alpha,r})\right).
\]
First, we claim that $\phi_{\beta}\left(\Psi(E_{\alpha,r})\star \Psi(e_{n-2,s_{n-i+2}})\right)=0$. To this end, we note that
\begin{equation}
\label{eq:psialphan-2-D}
  \Psi(E_{\alpha,r})\star \Psi(e_{n-2,s_{n-i+2}})=\sum_{\sigma\in\mathfrak{S}_{2}}F_{\beta}(\dots,x_{n-2,\sigma(1)},x_{n-2,\sigma(2)},\dots),
\end{equation}
where
\begin{align*}
  F_{\beta}=\ &\Psi(E_{\alpha,r})(x_{i,1},\dots,x_{n,1}) \Psi(e_{n-2,s_{n-i+2}})(x_{n-2,2})\times\\
    & \zeta\left(\frac{x_{n-3,1}}{x_{n-2,2}}\right)\zeta\left(\frac{x_{n-2,1}}{x_{n-2,2}}\right)
     \zeta\left(\frac{x_{n-1,1}}{x_{n-2,2}}\right)\zeta\left(\frac{x_{n,1}}{x_{n-2,2}}\right).
\end{align*}
Let us show that the $\phi_{\beta}$-specialization of each $\sigma(F_{\beta})$ in the symmetrization \eqref{eq:psialphan-2-D} vanishes:
\begin{itemize}[leftmargin=0.7cm]

\item
if $x^{(\beta,1)}_{n-2,2}$ is plugged into $\Psi(E_{\alpha,r})$, then $x^{(\beta,1)}_{n-2,1}$ is plugged into
$\Psi(e_{n-2,s_{n-i+2}})$ and so the $\phi^{(1)}_{\beta}$-specialization of the corresponding summand vanishes
due to the $\zeta$-factor $\zeta\left(\frac{x^{(\beta,1)}_{n-3,1}}{x^{(\beta,1)}_{n-2,1}}\right)$;

\item
if $x^{(\beta,1)}_{n-2,2}$ is plugged into $\Psi(e_{n-2,s_{n-i+2}})$, then
$
  \zeta\left(\frac{x^{(\beta,1)}_{n-3,1}}{x^{(\beta,1)}_{n-2,2}}\right)
  \zeta\left(\frac{x^{(\beta,1)}_{n-2,1}}{x^{(\beta,1)}_{n-2,2}}\right)
  \zeta\left(\frac{x^{(\beta,1)}_{n-1,1}}{x^{(\beta,1)}_{n-2,2}}\right)
$
contributes $B_{[i,n,n-2]}=(w_{\beta,1}-w'_{\beta,1})(w_{\beta,1}-v^{-4}w'_{\beta,1})$ into the $\phi^{(1)}_{\beta}$-specialization
of the corresponding summand, and so the $\phi_\beta$-specialization vanishes due to the $\zeta$-factor
$\zeta\left(\frac{x^{(\beta,1)}_{n,1}}{x^{(\beta,1)}_{n-2,2}}\right)$.

\end{itemize}
The evaluation of $\phi_{\beta}\left(\Psi(e_{n-2,s_{n-i+2}})\star\Psi(E_{\alpha,r})\right)$ proceeds in a similar way,
treating two cases:
\begin{itemize}[leftmargin=0.7cm]

\item
if $x^{(\beta,1)}_{n-2,2}$ is plugged into $\Psi(E_{\alpha,r})$, then $x^{(\beta,1)}_{n-2,1}$ is plugged into
$\Psi(e_{n-2,s_{n-i+2}})$ and so the $\phi^{(1)}_{\beta}$-specialization of the corresponding summand vanishes
due to the $\zeta$-factor $\zeta\left(\frac{x^{(\beta,1)}_{n-2,1}}{x^{(\beta,1)}_{n-1,1}}\right)$;

\item
if $x^{(\beta,1)}_{n-2,2}$ is plugged into $\Psi(e_{n-2,s_{n-i+2}})$, then
$
  \zeta\left(\frac{x^{(\beta,1)}_{n-2,2}}{x^{(\beta,1)}_{n-3,1}}\right)
  \zeta\left(\frac{x^{(\beta,1)}_{n-2,2}}{x^{(\beta,1)}_{n-2,1}}\right)
  \zeta\left(\frac{x^{(\beta,1)}_{n-2,2}}{x^{(\beta,1)}_{n-1,1}}\right)
$
contributes $B_{[i,n,n-2]}$ into the $\phi^{(1)}_{\beta}$-specialization of the corresponding summand,
so that the overall $\phi_\beta$-specialization of the corresponding summand has the form
$$
  \doteq
  \left\{\langle 1\rangle_{v}^{|\alpha|-1}  w_{\beta,1}^{r+|\alpha|-1} (w'_{\beta,1})^{s_{n-i+2}}
  (w'_{\beta,1}-v^{2} w_{\beta,1})\right\}_{w'_{\beta,1}\mapsto w_{\beta,1}} \doteq
  \langle 1\rangle_{v}^{|\beta|-1} \cdot w_{\beta,1}^{s+|\beta|-1},
$$
where we used $\phi_\alpha(E_{\alpha,r})\doteq \langle 1\rangle_{v}^{|\alpha|-1} \cdot w_{\alpha,1}^{r+|\alpha|-1}$
and utilized the remaining $\zeta$-factor $\zeta\left(\frac{x^{(\beta,1)}_{n-2,2}}{x^{(\beta,1)}_{n,1}}\right)$.

\end{itemize}
This completes our proof of \eqref{eq:D-general} for $\beta=[i,n,n-2]$.

We now verify \eqref{eq:D-general} for $\beta=[i,n,j]$ assuming it holds for any $\beta'=[i,n,k]$ with $j<k\leq n-2$.
Recall that $E_{\beta,s}=[E_{\alpha,r},e_{j,s_{2n-i-j}}]_{\lambda}$ with $\alpha=[i,n,j+1]$,  $r=s_{1}+\cdots+s_{2n-i-j-1}$,
and $\lambda\in v^{\BZ}$. Similarly to the previous case, we claim that
$\phi_{\beta}\left(\Psi(E_{\alpha,r})\star\Psi(e_{j,s_{2n-i-j}})\right)=0$. Indeed:
\begin{itemize}[leftmargin=0.7cm]

\item
if $x^{(\beta,1)}_{j,2}$ is plugged into $\Psi(E_{\alpha,r})$, then $x^{(\beta,1)}_{j,1}$ is plugged into
$\Psi(e_{j,s_{2n-i-j}})$ and so the $\phi^{(1)}_{\beta}$-specialization of the corresponding summand vanishes
due to the $\zeta$-factor $\zeta\left(\frac{x^{(\beta,1)}_{j-1,1}}{x^{(\beta,1)}_{j,1}}\right)$;

\item
if $x^{(\beta,1)}_{j,2}$ is plugged into $\Psi(e_{j,s_{2n-i-j}})$, then the $\phi^{(1)}_{\beta}$-specialization of
the corresponding summand vanishes again, due to the presence of the $\zeta$-factor
$\zeta\left(\frac{x^{(\beta,1)}_{j+1,2}}{x^{(\beta,1)}_{j,2}}\right)$.

\end{itemize}
The evaluation of $\phi_{\beta}\left(\Psi(e_{j,s_{2n-i-j}})\star \Psi(E_{\alpha,r})\right)$ proceeds by analyzing similar two cases:
\begin{itemize}[leftmargin=0.7cm]

\item
if $x^{(\beta,1)}_{j,2}$ is plugged into $\Psi(E_{\alpha,r})$, then $x^{(\beta,1)}_{j,1}$ is plugged into
$\Psi(e_{j,s_{2n-i-j}})$ and so the $\phi^{(1)}_{\beta}$-specialization of the corresponding summand vanishes
due to the $\zeta$-factor $\zeta\left(\frac{x^{(\beta,1)}_{j,1}}{x^{(\beta,1)}_{j+1,1}}\right)$;

\item
if $x^{(\beta,1)}_{j,2}$ is plugged into $\Psi(e_{j,s_{2n-i-j}})$, then by the induction we know that
$\phi^{(1)}_{\alpha}(\Psi(E_{\alpha,r}))$ is divisible by $B_{\alpha}$, and thus evoking the product of $\zeta$-factors
$
  \zeta\left(\frac{x^{(\beta,1)}_{j,2}}{x^{(\beta,1)}_{j-1,1}}\right)
  \zeta\left(\frac{x^{(\beta,1)}_{j,2}}{x^{(\beta,1)}_{j,1}}\right)
  \zeta\left(\frac{x^{(\beta,1)}_{j,2}}{x^{(\beta,1)}_{j+1,1}}\right)
$,
we see that the $\phi^{(1)}_{\beta}$-specialization of the corresponding summand is divisible by $B_{\beta}$.
Moreover, after dividing by $B_{\beta}$, we know by the induction assumption that the overall $\phi_\beta$-specialization
of the corresponding summand is
\begin{align*}
  \doteq &
  \left\{\langle 1\rangle_{v}^{|\alpha|-1}  w_{\beta,1}^{r+|\alpha|-1}  (w'_{\beta,1})^{s_{2n-i-j}}
  (v^{j+3-2n} w'_{\beta,1}-v^{j+5-2n} w'_{\beta,1})\right\}_{w'_{\beta,1}\mapsto w_{\beta,1}}\\
   \doteq&
  \langle 1\rangle_{v}^{|\beta|-1} \cdot w_{\beta,1}^{s+|\beta|-1},
\end{align*}
where we used $\phi_\alpha(E_{\alpha,r})\doteq \langle 1\rangle_{v}^{|\alpha|-1} \cdot w_{\alpha,1}^{r+|\alpha|-1}$
and utilized the remaining $\zeta$-factor $\zeta\left(\frac{x^{(\beta,1)}_{j,2}}{x^{(\beta,1)}_{j+1,2}}\right)$.

\end{itemize}
This completes our proof of \eqref{eq:D-general} for any $\beta=[i,n,j]$ with $i<j\leq n-2$.
\end{proof}

Let us now generalize the above lemma by computing $\phi_{\unl{d}}(\Psi(E_{h}))$ for any $h\in H_{\unl{k},\unl{d}}$.
Similarly to Proposition \ref{prop:phidEh}, we have:

\begin{Prop}\label{prop:phidEh-D}
For a summand $\sigma(F_{h})$ in the symmetrization \eqref{eq:part-symm}, we have $\phi_{\unl{d}}(\sigma(F_{h}))=0$ unless
for any $\beta\in\Delta^{+}$ and $1\leq s\leq d_{\beta}$, there is $s'$ with $1\leq s'\leq d_{\beta}$ so that
\begin{equation}
\label{fsunld-D}
  o(x^{(\beta,s')}_{i,t})=(\beta,s)\  \text{for any} \ i\in\beta\  \text{and}\  1\leq t\leq \nu_{\beta,i},
\end{equation}
that is we plug the variables $x^{(\beta,s')}_{*,*}$ into the same function $\Psi(E_{\beta,r_{\beta}(h,s)})$.
\end{Prop}

\begin{proof}
We shall use the same notation and argument as in the proof of Proposition \ref{prop:phidEh}. Fix $(\gamma,p)$ with
$\gamma\in \Delta^{+}$ and $1\leq p\leq d_{\gamma}$. It suffices to prove that if \eqref{fsunld-D} holds for any
$(\beta,s)<(\gamma,p)$, then $\phi_{\unl{d}}(\sigma(F_{h}))=0$ unless \eqref{fsunld-D} holds for $(\gamma,p)$.
The proof proceeds by an induction on~$n$.

\noindent
\underline{Step 1} (base of induction): Verification for type $D_{4}$.
\begin{itemize}[leftmargin=0.7cm]

\item
Case 1: $\gamma<[1,4,2]$. Suppose $o(x^{(\eta,r)}_{1,1})=(\gamma,p)$. If $\eta\neq [1,4,2]$, then due to the $A_3$-type
results we know that $\phi_{\unl{d}}(\sigma(F_{h}))=0$ unless $\eta=\gamma$ and we plug all the variables $x^{(\gamma,r)}_{*,*}$
into $\Psi(E_{\gamma,r_{\gamma}(h,p)})$. If $\eta=[1,4,2]$, then due to the $\zeta$-factors
  $\zeta\left(\frac{x^{(\eta,r)}_{1,1}}{x^{(\eta,r)}_{2,1}}\right)
   \zeta\left(\frac{x^{(\eta,r)}_{2,1}}{x^{(\eta,r)}_{3,1}}\right)
   \zeta\left(\frac{x^{(\eta,r)}_{2,1}}{x^{(\eta,r)}_{4,1}}\right)$,
we know that $\phi_{\unl{d}}(\sigma(F_{h}))=0$ unless
$
  o(x^{(\eta,r)}_{1,1})\geq o(x^{(\eta,r)}_{2,1})\geq o(x^{(\eta,r)}_{3,1}) \,\&\, o(x^{(\eta,r)}_{4,1}).
$
Since we already plugged variables into  all $\Psi(E_{\beta,r_{\beta}(h,s)})$ with $(\beta,s)<(\gamma,p)$, we must have
\[
  o(x^{(\eta,r)}_{1,1})= o(x^{(\eta,r)}_{2,1})= o(x^{(\eta,r)}_{3,1})=o(x^{(\eta,r)}_{4,1})=(\gamma,p),
\]
and $o(x^{(\eta,r)}_{2,2})>(\gamma,p)$. Then the $\zeta$-factors
$
  \zeta\left(\frac{x^{(\eta,r)}_{1,1}}{x^{(\eta,r)}_{2,2}}\right)
  \zeta\left(\frac{x^{(\eta,r)}_{2,1}}{x^{(\eta,r)}_{2,2}}\right)
  \zeta\left(\frac{x^{(\eta,r)}_{3,1}}{x^{(\eta,r)}_{2,2}}\right)
$
contribute $B_{\eta}=(w_{\eta,r}-w'_{\eta,r})(w_{\eta,r}-v^{-4}w'_{\eta,r})$ into the $\phi^{(1)}_{\eta}$-specialization
of the corresponding summand, and so the overall $\phi_{\eta}$-specialization vanishes due to the $\zeta$-factor
$\zeta\left(\frac{x^{(\eta,r)}_{4,1}}{x^{(\eta,r)}_{2,2}}\right)$.

\item Case 2: $\gamma=[1,4,2]$. Suppose $o(x^{(\eta,r)}_{1,1})=(\gamma,p)$. Since $1\in\eta$ and we already plugged variables
into all $\Psi(E_{\beta,r_{\beta}(h,s)})$ with $(\beta,s)<(\gamma,p)$ satisfying the rules \eqref{fsunld-D}, we have $\eta=\gamma$.
By the above argument in Case 1, we know that $\phi_{\unl{d}}(\sigma(F_{h}))=0$ unless we plug all the variables $x^{(\gamma,r)}_{*,*}$
into $\Psi(E_{\gamma,r_{\gamma}(h,p)})$.

\item Case 3: $\gamma>[1,4,2]$. We can use type $A_2$ results.

\end{itemize}
Thus the proposition is true for type $D_{4}$.

\medskip
\noindent
\underline{Step 2} (step of induction): Assuming the validity for type $D_{n-1}$, let us prove it for $D_n$.

To this end we present case-by-case study:
\begin{itemize}[leftmargin=0.7cm]

\item
Case 1: $\gamma\leq [1,n,3]$. Suppose $o(x^{(\eta,r)}_{1,1})=(\gamma,p)$, so that $\eta\geq \gamma$.

\begin{itemize}[leftmargin=0.5cm]

\item
If $\eta\leq [1,n,n-1]$, then the result follows from $A_{n-1}$-type case.

\item
If $\eta=[1,n,n-2]$, then we know that $\phi_{\unl{d}}(\sigma(F_{h}))=0$ unless
\[
  (\gamma,p)=o(x^{(\eta,r)}_{1,1})=\cdots=  o(x^{(\eta,r)}_{n-1,1})= o(x^{(\eta,r)}_{n,1}).
\]
If $o(x^{(\eta,r)}_{n-2,2})\ne (\gamma,p)$, then $o(x^{(\eta,r)}_{n-2,2})>(\gamma,p)$, and so the product of $\zeta$-factors
\[
  \zeta\left(\frac{x^{(\eta,r)}_{n-3,1}}{x^{(\eta,r)}_{n-2,2}}\right)
  \zeta\left(\frac{x^{(\eta,r)}_{n-2,1}}{x^{(\eta,r)}_{n-2,2}}\right)
  \zeta\left(\frac{x^{(\eta,r)}_{n-1,1}}{x^{(\eta,r)}_{n-2,2}}\right)
  \zeta\left(\frac{x^{(\eta,r)}_{n,1}}{x^{(\eta,r)}_{n-2,2}}\right)
\]
contributes $(w_{\eta,r}-v^{-4}w'_{\eta,r})(w_{\eta,r}-w'_{\eta,r})^{2}$ into the $\phi^{(1)}_{\eta}$-specialization of the summand.
Since the factor~\eqref{Bfactor-D} contains a single copy of $(w_{\eta,r}-w'_{\eta,r})$, we thus get $\phi_{\unl{d}}(\sigma(F_{h}))=0$.

\item
If $\eta=[1,n,j]$ with $2\leq j<n-2$, then by the induction assumption applied to $\tilde{\eta}=[1,n,j+1]$ we know that
$\phi_{\unl{d}}(\sigma(F_{h}))=0$ unless
\[
  (\gamma,p)=o(x^{(\eta,r)}_{1,1})=\cdots=  o(x^{(\eta,r)}_{n,1})= o(x^{(\eta,r)}_{n-1,1})=
  o(x^{(\eta,r)}_{n-2,2})=\cdots=o(x^{(\eta,r)}_{j+1,2}).
\]
If $o(x^{(\eta,r)}_{j,2})\ne (\gamma,p)$, then $o(x^{(\eta,r)}_{j,2})>(\gamma,p)$ and so $\phi_{\unl{d}}(\sigma(F_{h}))=0$
due to $\zeta\left(\frac{x^{(\eta,r)}_{j+1,2}}{x^{(\eta,r)}_{j,2}}\right)$.

\end{itemize}

\item
Case 2: $\gamma=[1,n,2]$. Suppose $o(x^{(\eta,r)}_{1,1})=(\gamma,p)$. If \eqref{fsunld-D} holds for any
$(\beta,s)<(\gamma,p)$, then $\eta=\gamma$ and $\phi_{\unl{d}}(\sigma(F))=0$ unless we plug all the variables
$x^{(\eta,r)}_{*,*}$ into $\Psi(E_{\gamma,r_{\gamma}(h,p)})$.
\smallskip

\item
Case 3: $\gamma> [1,n,2]$.
If \eqref{fsunld-D} holds for any $(\beta,s)<([2],1)$, then we can use the induction assumption for $D_{n-1}$ to conclude
$\phi_{\unl{d}}(\sigma(F))=0$ unless \eqref{fsunld-D} holds for all $(\gamma,p)$.

\end{itemize}
This completes the proof.
\end{proof}

Completely analogously to Propositions~\ref{spekpC} and~\ref{vanishC}, one can use Proposition \ref{prop:phidEh-D} to evaluate
$\phi_{\unl{d}'}(\Psi(E_{h}))$ for any $\unl{d}'\leq \unl{d}\in\text{KP}(\unl{k})$ and $h\in H_{\unl{k},\unl{d}}$:

\begin{Prop}\label{vanishD}
(a) For any $h\in H_{\unl{k},\unl{d}}$, we have
\begin{equation}
\label{speehD}
  \phi_{\unl{d}}(\Psi(E_{h}))\doteq
  \prod_{\beta,\beta'\in \Delta^+}^{\beta<\beta'}G_{\beta,\beta'}\cdot
  \prod_{\beta\in\Delta^{+}}\left(\langle 1\rangle_{v}^{d_{\beta}(\abs{\beta}-1)}\cdot G_{\beta}\right)\cdot
  \prod_{\beta\in\Delta^{+}} P_{\lambda_{h,\beta}},
\end{equation}
where the factors $\{P_{\lambda_{h,\beta}}\}_{\beta\in\Delta^{+}}$ are given by \eqref{hlp-C}, the terms
$G_{\beta,\beta'},G_{\beta}$ are products of linear factors $w_{\beta,s}$ and $w_{\beta,s}-v^{\BZ}w_{\beta',s'}$
which are independent of $h\in H_{\unl{k},\unl{d}}$ and are $\mathfrak{S}_{\unl{d}}$-symmetric.

\noindent
(b) Lemma {\rm \ref{vanish}} is valid for type $D_{n}$, with $\phi_{\unl{d}}$ of \eqref{spe-D-1}--\eqref{spe-D-3}.
\end{Prop}

\begin{Rk}\label{rk:gbeta-D}
The factors $\{G_{\beta}\}_{\beta\in\Delta^{+}}$ featuring in \eqref{speehD} are explicitly given by:
\begin{itemize}[leftmargin=0.7cm]

\item
If $\beta\neq [i,n,j]$ with $1\leq i< j\leq n-2$, then
\begin{equation}
\label{eq:formulagbeta-D-1}
  G_{\beta}=
  \prod_{1\leq s\leq d_{\beta}} w_{\beta,s}^{\abs{\beta}-1}
  \prod_{1\leq s\neq s'\leq d_{\beta}} (w_{\beta,s}-v^{2}w_{\beta,s'})^{\abs{\beta}-1}.
\end{equation}

\item
If $\beta=[i,n,j]$ with $1\leq i<j \leq  n-2$, then
\begin{equation}
\label{eq:formulagbeta-D-2}
\begin{aligned}
  G_{\beta}=
  \prod_{1\leq s\leq d_{\beta}} w_{\beta,s}^{\abs{\beta}-1}
  &\prod_{1\leq s\neq s'\leq d_{\beta}} (w_{\beta,s}-v^{2}w_{\beta,s'})^{\abs{\beta}-1}\times\\
  &\prod_{1\leq s\neq s'\leq d_{\beta}} \prod_{\ell=j}^{n-2}
    \left\{ (w_{\beta,s}-v^{2n-2\ell}w_{\beta,s'})(w_{\beta,s}-v^{2n-2\ell-4}w_{\beta,s'}) \right\}.
\end{aligned}
\end{equation}

\end{itemize}
\end{Rk}

The factors $G_{\beta,\beta'}$ featuring in \eqref{speehD} can be computed recursively, which  shall be used
in the proof of our next result:

\begin{Prop}\label{spanD}
Lemma {\rm \ref{span}}  is valid for type $D_{n}$, with $\phi_{\unl{d}}$ of \eqref{spe-D-1}--\eqref{spe-D-3}.
\end{Prop}

\begin{proof}
The proof closely follows that of Proposition~\ref{spanC}. In particular, for any pair $\beta\leq \beta'$, let us consider
\begin{equation*}
  \unl{d}=
  \begin{cases}
    \big\{ d_{\beta}=2,\, \mathrm{and}\  d_{\gamma}=0\ \mathrm{for\ other}\ \gamma \big\} & \mathrm{if}\ \beta=\beta'\\
    \big\{ d_{\beta}=d_{\beta'}=1,\, \mathrm{and}\  d_{\gamma}=0\ \mathrm{for\ other}\ \gamma \big\} & \mathrm{if}\ \beta<\beta'
  \end{cases}
\end{equation*}
as before, and let $\unl{d}\in\text{KP}(\unl{k})$. Similarly to $C$-type, it then suffices to show that for any
$F\in S_{\unl{k}}$, $\phi_{\unl{d}}(F)$ is divisible by $G_{\beta,\beta'}$ if $\phi_{\unl{d}'}(F)=0$ for any
$\unl{d}'<\unl{d}$, where we use $G_{\beta,\beta}=G_{\beta}$.
Using type $A_n$ results and the induction, we still have the following cases to analyze:
\begin{itemize}[leftmargin=0.7cm]

\item
$\beta=\beta'=[1,n,j]$ with $2\leq j\leq n-2$.

According to Remark \ref{rk:gbeta-D}, we have
\begin{equation*}
  G_{\beta}=w_{\beta,1}w_{\beta,2}(w_{\beta,1}-v^{\pm 2}w_{\beta,2})
  (w_{\beta,1}-v^{\pm(2n-2j)}w_{\beta,2})(w_{\beta,1}-v^{\pm(2n-2j-4)}w_{\beta,2})\cdot G_{\alpha}
\end{equation*}
for $\alpha=[1,n,j+1]$. For any $F\in S_{\unl{k}}$, as we specialize all the variables but
$\{x^{(\beta,1)}_{j,2},x^{(\beta,2)}_{j,2}\}$, the wheel conditions involving the specialized variables produce
the factor $G_{\alpha}$ by the induction assumption. As we specialize $x^{(\beta,1)}_{j,2}$, the wheel conditions
\[
  x^{(\beta,1)}_{j,2}=v^{2}x^{(\beta,1)}_{j,1}=vx^{(\beta,1)}_{j-1,1}, \qquad
  x^{(\beta,1)}_{j,1}=v^{2}x^{(\beta,1)}_{j,2}=vx^{(\beta,1)}_{j+1,1}
\]
contribute the factor $B_{\beta}/B_{\alpha}=(w_{\beta,1}-v^{-2n+2j}w'_{\beta,1})(w_{\beta,1}-v^{-2n+2j+4}w'_{\beta,1})$
to the first step of the specialization $\phi^{(1)}_{\beta}(F)$, cf.~\eqref{spe-D-2}. Then in the second step of the
specialization, cf.~\eqref{spe-D-3}, we divide by $B_{\beta}/B_{\alpha}$ and specialize
$w'_{\beta,1} \mapsto w_{\beta,1}, w'_{\beta,2}\mapsto w_{\beta,2}$. Then the  wheel conditions
\begin{equation*}
  x^{(\beta,1)}_{j+1,2}=v^{2}x^{(\beta,2)}_{j+1,2}=vx^{(\beta,1)}_{j,2}, \ \
  x^{(\beta,1)}_{j,2}=v^{2}x^{(\beta,2)}_{j,1}=vx^{(\beta,2)}_{j-1,1}, \ \
  x^{(\beta,2)}_{j,1}=v^{2}x^{(\beta,1)}_{j,2}=vx^{(\beta,2)}_{j+1,1},
\end{equation*}
contribute the factor $(w_{\beta,1}-v^{ 2}w_{\beta,2})(w_{\beta,1}-v^{2n-2j}w_{\beta,2})(w_{\beta,1}-v^{2n-2j-4}w_{\beta,2})$
to $\phi_{\unl{d}}(F)$. Thus, from the symmetry, we see that $\phi_{\unl{d}}(F)$ is indeed divisible by $G_{\beta}$.

\item $\beta=[1,i]$, $\beta'=[1,n,n-1]$.

If $i\leq n-3$, then $G_{\beta,\beta'}=G_{[1,i],[1,n-2]}$, so  $\phi_{\unl{d}}(F)$
is divisible by $G_{\beta,\beta'}$ due to type $A_n$.

If $i=n-2$, then
\begin{equation*}
  G_{\beta,\beta'}=(w_{\beta,1}-w_{\beta',1})\cdot G_{\beta,\alpha} \qquad \mathrm{with} \quad \alpha=[1,n-1].
\end{equation*}
As we specialize all the variables but $x^{(\beta',1)}_{n,1}$, the wheel conditions involving the specialized variables
produce the factor $G_{\beta,\alpha}$ by the induction assumption. As we specialize $x^{(\beta',1)}_{n,1}$, consider
$\unl{d}'=\{d'_{[1,n-1]}=d'_{[1,n]}=1,\ \text{and}\ d'_{\gamma}=0\ \text{for other}\ \gamma\}$. Then $\unl{d}'<\unl{d}$
and $\phi_{\unl{d}'}(F)=0$ implies that $\phi_{\unl{d}}(F)$ is divisible by $w_{\beta,1}-w_{\beta',1}$,
and hence by $G_{\beta,\beta'}$.

If $i=n-1$, then
\begin{equation*}
  G_{\beta,\beta'}=(w_{\beta,1}-v^{-2}w_{\beta',1})\cdot G_{\beta,\alpha} \qquad \mathrm{with} \quad \alpha=[1,n].
\end{equation*}
By the induction assumption and the wheel conditions $F=0$ at
$x^{(\beta',1)}_{n-1,1}=v^{2}x^{(\beta,1)}_{n-1,1}=vx^{(\beta,1)}_{n-2,1}$,
we see that $\phi_{\unl{d}}(F)$ is divisible by $G_{\beta,\beta'}$.

If $i=n$, then
\begin{equation*}
  G_{\beta,\beta'}=(w_{\beta,1}-v^{\pm 2}w_{\beta',1})\cdot G_{\alpha,\alpha'}
  \qquad \mathrm{with} \quad \alpha=[1,n-2], \alpha'=[1,n].
\end{equation*}
By the induction assumption and wheel conditions $F=0$ at
$x^{(\beta,1)}_{n-2,1}=v^{2}x^{(\beta',1)}_{n-2,1}=vx^{(\beta',1)}_{n-1,1}$ and
$x^{(\beta',1)}_{n,1}=v^{2}x^{(\beta,1)}_{n,1}=vx^{(\beta,1)}_{n-2,1}$
we see that $\phi_{\unl{d}}(F)$ is divisible by $G_{\beta,\beta'}$.

\item $\beta=[1,i]$, $\beta'=[1,n,n-2]$.

If $i\leq n-4$, then $G_{\beta,\beta'}=G_{[1,i],[1,i+1]}$, so  $\phi_{\unl{d}}(F)$
is divisible by $G_{\beta,\beta'}$, due to type $A_n$.

If $i=n-3$, then
\begin{equation*}
  G_{\beta,\beta'}=(w_{\beta,1}-v^{- 2}w_{\beta',1})\cdot G_{\beta,\alpha}
  \qquad \mathrm{with} \quad \alpha=[1,n,n-1].
\end{equation*}
Consider $\unl{d}'=\{d'_{[1,n-2]}=d'_{[1,n,n-1]}=1,\, \text{and}\, d'_{\gamma}=0\ \text{for other}\ \gamma\}$.
Then $\unl{d'}<\unl{d}$ and $\phi_{\unl{d}'}(F)=0$ implies that $\phi_{\unl{d}}(F)$ is divisible by
$w_{\beta,1}-v^{- 2}w_{\beta',1}$, and hence by $G_{\beta,\beta'}$.

If $i=n-2$, then
\begin{equation*}
  G_{\beta,\beta'}=(w_{\beta,1}-v^{-4}w_{\beta',1})\cdot G_{\beta,\alpha}
  \qquad \mathrm{with} \quad \alpha=[1,n,n-1].
\end{equation*}
From induction assumption and the wheel condition $F=0$ at
$x^{(\beta',1)}_{n-2,2}=v^{2}x^{(\beta,1)}_{n-2,1}=vx^{(\beta,1)}_{n-3,1}$
we see that $\phi_{\unl{d}}(F)$ is divisible by $G_{\beta,\beta'}$.

If $i=n-1$, then
\begin{equation*}
  G_{\beta,\beta'}=(w_{\beta,1}-w_{\beta',1})(w_{\beta,1}-v^{-4}w_{\beta',1})\cdot G_{\beta,\alpha}
  \qquad \mathrm{with} \quad \alpha=[1,n,n-1].
\end{equation*}
Due to the induction assumption and the wheel conditions $F=0$ at
$x^{(\beta',1)}_{n-2,2}=v^{2}x^{(\beta,1)}_{n-2,1}=vx^{(\beta,1)}_{n-3,1}$ and
$x^{(\beta,1)}_{n-2,1}=v^{2}x^{(\beta',1)}_{n-2,2}=vx^{(\beta,1)}_{n-1,1}$,
we see that $\phi_{\unl{d}}(F)$ is divisible by $G_{\beta,\beta'}$.

If $i=n$, then
\begin{equation*}
  G_{\beta,\beta'}=(w_{\beta,1}-v^{\pm 2}w_{\beta',1})\cdot G_{\alpha,\beta'}
  \qquad \mathrm{with} \quad \alpha=[1,n-2].
\end{equation*}
By the induction assumption and the wheel conditions $F=0$ at
$x^{(\beta',1)}_{n,1}=v^{2}x^{(\beta,1)}_{n,1}=vx^{(\beta,1)}_{n-2,1}$ and
$x^{(\beta,1)}_{n-2,1}=v^{2}x^{(\beta',1)}_{n-2,1}=vx^{(\beta,1)}_{n,1}$
we see that $\phi_{\unl{d}}(F)$ is divisible by $G_{\beta,\beta'}$.

\item $\beta=[1,i]$,  $\beta'=[1,n,j]$ with $2\leq j\leq n-3$.

If $i\leq j-2$, then $G_{\beta,\beta'}=G_{[1,i],[1,j-1]}$, and so $\phi_{\unl{d}}(F)$ is divisible by $G_{\beta,\beta'}$.

If $i=j-1$ and $j\geq 3$, then
\begin{equation*}
  G_{\beta,\beta'}=(w_{\beta,1}-v^{\pm 2}w_{\beta',1})(w_{\beta,1}-v^{-2n+2j+2}w_{\beta',1})\cdot G_{[1,j-2],\beta'}.
\end{equation*}
As we specialize the remaining variable $x^{(\beta,1)}_{j-1,1}$, the wheel conditions $F=0$ at
$x^{(\beta',1)}_{j-1,1}=v^{2}x^{(\beta,1)}_{j-1,1}=vx^{(\beta',1)}_{j,1}$ and
$x^{(\beta,1)}_{j-1,1}=v^{2}x^{(\beta',1)}_{j-1,1}=vx^{(\beta',1)}_{j-2,1}$
contribute the factor $(w_{\beta,1}-v^{\pm 2}w_{\beta',1})$ into $\phi_{\unl{d}}(F)$. Consider
  $\unl{d}'=\{d'_{[1,j]}=d'_{[1,n,j+1]}=1,\, \text{and}\ d'_{\gamma}=0\ \text{for other}\ \gamma\}$.
Then $\unl{d}'<\unl{d}$ and $\phi_{\unl{d}'}(F)=0$ implies that $\phi_{\unl{d}}(F)$ is divisible by
$w_{\beta,1}-v^{-2n+2j+2}w_{\beta',1}$. Combining this with the induction assumption we see that
$\phi_{\unl{d}}(F)$ is divisible by $G_{\beta,\beta'}$.

If $i=j-1$ and $j=2$, then $\beta=[1]$, and $G_{\beta,\beta'}=(w_{\beta,1}-v^{-2n+6}w_{\beta',1})\cdot G_{[1],[1,n,3]}$.
Consider $\unl{d}'=\{d'_{[1,2]}=d'_{[1,n,3]}=1,\, \text{and}\ d'_{\gamma}=0\ \ \text{for other}\ \gamma\}$. Then
$\unl{d}'<\unl{d}$ and $\phi_{\unl{d}'}(F)=0$ implies that $\phi_{\unl{d}}(F)$ is divisible by
$w_{\beta,1}-v^{-2n+6}w_{\beta',1}$. Combining this with the induction assumption we see that
$\phi_{\unl{d}}(F)$ is divisible by $G_{\beta,\beta'}$.

If $i=j$, then $G_{\beta,\beta'}=(w_{\beta,1}-v^{-2n+2j}w_{\beta',1})\cdot G_{\beta,[1,n,j+1]}$. From the wheel condition
$F=0$ at $x^{(\beta',1)}_{j,2}=v^{2}x^{(\beta,1)}_{j,1}=vx^{(\beta,1)}_{j-1,1}$, we see that $\phi_{\unl{d}}(F)$ is divisible
by $(w_{\beta,1}-v^{-2n+2j}w_{\beta',1})$, which together with the induction assumption implies the divisibility by $G_{\beta,\beta'}$.

If $i\geq j+1$, then
\begin{equation*}
  G_{\beta,\beta'}=(w_{\beta,1}-v^{-2n+2j}w_{\beta',1})(w_{\beta,1}-v^{-2n+2j+4}w_{\beta',1})\cdot G_{\beta,\alpha}
  \quad \mathrm{with} \quad  \alpha=[1,n,j+1].
\end{equation*}
As we specialize all the variables but $x^{(\beta',1)}_{j,2}$, the wheel conditions involving the specialized variables
produce the factor $G_{\beta,\alpha}$ by the induction assumption. As we specialize $x^{(\beta',1)}_{j,2}$, the wheel
conditions at $x^{(\beta',1)}_{j,2}=v^{2}x^{(\beta',1)}_{j,1}=vx^{(\beta',1)}_{j-1,1}$ and
$x^{(\beta',1)}_{j,1}=v^{2}x^{(\beta',1)}_{j,2}=vx^{(\beta',1)}_{j+1,1}$ contribute the factor
$B_{\beta'}/B_{\alpha}=(w_{\beta',1}-v^{-2n+2j}w'_{\beta',1})(w_{\beta',1}-v^{-2n+2j+4}w'_{\beta',1})$
to the first step of the specialization $\phi^{(1)}_{\beta}(F)$, cf.~\eqref{spe-D-2}. Then in the second step
of the specialization, cf.~\eqref{spe-D-3}, we divide by $B_{\beta'}/B_{\alpha}$ and specialize
$w'_{\beta',1}\mapsto w_{\beta',1}$. The  wheel conditions $F=0$ at
$x^{(\beta',1)}_{j,2}=v^{2}x^{(\beta,1)}_{j,1}=vx^{(\beta,1)}_{j-1,1}$ and
$x^{(\beta,1)}_{j,1}=v^{2}x^{(\beta',1)}_{j,2}=vx^{(\beta,1)}_{j+1,1}$ contribute the extra factor
$(w_{\beta,1}-v^{-2n+2j}w_{\beta',1})(w_{\beta,1}-v^{-2n+2j+4}w_{\beta',1})$ into $\phi_{\unl{d}}(F)$.
Thus $\phi_{\unl{d}}(F)$ is divisible by $G_{\beta,\beta'}$.

\item $\beta=[1,n,n-1]$, $\beta'=[1,n,j]$ with $2\leq j\leq n-2$.

If $j=n-2$, then
\begin{equation*}
  G_{\beta,\beta'}= (w_{\beta,1}-w_{\beta',1})(w_{\beta,1}-v^{- 2}w_{\beta',1})(w_{\beta,1}-v^{-4}w_{\beta',1})\cdot G_{\beta}.
\end{equation*}
By the induction assumption and the wheel conditions $F=0$ at
\begin{equation*}
\begin{aligned}
   x^{(\beta,1)}_{n-2,1}=v^{2}x^{(\beta',1)}_{n-2,2}=vx^{(\beta,1)}_{n,1},\ \
   x^{(\beta',1)}_{n-2,2}=v^{2}x^{(\beta,1)}_{n-2,1}=vx^{(\beta,1)}_{n-3,1},\ \
   x^{(\beta',1)}_{n-1,2}=v^{2}x^{(\beta,1)}_{n-1,1}=vx^{(\beta',1)}_{n-2,2}
\end{aligned}
\end{equation*}
we see that $\phi_{\unl{d}}(F)$ is divisible by $G_{\beta,\beta'}$.

If $j<n-2$,  then
\begin{equation*}
\begin{aligned}
  G_{\beta,\beta'}=(w_{\beta,1}-v^{-2n+2j}w_{\beta',1})(w_{\beta,1}-v^{-2n+2j+4}w_{\beta',1})\cdot G_{\beta,[1,n,j+1]},
\end{aligned}
\end{equation*}
and we can apply the same arguments as for $(\beta,\beta')=([1,j+1],[1,n,j])$.

\item $\beta=[1,n,k]$, $\beta'=[1,n,j]$.

If $j>2$, then $G_{\beta,\beta'}=(w_{\beta,1}-v^{\pm 2}w_{\beta',1})\cdot G_{[2,n,k],[2,n,j]}$, and so
$\phi_{\unl{d}}(F)$ is divisible by $G_{\beta,\beta'}$ due to the induction assumption and wheel conditions
at $x^{(\beta,1)}_{2,1}=v^{2}x^{(\beta',1)}_{2,1}=vx^{(\beta',1)}_{1,1}$ and
$x^{(\beta',1)}_{2,1}=v^{2}x^{(\beta,1)}_{2,1}=vx^{(\beta,1)}_{1,1}$.

If $j=2$ and $k>3$, then
  $G_{\beta,\beta'}=(w_{\beta,1}-v^{-2n+4}w_{\beta',1})(w_{\beta,1}-v^{-2n+8}w_{\beta',1})\cdot G_{\beta,[1,n,3]}$.
From the wheel conditions $F=0$ at
$x^{(\beta',1)}_{2,2}=v^{2}x^{(\beta,1)}_{2,1}=vx^{(\beta,1)}_{1,1}$,
$x^{(\beta,1)}_{2,1}=v^{2}x^{(\beta',1)}_{2,2}=vx^{(\beta,1)}_{3,1}$,  and
the induction assumption we see that $\phi_{\unl{d}}(F)$ is divisible by $G_{\beta,\beta'}$.

If $j=2$ and $k=3$,  then
\begin{equation*}
\begin{aligned}
  G_{\beta,\beta'}=
  (w_{\beta,1}-v^{\pm 2}w_{\beta',1})(w_{\beta,1}-v^{2n-6}w_{\beta',1})(w_{\beta,1}-v^{2n-10}w_{\beta',1})\cdot G_{[1,n,4],[1,n,2]}.
\end{aligned}
\end{equation*}
Due to the induction assumption and the wheel conditions $F=0$ at
\begin{equation*}
\begin{aligned}
   &x^{(\beta,1)}_{3,2}=v^{2}x^{(\beta',1)}_{3,2}=vx^{(\beta',1)}_{4,2},\qquad
   v^2x^{(\beta,1)}_{3,2}=x^{(\beta',1)}_{3,2}=vx^{(\beta',1)}_{2,2},\\
   & x^{(\beta,1)}_{3,2}=v^{2}x^{(\beta',1)}_{3,1}=vx^{(\beta',1)}_{2,1},\qquad
   v^2x^{(\beta,1)}_{3,2}=x^{(\beta',1)}_{3,1}=vx^{(\beta',1)}_{4,1},
\end{aligned}
\end{equation*}
we see that $\phi_{\unl{d}}(F)$ is divisible by $G_{\beta,\beta'}$.

\item $\beta'\geq [2]>\beta$.

If $\beta=[1,i]$ and $\beta'=[2,n,j]$, then $G_{\beta,\beta'}=(w_{\beta,1}-w_{\beta',1})\cdot G_{[2,i],\beta'}$. Consider
$\unl{d}'=\{d'_{[1,n,j]}=d'_{[2,i]}=1,\ \text{and}\ d'_{\gamma}=0\ \text{for other}\ \gamma\}$. Then $\phi_{\unl{d}}(F)$
is divisible by $G_{\beta,\beta'}$ due to the induction assumption and $\phi_{\unl{d}'}(F)=0$.

If $\beta=[1,n,3]$ and $\beta'=[2,j]$, then
$G_{\beta,\beta'}=(w_{\beta,1}-v^{\pm 2}w_{\beta',1})(w_{\beta,1}-v^{2n-6}w_{\beta',1})\cdot G_{\beta,[3,j]}$. Consider
$\unl{d}'=\{d'_{[1,n,2]}=d'_{[3,j]}=1,\, \text{and}\ d'_{\gamma}=0\ \text{for other}\ \gamma\}$, so that $\unl{d}'<\unl{d}$.
Then $\phi_{\unl{d}}(F)$ is divisible by $G_{\beta,\beta'}$ due to the induction assumption, the condition $\phi_{\unl{d}'}(F)=0$,
and wheel conditions at $x^{(\beta,1)}_{2,1}=v^{2}x^{(\beta',1)}_{2,1}=vx^{(\beta,1)}_{3,1}$,
$v^{2}x^{(\beta,1)}_{2,1}=x^{(\beta',1)}_{2,1}=vx^{(\beta,1)}_{1,1}$.

If $\beta=[1,n,i]$ and $\beta'=[2,n,j]$ with $i>j$, then $G_{\beta,\beta'}=(w_{\beta,1}-w_{\beta',1})\cdot G_{[2,n,i],\beta'}$.
Consider $\unl{d}'=\{d'_{[1,n,j]}=d'_{[2,n,i]}=1,\, \text{and}\ d'_{\gamma}=0\ \text{for other}\ \gamma\}$, so that $\unl{d}'<\unl{d}$.
Then $\phi_{\unl{d}}(F)$ is divisible by $G_{\beta,\beta'}$ due to the induction assumption and the condition $\phi_{\unl{d}'}(F)=0$.

For all other cases, the divisibility of $\phi_{\unl{d}}(F)$ by $G_{\beta,\beta'}$ follows from the induction assumption
and proper count of wheel conditions similarly to the cases above.

\end{itemize}
This completes our proof.
\end{proof}

Combining Propositions~\ref{vanishD} and~\ref{spanD}, we immediately obtain the shuffle algebra realization and
the PBWD theorem for $\qld$:

\begin{Thm}\label{shufflePBWD-D}
(a) $\Psi\colon \qld \,\iso\, S$ of~\eqref{eq:Psi-homom} is a $\BQ(v)$-algebra isomorphism.

\noindent
(b) For any choices of $s_k$ and $\lambda_k$ in the definition~\eqref{rootvector3} of quantum root vectors $E_{\beta,s}$,
the ordered PBWD monomials $\{E_{h}\}_{h\in H}$ from \eqref{PBWDbases} form a $\BQ(v)$-basis of $\qld$.
\end{Thm}


\subsection{Shuffle algebra realization of the Lusztig integral form in type $D$}\label{lusd}

For any $\unl{k}\in \BN^n$, consider the $\BZ[v,v^{-1}]$-submodule $\mathbf{S}_{\unl{k}}$ of $S_{\unl{k}}$
consisting of rational functions $F$ satisfying the following two conditions:
\begin{enumerate}[leftmargin=1cm]

\item
If $f$ denotes the numerator of $F$ from \eqref{polecondition}, then
\begin{equation}\label{lid-1}
  f\in \BZ[v,v^{-1}][\{x_{i,r}^{\pm 1}\}_{1\leq i\leq n}^{1\leq r\leq k_{i}}]^{\mathfrak{S}_{\underline{k}}}.
\end{equation}

\item
For any $\unl{d}\in\text{KP}(\underline{k})$, the specialization $\phi_{\unl{d}}(F)$ is divisible by the product
\begin{equation}\label{lid-2}
  \prod_{\beta\in\Delta^{+}} \langle 1\rangle_{v}^{d_{\beta}(|\beta|-1)}.
\end{equation}

\end{enumerate}

Define $\mathbf{S}\coloneqq \bigoplus_{\unl{k}\in\BN^{n}}\mathbf{S}_{\unl{k}}$ and recall the Lusztig integral form
$\integraldl$ from Definition \ref{def:lusintegral}. Then, similarly to Proposition~\ref{lintegralC}, we have:

\begin{Prop}\label{lintegralD}
$\Psi(\integraldl) \subset \mathbf{S}$.
\end{Prop}

\begin{proof}
For any $m\in\BN$, $1\leq i_{1},\dots,i_{m}\leq n$, $r_{1},\dots,r_{m}\in\BZ$, $\ell_{1},\dots,\ell_{m}\in\BN$, let
\[
  F\coloneqq \Psi\big(\mathbf{E}^{(\ell_{1})}_{i_{1},r_{1}}\cdots \mathbf{E}^{(\ell_{m})}_{i_{m},r_{m}}\big),
\]
and $f$ be the numerator of $F$ from \eqref{polecondition}. The validity of the condition \eqref{lid-1} for $f$ follows
from~\eqref{rank1-power}. To verify the validity of the divisibility \eqref{lid-2}, we need to show that for any
$\beta\in \Delta^+$ and $1\leq s\leq d_{\beta}$, the total contribution of $\phi_{\unl{d}}$-specializations of the
$\zeta$-factors between the variables $\{x^{(\beta,s)}_{i,t}\}_{i\in\beta}^{1\leq t\leq \nu_{\beta,i}}$ of $f$ is divisible
by $\langle 1\rangle_{v}^{|\beta|-1}$. It suffices to treat only the cases $\beta=[i,n,j]$ with $1\leq i <j \leq n-2$,
since the other cases are treated completely analogously to type $A_n$. Similarly to the proof of Proposition~\ref{lintegralC},
we shall use the notation $o(x^{(*,*)}_{*,*})=q$ if a variable $x^{(*,*)}_{*,*}$ is plugged into
$\Psi(\mathbf{E}^{(\ell_{q})}_{i_{q},r_{q}})$.

According to~\eqref{spe-D-2}, the $\phi_{\unl{d}}$-specialization of any summand in $F$ vanishes unless
\begin{gather*}
  o(x^{(\beta,s)}_{i,1})\geq o(x^{(\beta,s)}_{i+1,1})\geq \cdots \geq o(x^{(\beta,s)}_{n-2,1})\geq o(x^{(\beta,s)}_{n-1,1})\
    \&\  o(x^{(\beta,s)}_{n,1}),\\
  o(x^{(\beta,s)}_{n-2,2})\geq o(x^{(\beta,s)}_{n-3,2})\geq \cdots \geq o(x^{(\beta,s)}_{j,2}).
\end{gather*}
Since $o(x^{(\beta,s)}_{i,t})\ne o(x^{(\beta,s)}_{i',t'})$ for $i\ne i'$, we have strict inequalities:
\begin{gather*}
  o(x^{(\beta,s)}_{i,1})> o(x^{(\beta,s)}_{i+1,1})> \cdots > o(x^{(\beta,s)}_{n-2,1})> o(x^{(\beta,s)}_{n-1,1})\
    \&\  o(x^{(\beta,s)}_{n,1}),\\
  o(x^{(\beta,s)}_{n-2,2})> o(x^{(\beta,s)}_{n-3,2})> \cdots > o(x^{(\beta,s)}_{j,2}).
\end{gather*}
With symmetry between the variables $x^{(\beta,s)}_{n-1,1}, x^{(\beta,s)}_{n,1}$, we may assume that
$o(x^{(\beta,s)}_{n-1,1})>o(x^{(\beta,s)}_{n,1})$ in the following analysis. We have the following three cases to consider:
\begin{itemize}[leftmargin=0.7cm]

\item
if $o(x^{(\beta,s)}_{n-2,2})> o(x^{(\beta,s)}_{n-1,1})>o(x^{(\beta,s)}_{n,1})$, then the $\zeta$-factors
  $\zeta\left(\frac{x^{(\beta,s)}_{n-1,1}}{x^{(\beta,s)}_{n-2,2}}\right)
   \zeta\left(\frac{x^{(\beta,s)}_{n,1}}{x^{(\beta,s)}_{n-2,2}}\right)$
contribute $(w_{\beta,s}-w'_{\beta,s})^{2}$ to the $\phi^{(1)}_{\beta}$-specialization of the summand, and consecutively
$(w_{\beta,s}-w'_{\beta,s})$ to the $\phi_{\beta}$-specialization (as $B_{\beta}$ of~\eqref{Bfactor-D} contains only one
factor $(w_{\beta,s}-w'_{\beta,s})$), so that the $\phi_{\unl{d}}$-specialization of the corresponding summand in $F$ vanishes;

\item
if $o(x^{(\beta,s)}_{n-1,1})>o(x^{(\beta,s)}_{n-2,2})>o(x^{(\beta,s)}_{n,1})$, then
\[
  o(x^{(\beta,s)}_{i,1})> \cdots > o(x^{(\beta,s)}_{n-2,1})> o(x^{(\beta,s)}_{n-1,1})>o(x^{(\beta,s)}_{n-2,2})>
  o(x^{(\beta,s)}_{n-3,2})>\cdots > o(x^{(\beta,s)}_{j,2}),
\]
so that the $\zeta$-factors
\begin{equation}
\label{eq:Bbetafactor-D-L}
  \prod^{n-2}_{\ell=j}\left\{\zeta\left(\frac{x^{(\beta,s)}_{\ell,2}}{x^{(\beta,s)}_{\ell-1,1}}\right)
  \zeta\left(\frac{x^{(\beta,s)}_{\ell,2}}{x^{(\beta,s)}_{\ell,1}}\right)
  \zeta\left(\frac{x^{(\beta,s)}_{\ell,2}}{x^{(\beta,s)}_{\ell+1,1}}\right)\right\}
\end{equation}
contribute $B_{\beta}$ to the $\phi^{(1)}_{\beta}$-specialization of the summand, and thus  the $\phi_{\unl{d}}$-specialization
of the corresponding summand in $F$ vanishes due to the remaining $\zeta$-factor
$\zeta\left(\frac{x^{(\beta,s)}_{n,1}}{x^{(\beta,s)}_{n-2,2}}\right)$;

\item
if $ o(x^{(\beta,s)}_{n-1,1})>o(x^{(\beta,s)}_{n,1})>o(x^{(\beta,s)}_{n-2,2})$, then
\[
  o(x^{(\beta,s)}_{i,1})> \cdots > o(x^{(\beta,s)}_{n-2,1})> o(x^{(\beta,s)}_{n-1,1})>o(x^{(\beta,s)}_{n,1})>
  o(x^{(\beta,s)}_{n-2,2})> \cdots > o(x^{(\beta,s)}_{j,2}).
\]
The  $\zeta$-factors of \eqref{eq:Bbetafactor-D-L}  contribute $B_{\beta}$ to the  $\phi^{(1)}_{\beta}$-specialization,
and the remaining  $\zeta$-factors
\[
  \left\{\prod^{n-3}_{\ell=j}\zeta\left(\frac{x^{(\beta,s)}_{\ell,2}}{x^{(\beta,s)}_{\ell+1,2}}\right)\right\} \cdot
  \left\{\prod_{\ell=i+1}^{n-1} \zeta\left(\frac{x^{(\beta,s)}_{\ell,1}}{x^{(\beta,s)}_{\ell-1,1}}\right)\right\} \cdot
  \zeta\left(\frac{x^{(\beta,s)}_{n-2,2}}{x^{(\beta,s)}_{n,1}}\right)
  \zeta\left(\frac{x^{(\beta,s)}_{n,1}}{x^{(\beta,s)}_{n-2,1}}\right)
\]
contribute $\langle 1\rangle_{v}^{|\beta|-1}$ to the $\phi_{\unl{d}}$-specialization of the corresponding summand in $F$.
\end{itemize}

This completes our proof.
\end{proof}

Recall the normalized divided powers \eqref{eq:nrv-D} of the quantum root vectors
$\{\tilde{\mathbf{E}}_{\beta,s}^{\pm,(k)}\}^{k\in\BN}_{\beta\in\Delta^{+},s\in\BZ}$ and the ordered monomials
$\{\tilde{\mathbf{E}}^{\epsilon}_{h}\}_{h\in H}$ of  \eqref{eq:Lus-pbwd}. For $\epsilon\in\{\pm\}$, let
$\mathbf{S}^{\epsilon}_{\unl{k}}$ be the $\BZ[v,v^{-1}]$-submodule of $\mathbf{S}_{\unl{k}}$ spanned by
$\{\Psi(\tilde{\mathbf{E}}^{\epsilon}_{h})\}_{h\in H_{\unl{k}}}$. Then, the following analogue of Proposition~\ref{span-Lus-C} holds:

\begin{Prop}\label{span-Lus-D}
For any $F\in \mathbf{S}_{\unl{k}}$ and $\unl{d}\in\mathrm{KP}(\unl{k})$, if $\phi_{\unl{d}'}(F)=0$ for all
$\unl{d}'\in \mathrm{KP}(\unl{k})$ such that $\unl{d}'<\unl{d}$, then there exists
$F_{\unl{d}}\in \mathbf{S}^{\epsilon}_{\unl{k}}$ such that $\phi_{\unl{d}}(F)=\phi_{\unl{d}}(F_{\unl{d}})$
and $\phi_{\unl{d}'}(F_{\unl{d}})=0$ for all $\unl{d}'<\unl{d}$.
\end{Prop}

Combining Propositions \ref{lintegralD} and~\ref{span-Lus-D}, we obtain the following upgrade of Theorem \ref{shufflePBWD-D}:

\begin{Thm}\label{lusthm-D}
(a) The $\BQ(v)$-algebra isomorphism $\Psi\colon \qld \,\iso\, S$ of Theorem~\ref{shufflePBWD-D}(a) gives rise to
a $\BZ[v,v^{-1}]$-algebra isomorphism $\Psi\colon \integraldl \,\iso\, \mathbf{S}$.

\noindent
(b) Theorem {\rm \ref{pbwtheorem-Lus}} holds for $\fg$ of type $D_n$.
\end{Thm}


\subsection{Shuffle algebra realization of the RTT integral form in type $D$}\label{rttd}

For any $\unl{k}\in\BN^{n}$, consider the $\BZ[v,v^{-1}]$-submodule $\mathcal{S}_{\unl{k}}$ of $S_{\unl{k}}$ consisting
of rational functions $F$ satisfying the following two conditions:
\begin{enumerate}[leftmargin=1cm]

\item
If $f$ denotes the numerator of $F$ from \eqref{polecondition}, then
\begin{equation}
\label{rttconstant1-D}
  f\in \langle 1\rangle_{v}^{|\unl{k}|}\cdot
  \BZ[v,v^{-1}][\{x_{i,r}^{\pm 1}\}_{1\leq i\leq n}^{1\leq r\leq k_{i}}]^{\mathfrak{S}_{\underline{k}}},
\end{equation}
where $|\unl{k}|=|(k_1,\ldots,k_n)|:=k_1+\dots+k_n$.

\medskip
\item
$F$ is \textbf{integral} in the sense of \cite[Definition 4.12]{HT24}: the \emph{cross specialization}
\begin{equation*}
  \Upsilon_{\unl{d},\unl{t}}(F)\coloneqq
  \varpi_{\unl{t}}\left(\frac{\phi_{\unl{d}}(F)}{\langle 1\rangle_{v}^{|\unl{k}|}\cdot \prod_{\beta\in\Delta^{+}}G_{\beta}}\right)
\end{equation*}
is divisible by $\prod_{\beta\in\Delta^{+}}^{1\leq r\leq \ell_{\beta}} [t_{\beta,r}]_{v}!$
(note that $v_{\beta}=v$ for any $\beta\in\Delta^{+}$ in type $D_n$) for any $\unl{d}\in\mathrm{KP}(\unl{k})$
and $\unl{t}=\{t_{\beta,r}\}_{\beta\in\Delta^{+}}^{1\leq r\leq \ell_{\beta}}$ satisfying \eqref{verticalpartition-1},
with $\varpi_{\unl{t}}$ of~\eqref{verticalspe} and $G_{\beta}$ of~(\ref{eq:formulagbeta-D-1},~\ref{eq:formulagbeta-D-2});
the divisibility of $\phi_{\unl{d}}(F)$ by $G_{\beta}$ is proved in Proposition \ref{goodD}.

\end{enumerate}
We define $\mathcal{S}:=\bigoplus_{\unl{k}\in\BN^{n}}\mathcal{S}_{\unl{k}}$. Recall the RTT integral form $\integrald$
from Definition \ref{def:rttintegral}. Then, similarly to Proposition~\ref{goodC}, we have:

\begin{Prop}\label{goodD}
$\Psi(\integrald) \subset \mathcal{S}$.
\end{Prop}

\begin{proof}
For any $\epsilon\in\{\pm\}$, $m\in\BN$, $\beta_{1},\dots,\beta_{m}\in\Delta^{+}$, $r_{1},\dots,r_{m}\in\BZ$, let
\[
  F\coloneqq
  \Psi\big(\tilde{\mathcal{E}}^{\epsilon}_{\beta_{1},r_{1}}\cdots \tilde{\mathcal{E}}^{\epsilon}_{\beta_{m},r_{m}}\big),
\]
and $f$ be the numerator of $F$. We set $\unl{k}=\sum_{q=1}^{m} \beta_{q}$. First, we note that
the condition \eqref{rttconstant1-D} follows from Lemma \ref{lem:psirv-D}.

Next, we show that $\phi_{\unl{d}}(F)$ is divisible by $\prod_{\beta\in\Delta^{+}}G_{\beta}$ with $G_{\beta}$
of (\ref{eq:formulagbeta-D-1},~\ref{eq:formulagbeta-D-2}). Similarly to the proof of Proposition \ref{goodC}, we can
expand $\prod_{\ell=1}^{m}\tilde{\mathcal{E}}^{\epsilon}_{\beta_{\ell},r_{\ell}}$ as a linear combination of monomials
$\prod_{\ell=1}^{k}e_{i_{\ell},s_{\ell}}$ over $\BZ[v,v^{-1}]$, with $\unl{k}=\sum^{k}_{\ell=1}\alpha_{i_{\ell}}$, and
prove that each $\phi_{\unl{d}}(\Psi(e_{i_{1},s_{1}}\cdots e_{i_k,s_{k}}))$ is divisible by $G_{\beta}$ for any $\beta\in\Delta^{+}$.
For $\beta=[i,j]$ (with $1\leq i\leq j\leq n$) this follows from \cite[Lemma 3.51]{Tsy18}. It remains to treat the cases
$\beta=[i,n,n-1]$ with $1\leq i\leq n-2$, and $\beta=[i,n,j]$ with $1\leq i<j\leq n-2$. Henceforth, we shall use
the notation $\hat{o}(x^{(*,*)}_{*,*})=q$ if a variable $x^{(*,*)}_{*,*}$ is plugged into $\Psi(e_{i_q,s_q})$ for some $1\leq q\leq k$.
\begin{itemize}[leftmargin=0.7cm]

\item $\beta=[i,n,n-1]$.
Fix any $1\leq s\neq s'\leq d_{\beta}$. We can assume that
\begin{gather*}
  \hat{o}(x^{(\beta,s)}_{i,1})>  \cdots > \hat{o}(x^{(\beta,s)}_{n-2,1})>\hat{o}(x^{(\beta,s)}_{n,1}) \ \&\
  \hat{o}(x^{(\beta,s)}_{n-1,1}),\\
  \hat{o}(x^{(\beta,s')}_{i,1})> \cdots >\hat{o}(x^{(\beta,s')}_{n-2,1})> \hat{o}(x^{(\beta,s')}_{n,1}) \ \&\
  \hat{o}(x^{(\beta,s')}_{n-1,1}),
\end{gather*}
as otherwise the corresponding term is specialized to zero under $\phi_{\unl{d}}$. Using the same analysis as for
the variables \eqref{icase-C} in type $C_n$, we see that the $\phi_{\unl{d}}$-specialization of the $\zeta$-factors
arising from the quadruples
\[
  \left\{x^{(\beta,s)}_{\ell,1},x^{(\beta,s)}_{\ell+1,1},x^{(\beta,s')}_{\ell,1},x^{(\beta,s')}_{\ell+1,1}\right\}\
    (i\leq \ell\leq n-2),\quad
  \left\{x^{(\beta,s)}_{n-2,1},x^{(\beta,s)}_{n,1},x^{(\beta,s')}_{n-2,1},x^{(\beta,s')}_{n,1}\right\}
\]
produces a total factor $\{(w_{\beta,s}-v^{2}w_{\beta,s'})(w_{\beta,s'}-v^{2}w_{\beta,s})\}^{n-i}$,
which is $G_{\beta}$ of \eqref{eq:formulagbeta-D-1}, up to a monomial.

\smallskip
\item $\beta=[i,n,j]$.
Fix any $1\leq s\neq s'\leq d_{\beta}$. According to (\ref{spe-D-2}, \ref{spe-D-3}) and the analysis in the proof
of Proposition \ref{lintegralD}, we can assume that
\begin{align*}
  \hat{o}(x^{(\beta,t)}_{i,1})> \hat{o}(x^{(\beta,t)}_{i+1,1})> \cdots > &\hat{o}(x^{(\beta,t)}_{n-2,1})> \hat{o}(x^{(\beta,t)}_{n-1,1})\
    \&\  \hat{o}(x^{(\beta,t)}_{n,1})>\\
  & \hat{o}(x^{(\beta,t)}_{n-2,2})> \hat{o}(x^{(\beta,t)}_{n-3,2})> \cdots > \hat{o}(x^{(\beta,t)}_{j,2}),\quad t=s\ \text{or}\ s',
\end{align*}
as otherwise the $\phi_{\unl{d}}$-specialization of the corresponding summand vanishes. Then, similarly to $\beta=[i,n,n-1]$ case,
the $\phi_{\unl{d}}$-specialization of the $\zeta$-factors arising from the following quadruples
\begin{align*}
 & \left\{x^{(\beta,s)}_{\ell,1},x^{(\beta,s)}_{\ell+1,1},x^{(\beta,s')}_{\ell,1},x^{(\beta,s')}_{\ell+1,1}\right\}\
   (i\leq \ell\leq n-2),\quad
   \left\{x^{(\beta,s)}_{n-2,1},x^{(\beta,s)}_{n,1},x^{(\beta,s')}_{n-2,1},x^{(\beta,s')}_{n,1}\right\}, \\
 & \left\{x^{(\beta,s)}_{\ell+1,2},x^{(\beta,s)}_{\ell,2},x^{(\beta,s')}_{\ell+1,2},x^{(\beta,s')}_{\ell,2}\right\}\
   (j\leq \ell\leq n-3),\quad
   \left\{x^{(\beta,s)}_{n-1,1},x^{(\beta,s)}_{n-2,2},x^{(\beta,s')}_{n-1,1},x^{(\beta,s')}_{n-2,2}\right\},
\end{align*}
produces a total contribution of the factor $\{(w_{\beta,s}-v^{2}w_{\beta,s'})(w_{\beta,s'}-v^{2}w_{\beta,s})\}^{2n-i-j-1}$.

Next, for any $j\leq \ell\leq n-2$, let us consider the $\zeta$-factors arising from the variables
\begin{equation}
\label{eq:d-l}
  \big\{x^{(\beta,s)}_{\ell,2}  \,,\, x^{(\beta,s')}_{\ell-1,1} \,,\, x^{(\beta,s')}_{\ell,1} \,,\, x^{(\beta,s')}_{\ell+1,1}\big\} \,,
\end{equation}
where we recall that $\hat{o}(x^{(\beta,s')}_{\ell-1,1})> \hat{o}(x^{(\beta,s')}_{\ell,1})> \hat{o}(x^{(\beta,s')}_{\ell+1,1})$.
\begin{itemize}[leftmargin=0.5cm]

\item
If $\hat{o}(x^{(\beta,s)}_{\ell,2})>\hat{o}(x^{(\beta,s')}_{\ell-1,1})$, then the $\zeta$-factors
  $\zeta\left(\frac{x^{(\beta,s')}_{\ell-1,1}}{x^{(\beta,s)}_{\ell,2}}\right)
   \zeta\left(\frac{x^{(\beta,s')}_{\ell,1}}{x^{(\beta,s)}_{\ell,2}}\right)
   \zeta\left(\frac{x^{(\beta,s')}_{\ell+1,1}}{x^{(\beta,s)}_{\ell,2}}\right)$
contribute the overall factor $(w_{\beta,s}-v^{2n-2\ell}w_{\beta,s'})(w_{\beta,s}-v^{2n-2\ell-4}w_{\beta,s'})$
into the $\phi_{\unl{d}}$-specialization.

\item
If $\hat{o}(x^{(\beta,s')}_{\ell-1,1})>\hat{o}(x^{(\beta,s)}_{\ell,2})>\hat{o}(x^{(\beta,s')}_{\ell+1,1})$, then the $\zeta$-factors
  $\zeta\left(\frac{x^{(\beta,s)}_{\ell,2}}{x^{(\beta,s')}_{\ell-1,1}}\right)
   \zeta\left(\frac{x^{(\beta,s')}_{\ell+1,1}}{x^{(\beta,s)}_{\ell,2}}\right)$
contribute the overall factor $(w_{\beta,s}-v^{2n-2\ell}w_{\beta,s'})(w_{\beta,s}-v^{2n-2\ell-4}w_{\beta,s'})$
into the $\phi_{\unl{d}}$-specialization.

\item
If $\hat{o}(x^{(\beta,s')}_{\ell+1,1})>\hat{o}(x^{(\beta,s)}_{\ell,2})$, then the $\zeta$-factors
  $\zeta\left(\frac{x^{(\beta,s)}_{\ell,2}}{x^{(\beta,s')}_{\ell+1,1}}\right)
   \zeta\left(\frac{x^{(\beta,s)}_{\ell,2}}{x^{(\beta,s')}_{\ell,1}}\right)
   \zeta\left(\frac{x^{(\beta,s)}_{\ell,2}}{x^{(\beta,s')}_{\ell-1,1}}\right)$
contribute the overall factor $(w_{\beta,s}-v^{2n-2\ell}w_{\beta,s'})(w_{\beta,s}-v^{2n-2\ell-4}w_{\beta,s'})$
into the $\phi_{\unl{d}}$-specialization.

\end{itemize}
Thus the $\phi_{\unl{d}}$-specialization of the $\zeta$-factors arising from the quadruples \eqref{eq:d-l} produces a total
contribution of the factor $\prod_{\ell=j}^{n-2}\{(w_{\beta,s}-v^{2n-2\ell}w_{\beta,s'})(w_{\beta,s}-v^{2n-2\ell-4}w_{\beta,s'})\}$.
Therefore, the above contributions produce exactly the factor $G_{\beta}$ of \eqref{eq:formulagbeta-D-2},
up to a monomial.
\end{itemize}

Finally, to show that $F$ is integral, it suffices to prove that under the $\Upsilon_{\unl{d},\unl{t}}$, the contribution
of the $\zeta$-factors between the variables $x^{(*,*)}_{*,*}$ that got specialized to $v^{?}z_{\beta,r}$ is divisible by
$[t_{\beta,r}]_{v}!$ for any $\beta\in\Delta^{+}$ and $1\leq r\leq \ell_{\beta}$, cf.~\eqref{verticalspe}. For $\beta=[i,j]$,
this follows from \cite[Lemma~3.51]{Tsy18}. Similarly, for $\beta=[i,n,j]$ with $i<j<n$, we have not used
$\zeta\left(\frac{x^{(\beta,s)}_{i,1}}{x^{(\beta,s')}_{i,1}}\right)$ with $1\leq s\neq s'\leq d_{\beta}$ for the divisibility
of $\phi_{\unl{d}}(F)$ by $G_{\beta}$, thus we can appeal to the ``rank~$1$'' computation of~\cite[Lemma~3.46]{Tsy18} to deduce
the required divisibility by $[t_{\beta,r}]_{v}!$.
\end{proof}

Combining Propositions \ref{vanishD}, \ref{spanD}, and \ref{goodD}, we obtain
the following upgrade of Theorem~\ref{shufflePBWD-D}:

\begin{Thm}\label{rttthm-D}
(a) The $\BQ(v)$-algebra isomorphism $\Psi\colon \qld \,\iso\, S$ of Theorem {\rm \ref{shufflePBWD-D}(a)}
gives rise to a $\BZ[v,v^{-1}]$-algebra isomorphism $\Psi\colon \integrald \,\iso\, \mathcal{S}$.

\noindent
(b) Theorem {\rm \ref{PBWDintegralrtt}} holds for $\fg$ of type $D_n$.
\end{Thm}


\medskip

\section{Yangian counterpart}\label{yangian}

In this section, we generalize the results of Sections \ref{type C}--\ref{type D} to the Yangian case, thus establishing
shuffle algebra realizations of Yangians  and their Drinfeld-Gavarini duals in types $C_{n}$, $D_{n}$. This should be
viewed as the ``rational vs trigonometric'' counterpart, where we replace factors $\frac{z}{w}-v^k$ by $z-w-\frac{k}{2}\hbar$.
In particular, $\zeta_{i,j}(z)$ of~\eqref{eq:zeta} will be replaced by $\hzeta_{i,j}(z)=1+\frac{(\alpha_{i},\alpha_{j})\cdot \hbar}{2z}$.


\subsection{Yangians and their shuffle algebra realization}

We still use the notations from Section \ref{pre}. Let $\mathfrak{g}$ be a finite dimensional simple Lie algebra of type
$C_{n}$ or $D_{n}$. Following \cite{Dri88}, the \textbf{``positive subalgebra'' of the Yangian of $\fg$} in the new Drinfeld
realization, denoted by $\Yangian$, is the $\BQ[\hbar]$-algebra generated by $\{\mathsf{x}_{i,r}\}_{i\in I}^{r\in\mathbb{N}}$
subject to the following defining relations:
\begin{equation*}
  [\sx_{i,r+1},\sx_{j,s}]-[\sx_{i,r},\sx_{j,s+1}]=\frac{d_{i}a_{ij}\hbar}{2}(\sx_{i,r}\sx_{j,s}+\sx_{j,s}\sx_{i,r})
  \qquad \forall\, i,j\in I, r,s\in\BN,
\end{equation*}
\begin{equation*}
  \mathop{Sym}_{s_{1},\dots,s_{1-a_{ij}}}[\sx_{i,s_{1}},[\sx_{i,s_{2}},\cdots,[\sx_{i,s_{1-a_{ij}}},\sx_{j,r}]\cdots]]=0
  \qquad \forall\, i\neq j, s_{1},\dots,s_{1-a_{ij}},r\in \BN.
\end{equation*}
Analogously to \eqref{rootvector1}--\eqref{rootvector3}, let us define the \emph{root vectors}
$\{\sX_{\beta,s}\}_{\beta\in\Delta^{+}}^{s\in\BN}$ of $\Yangian$ in types $C_n, D_n$:
\begin{itemize}[leftmargin=0.7cm]

\item
$C_{n}$-type.

\noindent
For $\beta=[i_{1},\dots,i_{\ell}]\neq [i,n,i]$ and $s\in\BN$, we choose a decomposition $s=s_{1}+\cdots+s_{\ell}$
with $s_{1}, \dots,s_{\ell}\in\BN$. Then, we define
\begin{equation}
\label{rootvector1-Y}
  \sX_{\beta,s}\coloneqq
  [\cdots[[\sx_{i_{1},s_{1}},\sx_{i_{2},s_{2}}],\sx_{i_{3},s_{3}}],\cdots,
  \sx_{i_{\ell},s_{\ell}}].
\end{equation}
For $\beta=[i,n,i]$ and $s\in\BN$, we choose a decomposition $s=s_{1}+s_{2}$ with $s_{1},s_{2}\in\BN$, and consider
the root vectors $\sX_{[i,n-1],s_{1}}, \sX_{[i,n],s_{2}}$ defined in \eqref{rootvector1-Y}. Then, we define
\begin{equation}
\label{rootvector2-Y}
  \sX_{\beta,s}\coloneqq [\sX_{[i,n-1],s_{1}},\sX_{[i,n],s_{2}}].
\end{equation}

\item
$D_{n}$-type.

\noindent
For any $\beta=[i_{1},\dots,i_{\ell}]\in\Delta^{+}$ and $s\in\BN$, we choose a  decomposition $s=s_{1}+\cdots+s_{\ell}$
with $s_{1}, \dots,s_{\ell}\in\BN$. Then, we define
\begin{equation}
\label{rootvector3-Y}
  \sX_{\beta,s}\coloneqq
  [\cdots[[\sx_{i_{1},s_{1}},\sx_{i_{2},s_{2}}],\sx_{i_{3},s_{3}}],\cdots,
   \sx_{i_{\ell},s_{\ell}}].
\end{equation}

\end{itemize}
In particular, we have the following specific choices of root vectors $\{\tilde{\sX}_{\beta,s}\}_{\beta\in\Delta^{+}}^{s\in \BN}$:
\begin{itemize}[leftmargin=0.7cm]

\item
For $\beta=[i,n,i]$ and $s\in\BN$ ($\fg$ is of type $C_n$), we define
\begin{equation*}
\begin{aligned}
  \tilde{\sX}_{[i,n,i],s}\coloneqq
  [[\cdots&[\sx_{i,0},\sx_{i+1,0}],\cdots,\sx_{n-1,0}],[[\cdots[\sx_{i,0},\sx_{i+1,0}],\cdots,\sx_{n-1,0}],\sx_{n,s}]].
\end{aligned}
\end{equation*}

\item
Otherwise, for $\beta=[i_{1},\dots,i_{\ell}]$ and $s\in\BN$, we define
\begin{equation*}
  \tilde{\sX}_{\beta,s}\coloneqq[\cdots[[\sx_{i_{1},s},\sx_{i_{2},0}],\sx_{i_{3},0}],\cdots, \sx_{i_{\ell},0}].
\end{equation*}

\end{itemize}
Let $\sH$ denote the set of all functions $h\colon \Delta^{+}\times\BN\rightarrow \BN$ with finite support.
For any $h\in\sH$, we consider the ordered monomials
\begin{equation}
\label{eq:pbwd-yangian}
  \sX_{h}\, =\prod_{(\beta,s)\in\Delta^{+}\times\mathbb{N}}\limits^{\rightarrow}\sX_{\beta,s}^{h(\beta,s)}
  \qquad \text{and} \qquad
  \tilde{\sX}_{h}\, =\prod_{(\beta,s)\in\Delta^{+}\times\mathbb{N}}\limits^{\rightarrow}\tilde{\sX}_{\beta,s}^{h(\beta,s)}.
\end{equation}
Then, similarly to \cite{Lev93} (cf.\ \cite[Theorem B.3]{FT19}), we have:

\begin{Thm}\label{yangianbasis}
The elements $\{ \tilde{\sX}_{h}\}_{h\in \sH}$ form a basis of the free $\BQ[\hbar]$-module $\Yangian$.
\end{Thm}

\begin{proof}
Comparing $\tilde{\sX}_{\beta,s}$ to the root vectors $e_{\beta}^{(s)}$ used in \cite[(A.11)]{FT19}, we see that
the only difference is in the root vectors $\tilde{\sX}_{[i,n,i],s}$ in $C_n$-type. However, the two key properties
(B.1) and (B.2) of \cite[Appendix B]{FT19} still hold for our root vectors. Hence, the proof of \cite[Theorem~B.2]{FT19}
and thus of \cite[Theorem B.3]{FT19} still goes through.
\end{proof}

We define the shuffle algebra $(\bar\BW,\star)$ analogously to the shuffle algebra $(S,\star)$ of Section \ref{pre}
with the following modifications:
\begin{itemize}[leftmargin=0.7cm]

\item
All rational functions $F\in \bar\BW$ are defined over $\BQ[\hbar]$.

\item
The matrix $(\hzeta_{i,j}(z))_{i,j\in I}$ is defined via
\begin{equation*}
  \hzeta_{i,j}(z)=1+\frac{(\alpha_{i},\alpha_{j})\cdot \hbar}{2z}.
\end{equation*}

\item
(\emph{pole conditions}) $F\in \bar\BW_{\unl{k}}$ has the form
\begin{equation}
  F=\frac{f(\{x_{i,r}\}_{i\in I}^{1\leq r\leq k_{i}})}
         {\prod_{i<j}^{a_{ij}\neq 0}\prod_{1\leq r\leq k_{i}}^{1\leq s\leq k_{j}}(x_{i,r}-x_{j,s})},
\label{poleyg}
\end{equation}
where $f\in \BQ[\hbar][\{x_{i,r}\}_{i\in I}^{1\leq r\leq k_{i}}]^{\mathfrak{S}_{\underline{k}}}$ and $<$ is an arbitrary order on $I$.

\item
(\emph{wheel conditions}) Let $f$ be the numerator of $F\in \bar\BW_{\unl{k}}$ from \eqref{poleyg}, then
\begin{equation}
\label{wheel-Y}
  f(\{x_{i,r}\}_{i\in I}^{1\leq r\leq k_{i}})=0 \  \text{once}\
  x_{i,s_{1}}=x_{i,s_{2}}+d_{i}\hbar=\cdots=x_{i,s_{1-a_{ij}}}-d_{i}a_{ij}\hbar=x_{j,r}-\frac{d_{i}a_{ij}}{2}\hbar
\end{equation}
for any $ i\neq j$ such that $a_{ij}\neq 0$, pairwise distinct $1\leq s_{1},\dots,s_{1-a_{ij}}\leq k_{i}$, and $1\leq r\leq k_{j}$.

\item
The shuffle product is defined like \eqref{shuffleproduct}, but $\zeta_{i,j}(\frac{x_{i,r}}{x_{j,s}})$
are replaced by $\hzeta_{i,j}(x_{i,r}-x_{j,s})$.

\end{itemize}
This definition is precisely engineered, so that  the assignment
$\sx_{i,r}\mapsto x_{i,1}^{r}\in \bar\BW_{\mathbf{1}_i}\, (i\in I, r\in\BN)$ gives rise to a $\BQ[\hbar]$-algebra homomorphism
\begin{equation}
\label{eq:Psi-homom-rat}
  \Psi\colon \Yangian \longrightarrow \bar\BW.
\end{equation}
Henceforth, we shall use the notation $\circeq$ as in~\cite[(5.19)]{HT24} (cf.~\eqref{eq:trig-const}):
\begin{equation*}
  A\circeq B \quad  \text{if} \quad  A=c\cdot B \quad \text{for some}\  c\in \BQ^{\times}.
\end{equation*}
We shall also use $\mathrm{denom}_\beta$ to denote the denominator in~\eqref{poleyg} for any $F\in \bar\BW_\beta$.

Then, we have the following straightforward analogues of Lemmas \ref{lem:psirv-C} and \ref{lem:psirv-D}:

\begin{Lem}\label{Y-psirv-C}
For type $C_{n}$, we have:
\begin{equation*}
\begin{aligned}
  \Psi(\tilde{\sX}_{[i,j],s})&\circeq \frac{\hbar^{j-i}x_{i,1}^{s}}{\mathrm{denom}_{[i,j]}}
    \ \ \mathrm{for} \ \ i\leq j\leq n , \\
  \Psi(\tilde{\sX}_{[i,n,j],s})&\circeq
    \frac{\hbar^{2n-i-j}x_{i,1}^{s}} {\mathrm{denom}_{[i,n,j]}}(2x_{j-1,1}-x_{j,1}-x_{j,2})
    \prod^{n-2}_{\ell=j}\hat{Q}(x_{\ell,1},x_{\ell,2},x_{\ell+1,1},x_{\ell+1,2})
    \ \mathrm{for} \ i<j<n , \\
  \Psi(\tilde{\sX}_{[i,n,i],s})&\circeq
    \frac{\hbar^{2n-2i}x_{n,1}^{s}} {\mathrm{denom}_{[i,n,i]}}
    \prod^{n-2}_{\ell=i}\hat{Q}(x_{\ell,1},x_{\ell,2},x_{\ell+1,1},x_{\ell+1,2})
  \ \ \mathrm{for} \ \ i<n,
\end{aligned}
\end{equation*}
where $\hat{Q}(x_{1},x_{2},y_{1},y_{2})=4(x_{1}x_{2}+y_1y_2)-2(x_1+x_2)(y_1+y_2)+\hbar^2$.
\end{Lem}

\begin{Lem}\label{Y-psirv-D}
For type $D_{n}$, we have:
\begin{align*}
  \Psi(\tilde{\sX}_{\beta,s}) &\circeq \frac{\hbar^{\abs{\beta}-1}x_{i,1}^{s}}{\mathrm{denom}_{\beta}}
    \quad \mathrm{for}\ \beta=[i,j]\ \mathrm{or}\ [i,n,n-1],\\
  \Psi(\tilde{\sX}_{[i,n,j],s}) &\circeq
    \frac{\hbar^{2n-i-j-1}} {\mathrm{denom}_{[i,n,j]}}x_{i,1}^{s}
    \prod_{\ell=j}^{n-2} (\hbar+x_{\ell,1}-x_{\ell,2})(\hbar-x_{\ell,1}+x_{\ell,2})
  \ \ \mathrm{for} \ \ i<j<n-1.
\end{align*}
\end{Lem}

Moreover, due to the equality $\hzeta_{i,j}(z)-\hzeta_{j,i}(-z)=\frac{(\alpha_{i},\alpha_{j})}{z}\hbar$, for more
general root vectors $\sX_{\beta,s}$ defined in~\eqref{rootvector1-Y}--\eqref{rootvector3-Y}, we have:

\begin{Lem}\label{conrv}
For any $\beta\in\Delta^{+}$ and $s\in\BN$, $\Psi(\sX_{\beta,s})$ is divisible by $\hbar^{|\beta|-1}$.
\end{Lem}

Let us now adapt our key tool of {\em specialization maps} to the Yangian setup. For any $F\in\bar\BW_{\unl{k}}$
and $\unl{d}\in\text{KP}(\unl{k})$, let $f$ be the numerator of $F$ from  \eqref{poleyg}. The specialization map
$\phi_{\underline{d}}(F)$ is defined by successive specializations $\phi_{\beta,s}$ of the variables $x^{(\beta,s)}_{*,*}$
in $f$ for each $\beta\in\Delta^+$ and $1\leq s\leq d_{\beta}$ as follows (cf.\ \eqref{spe-C-1}--\eqref{spe-D-3}):
\begin{itemize}[leftmargin=0.7cm]

\item
$C_n$-type.

For $\beta\neq [i,n,i]$, we define $\phi_{\beta,s}(F)$ by specializing :
\begin{equation*}
  x^{(\beta,s)}_{\ell\neq n,1}\mapsto w_{\beta,s}-\frac{\ell-1}{2}\hbar,\quad
  x^{(\beta,s)}_{\ell\neq n,2}\mapsto w_{\beta,s}-\frac{2n+1-\ell}{2}\hbar,\quad
  x^{(\beta,s)}_{n,1}\mapsto w_{\beta,s}-\frac{n}{2}\hbar.
\end{equation*}
For $\beta= [i,n,i]$, we first define $\phi^{(1)}_{\beta,s}(F)$ by specializing:
\begin{align*}
  x^{(\beta,s)}_{\ell\neq n,1}\mapsto w_{\beta,s}-\frac{\ell-1}{2}\hbar,\quad
  x^{(\beta,s)}_{\ell\neq n,2}\mapsto w'_{\beta,s}-\frac{\ell-1}{2}\hbar,\quad
  x^{(\beta,s)}_{n,1}\mapsto w'_{\beta,s}-\frac{n}{2}\hbar.
\end{align*}
According to wheel conditions \eqref{wheel-Y}, $\phi^{(1)}_{\beta,s}(F)$ is divisible by
\begin{equation*}
  B_{\beta}=\big\{(w_{\beta,s}-w'_{\beta,s}+\hbar)(w_{\beta,s}-w'_{\beta,s}-\hbar)\big\}^{n-i-1}.
\end{equation*}
Then the overall specialization $\phi_{\beta,s}(F)$ is defined by
\begin{equation*}
  \phi_{\beta,s}(F)\coloneqq \phi^{(2)}_{\beta,s}\left(\phi^{(1)}_{\beta,s}(F)\right) =
  \eval{\frac{\phi^{(1)}_{\beta,s}(F)}{B_{\beta}}}_{w'_{\beta,s}\mapsto w_{\beta,s}+\hbar}.
\end{equation*}

\item
$D_n$-type.

For $\beta \neq [i,n,j]$ with $i<j\leq n-2$, we define $\phi_{\beta,s}(F)$ by specializing:
\begin{equation*}
  x^{(\beta,s)}_{\ell\neq n,1}\mapsto w_{\beta,s}-\frac{\ell-1}{2}\hbar,\quad
  x^{(\beta,s)}_{n,1}\mapsto w_{\beta,s}-\frac{n-2}{2}\hbar.
\end{equation*}
For $\beta= [i,n,j]$ with $1\leq i<j\leq n-2$, we first define $\phi^{(1)}_{\beta,s}(F)$ by  specializing:
\begin{align*}
   x^{(\beta,s)}_{\ell\neq n,1}\mapsto w_{\beta,s}-\frac{\ell-1}{2}\hbar,\quad
   x^{(\beta,s)}_{n,1}\mapsto w_{\beta,s}-\frac{n-2}{2}\hbar,\quad
   x^{(\beta,s)}_{\ell\neq n-1\& n,2}\mapsto w'_{\beta,s}-\frac{2n-3-\ell}{2}\hbar.
\end{align*}
According to wheel conditions \eqref{wheel-Y}, $\phi^{(1)}_{\beta,s}(F)$ is divisible by
\begin{equation*}
  B_{\beta}=\prod^{n-2}_{\ell=j}(w_{\beta,s}-w'_{\beta,s}-(n-\ell-2)\hbar)(w_{\beta,s}-w'_{\beta,s}-(n-\ell)\hbar).
\end{equation*}
Then, the overall specialization $\phi_{\beta,s}(F)$ is defined by:
\begin{equation*}
  \phi_{\beta,s}(F)\coloneqq \phi^{(2)}_{\beta,s}\left(\phi^{(1)}_{\beta,s}(F)\right) =
  \eval{\frac{\phi^{(1)}_{\beta,s}(F)}{B_{\beta}}}_{w'_{\beta,s}\mapsto w_{\beta,s}}.
\end{equation*}

\end{itemize}
For $\unl{d}\in \mathrm{KP}(\unl{k})$, the specialization map $\phi_{\underline{d}}(F)$ is defined by applying those
separate maps $\phi_{\beta,s}$ in each group $\big\{x^{(\beta,s)}_{i,t}\big\}_{1\leq t\leq \nu_{\beta,i}}^{i\in I}$
of variables (the result is independent of splitting):
\begin{equation*}
  \phi_{\underline{d}}\colon \bar\BW_{\unl{k}}\longrightarrow
  \BQ[\hbar][\{w_{\beta,s}\}_{\beta\in\Delta^{+}}^{1\leq s\leq d_{\beta}}]^{\mathfrak{S}_{\unl{d}}},
\end{equation*}
and we extend it by zero to all other components $\bar\BW_{\unl{\ell}}$ with $\unl{\ell}\ne \unl{k}$.
Then, we have the following straightforward analogues of Lemmas \ref{phirv-C} and \ref{phirv-D}:

\begin{Lem}\label{phiyangianrs}
If $\fg$ is of type $C_{n}$ or $D_{n}$, then we have:
\begin{equation*}
  \phi_{\beta}(\Psi(\sX_{\beta,s}))\circeq
  \hbar^{\kappa_{\beta}}\cdot p_{\beta,s}(w_{\beta,1}) \qquad \forall\, (\beta,s)\in\Delta^{+}\times \BN,
\end{equation*}
where $\kappa_{\beta}$ is given by \eqref{kappaC} in type $C_n$, $\kappa_{\beta}=\abs{\beta}-1$ in type $D_n$,
and $p_{\beta,s}(w)\in \BQ[\hbar][w]$ is a monic degree $s$ polynomial in $w$ over $\BQ[\hbar]$.
\end{Lem}

For any $\unl{k}\in \BN^I$ and $\unl{d}\in \mathrm{KP}(\unl{k})$, we define the subsets $\sH_{\unl{k}}$,
$\sH_{\unl{k},\unl{d}}$ of $\sH$ similarly to \eqref{hunlkunld}, but with $h\in H$ been replaced by $h\in \sH$.
Using Lemma \ref{phiyangianrs} and arguing as in Sections \ref{type C}--\ref{type D}, we obtain the following
analogues of Propositions \ref{spekpC}, \ref{vanishC}, \ref{vanishD} for the Yangians of types $C_{n},D_{n}$:

\begin{Prop}\label{shuffleeleyang}
Let $\fg$ be of type $C_n$ or $D_n$. Then we have:

\noindent
(a) For  any $h\in\sH_{\unl{k},\unl{d}}$, we have
\begin{equation*}
  \phi_{\unl{d}}(\Psi(\sX_{h}))\circeq
  \hbar^{\sum_{\beta\in\Delta^{+}}d_{\beta}\kappa_{\beta}}\cdot
  \prod_{\beta,\beta'\in \Delta^+}^{\beta<\beta'}\hat{G}_{\beta,\beta'}\cdot \prod_{\beta\in\Delta^{+}}\hat{G}_{\beta}\cdot
  \prod_{\beta\in\Delta^{+}}\hat{P}_{\lambda_{h,\beta}},
\end{equation*}
where $\hat{G}_{\beta,\beta'},\hat{G}_{\beta}$ are \underline{independent of $h\in \sH_{\unl{k},\unl{d}}$} and
are rational counterparts of $G_{\beta,\beta'},G_{\beta}$ from Propositions {\rm \ref{spekpC}, \ref{vanishD}}
(obtained by replacing factors $(x-v^{t}y)$ with $(x-y-\frac{t}{2}\hbar)$), while
\begin{equation}
\label{hlp-rat}
  \hat{P}_{\lambda_{h,\beta}}={\mathop{Sym}}_{\mathfrak{S}_{d_{\beta}}}
  \left(\prod_{s=1}^{d_{\beta}}p_{\beta,r_{\beta}(h,s)}(w_{\beta,s})
  \prod_{1\leq s<r\leq d_{\beta}}\Big(1+\frac{(\beta,\beta)\cdot \hbar}{2(w_{\beta,s}-w_{\beta,r})}\Big)\right).
\end{equation}

\noindent
(b) For any $h\in\sH_{\unl{k},\unl{d}}$ and $\unl{d}'<\unl{d}$, we have $\phi_{\unl{d}'}(\Psi(\sX_{h}))=0$.
\end{Prop}

This features a ``rank $1$ reduction'': each $\hat{P}_{\lambda_{h,\beta}}$ from~\eqref{hlp-rat} can be viewed as the
shuffle product $p_{\beta,r_{\beta}(h,1)}(x)\star\cdots \star p_{\beta,r_{\beta}(h,d_\beta)}(x)$ in the $A_1$-type
shuffle algebra $\bar\BW$, evaluated at $\{w_{\beta,s}\}_{s=1}^{d_\beta}$. Therefore, combining
Proposition~\ref{shuffleeleyang} with Theorem \ref{yangianbasis}, we obtain:

\begin{Prop}\label{injyang}
The homomorphism $\Psi$ of~\eqref{eq:Psi-homom-rat} is injective.
\end{Prop}

Following \cite[Definition 3.27]{Tsy19}, we introduce:

\begin{Def}
$F\in \bar\BW_{\unl{k}}$ is \textbf{good} if $\phi_{\unl{d}}(F)$ is divisible by
$\hbar^{\sum_{\beta\in \Delta^{+}}d_{\beta}\kappa_{\beta}}$ for any $\unl{d}\in\text{\rm KP}(\unl{k})$.
\end{Def}

Let $\BW_{\unl{k}}$ be the $\BQ[\hbar]$-submodule of all good elements in $\bar\BW_{\unl{k}}$, and set
$\BW:=\bigoplus_{\unl{k}\in\BN^{I}}\BW_{\unl{k}}$. Then analogously to our proofs of
Propositions \ref{lintegralC} and \ref{lintegralD}, we obtain (cf.\ \cite[Proposition 5.12]{HT24}):

\begin{Prop}\label{subsetyang}
$\Psi(\Yangian)\subset \BW$.
\end{Prop}

Let $\BW'_{\unl{k}}$ be the $\BQ[\hbar]$-submodule of $\BW_{\unl{k}}$ spanned by $\{\Psi(\sX_{h})\}_{h\in \sH_{\unl{k}}}$.
Then, the following Yangian counterpart of Lemma \ref{span} holds true in types $C_{n}$ and $D_{n}$:

\begin{Prop}\label{spanyang}
For any $F\in \BW_{\unl{k}}$, $\unl{d}\in \text{\rm KP}(\unl{k})$, if $\phi_{\unl{d}'}(F)=0$
for all $\unl{d}'\in \text{\rm KP}(\unl{k})$ such that $\unl{d}'<\unl{d}$, then there exists
$F_{\unl{d}}\in \BW'_{\unl{k}}$ such that $\phi_{\unl{d}}(F)=\phi_{\unl{d}}(F_{\unl{d}})$
and $\phi_{\unl{d}'}(F_{\unl{d}})=0$ for all $\unl{d}'<\unl{d}$.
\end{Prop}

\begin{proof}
The proof is analogous to that of \cite[Proposition 5.13]{HT24}.
\end{proof}

Combining Propositions \ref{subsetyang}--\ref{spanyang}, we immediately obtain the shuffle algebra realization and
an upgrade of Theorem~\ref{yangianbasis} for $\Yangian$ in types $C_{n}$ and $D_{n}$, cf.~\cite[Theorem 5.14]{HT24}:

\begin{Thm}\label{yangshuffle}
Let $\fg$ be of type $C_n$ or $D_n$. Then we have:

\noindent
(a) The $\BQ[\hbar]$-algebra homomorphism $\Psi\colon \Yangian \to \bar\BW$ of \eqref{eq:Psi-homom-rat}
gives rise to a $\BQ[\hbar]$-algebra isomorphism $\Psi\colon \Yangian \,\iso\, \BW$.

\noindent
(b) The ordered monomials $\{\sX_{h}\}_{h\in\sH}$ of~\eqref{eq:pbwd-yangian} form a basis of the free
$\BQ[\hbar]$-module $\Yangian$.
\end{Thm}


\subsection{The Drinfeld-Gavarini dual $\dgdual$ and its shuffle algebra realization}\label{DGdual}

For any $(\beta,s)\in\Delta^{+}\times\BN$, define $\bar{\sX}_{\beta,s}\in Y_{\hbar}^{>}(\mathfrak{g})$ via
\begin{equation*}
  \bar{\sX}_{\beta,s}\coloneqq \hbar\cdot \sX_{\beta,s}.
\end{equation*}
We define $\dgdual$, the \emph{``positive subalgebra'' of the Drinfeld-Gavarini dual}, as the  $\BQ[\hbar]$-subalgebra
of $\Yangian$ generated by $\{ \bar{\sX}_{\beta,s}\}_{\beta\in\Delta^{+}}^{s\in\BN}$. For any $h\in \sH$, define
the ordered monomial (cf.~\eqref{eq:pbwd-yangian}):
\begin{equation}
\label{eq:pbwd-yangian-gavarini}
  \bar{\sX}_{h}\coloneqq \prod_{(\beta,s)\in\Delta^{+}\times\mathbb{N}}\limits^{\rightarrow}\bar{\sX}_{\beta,s}^{h(\beta,s)}.
\end{equation}
Following~\cite[Definition 3.8]{Tsy19}, we introduce:

\begin{Def}
$F\in \bar\BW_{\unl{k}}$ is \textbf{integral} if $F$ is divisible by $\hbar^{|\unl{k}|}$ and $\phi_{\unl{d}}(F)$
is divisible by $\hbar^{\sum_{\beta\in \Delta^{+}}d_{\beta}(\kappa_{\beta}+1)}$ for any $\unl{d}\in \mathrm{KP}(\unl{k})$.
\end{Def}

Let $\mathbf{W}_{\unl{k}}\subset \bar\BW_{\unl{k}}$ be the $\BQ[\hbar]$-submodule of all integral elements, and set
$\mathbf{W}:=\bigoplus_{\unl{k}\in\BN^{I}}\mathbf{W}_{\unl{k}}$. Then, due to Lemmas \ref{conrv}--\ref{phiyangianrs} and
Proposition \ref{subsetyang}, we have the following upgrade of Theorem~\ref{yangshuffle} (cf.~\cite[Theorems 5.16, 5.20]{HT24}):

\begin{Thm}\label{dgdualshuffle}
Let $\fg$ be of type $C_n$ or $D_n$. Then we have:

\noindent
(a) $\dgdual$ is independent of the choice of root vectors $\sX_{\beta,s}$ in \eqref{rootvector1-Y}--\eqref{rootvector3-Y}.

\noindent
(b) The $\BQ[\hbar]$-algebra isomorphism $\Psi\colon \Yangian \,\iso\, \BW$ of Theorem {\rm \ref{yangshuffle}(a)} gives rise
to a $\BQ[\hbar]$-algebra isomorphism $\Psi\colon \dgdual \,\iso\, \mathbf{W}$.

\noindent
(c) For any choices of $s_k$ in \eqref{rootvector1-Y}--\eqref{rootvector3-Y}, the ordered monomials $\{\bar{\sX}_{h}\}_{h\in \sH}$
of~\eqref{eq:pbwd-yangian-gavarini} form a basis of the free $\BQ[\hbar]$-module $\dgdual$.
\end{Thm}


\medskip

\appendix
\section{The RTT realization in types $C_n$ and $D_n$}\label{sec:app}

In this section, we recall the RTT realization of $U_v(L\mathfrak{sp}_{2n})$ and $U_v(L\mathfrak{o}_{2n})$,
established in~\cite{JLM21,JLM20}, and use it to explain the natural origin and the name of the integral forms
$\integralc$ and $\integrald$ from Definition~\ref{def:rttintegral} and Subsections~\ref{rttc},~\ref{rttd}.
While the analysis is very similar, we shall start with $D_n$-type, which ends up in slightly simpler formulas.


\subsection{RTT realization of $U_{v}(L\sso_{2n})$}

Set $N=2n$. For $1\leq i\leq N$, we define $i'$ and $\bar{i}$ via:
\begin{align}
  i'& := N+1-i,
    \label{eq:prime} \\
  (\bar{1},\ldots,\bar{N}) &:= (n-1,\ldots,1, 0, 0, -1, \ldots, -n+1).
    \label{eq:bar}
\end{align}
To follow the notations of~\cite{JLM20}, we also define
  $$\xi=v^{2-N}.$$
Consider the trigonometric $R$-matrix with a spectral parameter $\bar{R}_\trig(z)$ given by
\begin{equation}
\label{trigR}
  \bar{R}_\trig(z):=
  \frac{z-1}{zv-v^{-1}} \, R + \frac{v-v^{-1}}{zv-v^{-1}} \, P - \frac{(v-v^{-1})(z-1)\xi}{(zv-v^{-1})(z-\xi)} \, Q ,
\end{equation}
where $P,Q,R\in (\End\, \BC^N)^{\otimes 2}$ are defined via:
\begin{equation*}
\begin{split}
  & P\ =\sum_{1\leq i,j\leq N} e_{ij}\otimes e_{ji}, \qquad
    Q\ =\sum_{1\leq i,j\leq N} v^{\bar{i}-\bar{j}} e_{i'j'}\otimes e_{ij},\\
  & R\ = v\sum_{1\leq i\leq N} e_{ii}\otimes e_{ii} \, +
    \sum_{1\leq i,j\leq N}^{i\ne j,j'} e_{ii}\otimes e_{jj} + v^{-1} \sum_{1\leq i\leq N} e_{ii}\otimes e_{i'i'} \ + \\
  & \qquad \quad (v-v^{-1})\sum_{i<j} e_{ij}\otimes e_{ji} - (v-v^{-1})\sum_{i>j} v^{\bar{i}-\bar{j}} e_{i'j'}\otimes e_{ij}.
\end{split}
\end{equation*}
This $\bar{R}_\trig(z)$ satisfies the famous \emph{Yang-Baxter equation} (with a spectral parameter):
\begin{equation}
\label{qYB}
  \bar{R}_{\trig;12}(z) \bar{R}_{\trig;13}(zw) \bar{R}_{\trig;23}(w)=
  \bar{R}_{\trig;23}(w) \bar{R}_{\trig;13}(zw) \bar{R}_{\trig;12}(z).
\end{equation}

Following~\cite{JLM20} (with the conceptual ideology going back to~\cite{frt}), we define the
\textbf{RTT integral form of the quantum loop algebra of $\sso_{N}$}, denoted by $\rtU^\rtt_v(L\sso_{N})$, to be
the associative $\BZ[v,v^{-1}]$-algebra generated by $\{\ell^\pm_{ij}[\mp r]\}_{1\leq i,j\leq N}^{r\in \BN}$ with
the following defining relations:
\begin{equation}
\label{affRTT}
\begin{split}
  & \ell^+_{ij}[0]=\ell^-_{ji}[0]=0\ \ \mathrm{for}\ 1\leq i<j\leq N,\\
  & \ell^\pm_{ii}[0]\ell^\mp_{ii}[0]=1\ \ \mathrm{for}\ 1\leq i\leq N,\\
  & \bar{R}_{\trig}(z/w)\Lf^\pm_1(z)\Lf^\pm_2(w)=\Lf^\pm_2(w)\Lf^\pm_1(z)\bar{R}_\trig(z/w),\\
  & \bar{R}_{\trig}(z/w)\Lf^+_1(z)\Lf^-_2(w)=\Lf^-_2(w)\Lf^+_1(z)\bar{R}_\trig(z/w),
\end{split}
\end{equation}
(the last two are commonly called the \emph{RTT relations}) as well as
\begin{equation}
\label{affRTT-extra}
  \Lf^\pm(z) D \Lf^\pm(z\xi)^{\mathrm{t}} D^{-1} = 1,
\end{equation}
where ${\mathrm{t}}$ denotes the matrix transposition with $E^{\mathrm{t}}_{ij}=E_{j'i'}$ and $D$ is the diagonal matrix
\begin{equation*}
  D=\mathrm{diag}\big(v^{\bar{1}},v^{\bar{2}},\ldots,v^{\bar{N}}\big).
\end{equation*}
Here, $\Lf^\pm(z)\in \rtU^\rtt_v(L\sso_N)[[z^{\pm 1}]]\otimes \End\, \BC^N$ is defined by
\begin{equation}
\label{eq:L-matrix}
  \Lf^\pm(z)\ =\sum_{1\leq i,j\leq N} \ell^\pm_{ij}(z)\otimes E_{ij} \qquad \mathrm{with} \qquad
  \ell^\pm_{ij}(z):=\sum_{r\geq 0} \ell^\pm_{ij}[\mp r] z^{\pm r}.
\end{equation}
We also define the $\BC(v)$-counterpart $U^\rtt_v(L\sso_N):=\rtU^\rtt_v(L\sso_N)\otimes_{\BZ[v,v^{-1}]} \BC(v)$.

Let $U_v(L\sso_N)$ be the quantum loop algebra of type $D_n$ in the new Drinfeld realization. It is a $\BC(v)$-algebra
generated by $\{x^\pm_{i,r},\varphi_{i,-k},\psi_{i,k},k_i^{\pm 1}\}_{1\leq i\leq n}^{r\in \BZ,k>0}$ with the relations
as in~\cite[\S1]{JLM20}. Identifying $x^+_{i,r}$ with our $e_{i,r}$, the subalgebra generated by
$\{x^+_{i,r}\}_{1\leq i\leq n}^{r\in \BZ}$ recovers our $U^>_v(L\sso_N)$ from Subsection~\ref{ssec:qlas}.
In what follows, we will consider the following generating series:
\begin{equation}
\label{eq:loop-series}
  x^\pm_i(z)=\sum_{r\in \BZ} x^\pm_{i,r} z^{-r},\qquad
  \varphi_i(z)=\sum_{k\geq 0} \varphi_{i,-k} z^k,\qquad
  \psi_i(z)=\sum_{k\geq 0} \psi_{i,k} z^{-k}.
\end{equation}

The relation between the algebras $U_v(L\sso_N)$ and $\rtU^\rtt_v(L\sso_{N})$ was established in~\cite{JLM20}.
To state the main result, we consider the Gauss decomposition of the matrices $\Lf^\pm(z)$ from~\eqref{eq:L-matrix}:
\begin{equation*}
  \Lf^\pm(z)=F^\pm(z)\cdot H^\pm(z)\cdot E^\pm(z).
\end{equation*}
Here, $F^\pm(z), H^\pm(z), E^\pm(z)\in \rtU^\rtt_v(L\sso_N)[[z^{\pm 1}]]\otimes \End\, \BC^N$ are of the form
\begin{equation*}
  F^\pm(z)=\sum_{i} E_{ii}+\sum_{i>j} f^\pm_{ij}(z)\cdot E_{ij},\
  H^\pm(z)=\sum_{i} h^\pm_i(z)\cdot E_{ii},\
  E^\pm(z)=\sum_{i} E_{ii}+\sum_{i<j} e^\pm_{ij}(z)\cdot E_{ij}.
\end{equation*}

\begin{Thm}[\cite{JLM20}]\label{thm:JLM-iso}
There is a unique $\BC(v)$-algebra isomorphism
\begin{equation*}
  \varrho \colon U_v(L\sso_N) \, \iso \, U^\rtt_v(L\sso_{N})
\end{equation*}
defined by
\begin{equation}
\label{JLM formulas D1}
\begin{split}
  & x^+_i(z)\mapsto \frac{e^+_{i,i+1}(zv^i)-e^-_{i,i+1}(zv^i)}{v-v^{-1}}, \qquad
    x^-_i(z) \mapsto \frac{f^+_{i+1,i}(zv^i)-f^-_{i+1,i}(zv^i)}{v-v^{-1}}, \\
  & \quad \psi_i(z)\mapsto h^-_{i+1}(zv^i)h^-_{i}(zv^i)^{-1}, \qquad
    \varphi_i(z)\mapsto h^+_{i+1}(zv^i)h^+_{i}(zv^i)^{-1}
\end{split}
\end{equation}
for $1\leq i<n$ and
\begin{equation}
\label{JLM formulas D2}
\begin{split}
  & x^+_n(z)\mapsto \frac{e^+_{n-1,n+1}(zv^{n-1})-e^-_{n-1,n+1}(zv^{n-1})}{v-v^{-1}}, \
    \psi_n(z)\mapsto h^-_{n+1}(zv^{n-1})h^-_{n-1}(zv^{n-1})^{-1}, \\
  & x^-_n(z) \mapsto \frac{f^+_{n+1,n-1}(zv^{n-1})-f^-_{n+1,n-1}(zv^{n-1})}{v-v^{-1}}, \
     \varphi_n(z)\mapsto h^+_{n+1}(zv^{n-1})h^+_{n-1}(zv^{n-1})^{-1}.
\end{split}
\end{equation}
\end{Thm}


\subsection{The RTT realization of $\integrald$}

Let $\rtU^{\rtt,>}_v(L\sso_{N})$ be the $\BZ[v,v^{-1}]$-subalgebra of $\rtU^\rtt_v(L\sso_{N})$ generated by the coefficients
of $\{e^\pm_{ij}(z)\}_{1\leq i<j\leq N}$, the matrix coefficients of $E^\pm(z)$. The key goal of this Appendix is to highlight
the natural origin of the integral form $\integrald$ introduced in Definition~\ref{def:rttintegral} and its specific quantum root
vectors (a special case of~\eqref{eq:rtt-vectors-D})
\begin{equation}
\label{eq:rtt-root-vectors}
\begin{split}
  & \tilde{\mathcal{E}}^{\rtt}_{[i,j],s}:= \langle 1\rangle_{v}\cdot
    [\cdots[[e_{i,s},e_{i+1,0}]_{v},e_{i+2,0}]_{v},\cdots,e_{j,0}]_{v},\\
  & \tilde{\mathcal{E}}^{\rtt}_{[i,n],s}:= \langle 1\rangle_{v}\cdot
    [[\cdots[e_{i,s},e_{i+1,0}]_{v},,\cdots,e_{n-2,0}]_{v},e_{n,0}]_{v},\\
  & \tilde{\mathcal{E}}^{\rtt}_{[i,n,j],s}:= \langle 1\rangle_{v}\cdot
    [\cdots[[[\cdots[e_{i,s},e_{i+1,0}]_{v},\cdots, e_{n-2,0}]_{v}, e_{n,0}]_{v},e_{n-1,0}]_{v},\cdots,e_{j,0}]_{v}
\end{split}
\end{equation}
for any $1\leq i<j<n$.  We also express the matrix coefficients of $E^\pm(z)$ as series in $z^{\pm 1}$:
\begin{equation}
\label{eq:rtt-matrix-coef}
  e^+_{ij}(z)=\sum_{r>0} e^{(-r)}_{ij} z^r,\qquad
  e^-_{ij}(z)=\sum_{r\geq 0} e^{(r)}_{ij} z^{-r}
  \qquad \forall\, 1\leq i<j\leq N.
\end{equation}
Finally we define $e_{ij}(z):=e^+_{ij}(z)-e^-_{ij}(z)$. The key technical result of this subsection is:

\begin{Prop}\label{prop:E-explicitly}
(a) For any $1\leq i<j<n$, we have:
\begin{equation}
\label{eq:iterated-1}
  e_{i,j+1}(z)=(1-v^2)^{i-j}\cdot
  [\cdots[[e_{i,i+1}(z),e^{(0)}_{i+1,i+2}]_{v},e^{(0)}_{i+2,i+3}]_{v},\cdots,e^{(0)}_{j,j+1}]_{v}.
\end{equation}

\noindent
(b) For any $1\leq i<n-1$, we have:
\begin{equation}
\label{eq:iterated-1.2}
  e_{i,n+1}(z)=(1-v^2)^{i-n+1}\cdot
  [[\cdots[e_{i,i+1}(z),e^{(0)}_{i+1,i+2}]_{v},\cdots,e^{(0)}_{n-2,n-1}]_{v},e^{(0)}_{n-1,n+1}]_{v}.
\end{equation}

\noindent
(c) For any $1\leq i<j<n$, we have:
\begin{multline}
\label{eq:iterated-2}
  e_{i,j'}(z)=(1-v^2)^{i+j-2n+1}(-1)^{j-n}\times \\
  [\cdots[[[\cdots[e_{i,i+1}(z),e^{(0)}_{i+1,i+2}]_{v},\cdots, e^{(0)}_{n-2,n-1}]_{v},e^{(0)}_{n-1,n+1}]_{v},
   e^{(0)}_{n-1,n}]_{v},\cdots,e^{(0)}_{j,j+1}]_{v}.
\end{multline}
\end{Prop}

\begin{proof}
Due to the ``rank reduction'' embeddings of~\cite[\S3.2, Proposition 4.2]{JLM20}, it suffices to
prove formulas~\eqref{eq:iterated-1}--\eqref{eq:iterated-2} for $i=1$. In fact, both \eqref{eq:iterated-1}
and~\eqref{eq:iterated-2} for $i=1$ are proved exactly as~\cite[(A.13, A.14)]{HT24}. Thus, we shall only
provide details for $i=1$ case of~\eqref{eq:iterated-1.2}.

Comparing matrix coefficients $\langle v_1\otimes v_{n-1}|\cdots|v_{n-1}\otimes v_{n+1}\rangle$ of both sides of
the RTT relation $\bar{R}_{\trig}(z/w)\Lf^-_1(z)\Lf^-_2(w)=\Lf^-_2(w)\Lf^-_1(z)\bar{R}_\trig(z/w)$, we get:
\begin{multline*}
   \frac{z-w}{vz-v^{-1}w} \ell^-_{1,n-1}(z)\ell^-_{n-1,n+1}(w)+\frac{(v-v^{-1})z}{vz-v^{-1}w}\ell^-_{n-1,n-1}(z)\ell^-_{1,n+1}(w)=\\
   \frac{z-w}{vz-v^{-1}w} \ell^-_{n-1,n+1}(w)\ell^-_{1,n-1}(z)+\frac{(v-v^{-1})w}{vz-v^{-1}w}\ell^-_{n-1,n-1}(w)\ell^-_{1,n+1}(z).
\end{multline*}
Expanding all rational factors as series in $z/w$ and evaluating the $[w^0]$-coefficients, we obtain:
\begin{equation}
\label{eq:rtt-2}
  v \ell^-_{1,n-1}(z)\ell^-_{n-1,n+1}[0]=v\ell^-_{n-1,n+1}[0]\ell^-_{1,n-1}(z)+(1-v^2)\ell^-_{n-1,n-1}[0]\ell^-_{1,n+1}(z).
\end{equation}
Comparing matrix coefficients $\langle v_1\otimes v_{n-1}|\cdots|v_{n-1}\otimes v_{n-1}\rangle$ of the same RTT relation, we get:
\begin{equation*}
   \frac{z-w}{vz-v^{-1}w} \ell^-_{1,n-1}(z)\ell^-_{n-1,n-1}(w) +
   \frac{(v-v^{-1})z}{vz-v^{-1}w}\ell^-_{n-1,n-1}(z)\ell^-_{1,n-1}(w) =
   \ell^-_{n-1,n-1}(w)\ell^-_{1,n-1}(z).
\end{equation*}
Expanding both rational factors as series in $z/w$ and evaluating the $[w^0]$-coefficients, we obtain:
\begin{equation}
\label{eq:rtt-4}
   {\ell^-_{n-1,n-1}[0]}^{-1}\ell^-_{1,n-1}(z)=v^{-1} \ell^-_{1,n-1}(z){\ell^-_{n-1,n-1}[0]}^{-1}.
\end{equation}
Multiplying both sides of~\eqref{eq:rtt-2} by ${\ell^-_{n-1,n-1}[0]}^{-1}$ on the left and applying \eqref{eq:rtt-4}, we obtain:
\begin{equation*}
  (1-v^2)\ell^-_{1,n+1}(z)=[\ell^-_{1,n-1}(z),e^{(0)}_{n-1,n+1}]_v.
\end{equation*}
As $\ell^-_{1,n+1}(z)=h^-_1(z)e^-_{1,n+1}(z)$, $\ell^-_{1,n-1}(z)=h^-_1(z)e^-_{1,n-1}(z)$, and $[h^-_1(z),e^{(0)}_{n-1,n+1}]=0$, we get:
\begin{equation}
\label{eq:rtt-6}
  e^-_{1,n+1}(z)=(1-v^2)^{-1} \cdot [e^-_{1,n-1}(z),e^{(0)}_{n-1,n+1}]_v.
\end{equation}
Arguing in the same way, but using $\bar{R}_{\trig}(z/w)\Lf^+_1(z)\Lf^-_2(w)=\Lf^-_2(w)\Lf^+_1(z)\bar{R}_\trig(z/w)$ instead,
we also obtain:
\begin{equation}
\label{eq:rtt-6'}
  e^+_{1,n+1}(z)=(1-v^2)^{-1} \cdot [e^+_{1,n-1}(z),e^{(0)}_{n-1,n+1}]_v.
\end{equation}
Subtracting \eqref{eq:rtt-6} from \eqref{eq:rtt-6'}, we finally get:
\begin{equation*}
  e_{1,n+1}(z)=(1-v^2)^{-1} \cdot [e_{1,n-1}(z),e^{(0)}_{n-1,n+1}]_v.
\end{equation*}
Applying formula~\eqref{eq:iterated-1} for $e_{1,n-1}(z)$ completes our proof of \eqref{eq:iterated-1.2} for $i=1$.
\end{proof}

Combining Proposition~\ref{prop:E-explicitly} with identification~(\ref{JLM formulas D1},~\ref{JLM formulas D2}) and
formulas~(\ref{eq:rtt-root-vectors}), we get:

\begin{Cor}
For any $1\leq i < j< n$ and $s\in \BZ$, we have:
\begin{equation*}
  \varrho(\tilde{\mathcal{E}}^{\rtt}_{[i,j],s}) \doteq e^{(s)}_{i,j+1} ,\qquad
  \varrho(\tilde{\mathcal{E}}^{\rtt}_{[i,n],s}) \doteq e^{(s)}_{i,n+1} ,\qquad
  \varrho(\tilde{\mathcal{E}}^{\rtt}_{[i,n,j],s}) \doteq e^{(s)}_{i,j'}.
\end{equation*}
\end{Cor}

Since the elements~\eqref{eq:rtt-root-vectors} are specific case of quantum root vectors~\eqref{eq:rtt-vectors-D},
we finally obtain:

\begin{Prop}
$\varrho(\integrald)=\rtU^{\rtt,>}_v(L\sso_{2n})$.
\end{Prop}

This result explains why we called $\integrald$ the RTT integral form of $\qld$. Moreover, Theorem~\ref{PBWDintegralrtt}(b)
implies the PBWD theorem for $\rtU^{\rtt,>}_v(L\sso_{2n})$, cf.~\cite[Theorem~3.25]{FT19}:

\begin{Cor}
The ordered monomials in $\big\{e^{(r)}_{ij} \,|\ i<j \ \mathrm{such\ that}\ i+j\leq N, r\in \BZ\big\}$ form a
basis of a free $\BZ[v,v^{-1}]$-module $\rtU^{\rtt,>}_v(L\sso_{2n})$, where the ordering is given by
$e^{(r)}_{ij}\leq e^{(s)}_{k\ell}$ if $i<k$, or $i=k$ and $j<\ell$, or $i=k,j=\ell$ and $r\leq s$.
\end{Cor}


\subsection{RTT realization of $U_{v}(L\mathfrak{sp}_{2n})$}

Set $N=2n$. For $1\leq i\leq N$, we amend~\eqref{eq:bar} via:
\begin{equation*}
  (\bar{1},\ldots,\bar{N}):=(n,\ldots,2, 1, -1, -2, \ldots, -n),
\end{equation*}
while $i'$ is defined via~\eqref{eq:prime}. We define $\xi=v^{2-N}$ as before. Finally we also introduce:
\begin{equation*}
  \varepsilon_i=1 ,\qquad \varepsilon_{i'}=-1 \qquad \forall\, 1\leq i\leq n.
\end{equation*}
The corresponding trigonometric $R$-matrix $\bar{R}_\trig(z)$ (satisfying~\eqref{qYB}) is still given by~\eqref{trigR},
but $P,Q,R\in (\End\, \BC^N)^{\otimes 2}$ are now modified as follows:
\begin{equation*}
\begin{split}
  & P\ =\sum_{1\leq i,j\leq N} e_{ij}\otimes e_{ji},\qquad
    Q\ =\sum_{1\leq i,j\leq N} v^{\bar{i}-\bar{j}} \varepsilon_i\varepsilon_j\, e_{i'j'}\otimes e_{ij},\\
  & R\ = v\sum_{1\leq i\leq N} e_{ii}\otimes e_{ii} \, +
    \sum_{1\leq i,j\leq N}^{i\ne j,j'} e_{ii}\otimes e_{jj} + v^{-1} \sum_{1\leq i\leq N} e_{ii}\otimes e_{i'i'} \ + \\
  & \qquad \quad (v-v^{-1})\sum_{i<j} e_{ij}\otimes e_{ji} -
    (v-v^{-1})\sum_{i>j} v^{\bar{i}-\bar{j}} \varepsilon_i\varepsilon_j\, e_{i'j'}\otimes e_{ij}.
\end{split}
\end{equation*}

Define the \textbf{RTT integral form of the quantum loop algebra of $\mathfrak{sp}_{N}$}, denoted by
$\rtU^\rtt_v(L\mathfrak{sp}_{N})$, to be the associative $\BZ[v,v^{-1}]$-algebra generated by
$\{\ell^\pm_{ij}[\mp r]\}_{1\leq i,j\leq N}^{r\in \BN}$ with the same defining relations~(\ref{affRTT},~\ref{affRTT-extra}),
whereas ${\mathrm{t}}$ is now defined via $E^{\mathrm{t}}_{ij}=\varepsilon_i\varepsilon_j E_{j'i'}$.
Here, the generators are encoded via $\Lf^\pm(z)\in \rtU^\rtt_v(L\mathfrak{sp}_N)[[z^{\pm 1}]]\otimes \End\, \BC^N$
defined as in~\eqref{eq:L-matrix}. We also define the $\BC(v)$-counterpart
$U^\rtt_v(L\mathfrak{sp}_N):=\rtU^\rtt_v(L\mathfrak{sp}_N)\otimes_{\BZ[v,v^{-1}]} \BC(v)$.

Let $U_v(L\mathfrak{sp}_N)$ be the quantum loop algebra of type $C_n$ in the new Drinfeld realization. It is a
$\BC(v)$-algebra generated by $\{x^\pm_{i,r},\varphi_{i,-k},\psi_{i,k},k_i^{\pm 1}\}_{1\leq i\leq n}^{r\in \BZ,k>0}$
with the relations as in~\cite[\S1]{JLM21}. Identifying $x^+_{i,r}$ with our $e_{i,r}$, the subalgebra generated by
$\{x^+_{i,r}\}_{1\leq i\leq n}^{r\in \BZ}$ recovers our $U^>_v(L\mathfrak{sp}_N)$.

The relation between the algebras $U_v(L\mathfrak{sp}_N)$ and $\rtU^\rtt_v(L\mathfrak{sp}_{N})$ was established in~\cite{JLM21}.
Evoking the generating series~\eqref{eq:loop-series} and the Gauss decomposition of $\Lf^\pm(z)$, we have:

\begin{Thm}[\cite{JLM21}]\label{thm:JLM-iso-C}
There is a unique $\BC(v)$-algebra isomorphism
\begin{equation*}
  \varrho \colon U_v(L\mathfrak{sp}_N) \, \iso \, U^\rtt_v(L\mathfrak{sp}_{N})
\end{equation*}
defined by
\begin{equation}
\label{JLM formulas C1}
\begin{split}
  & x^+_i(z)\mapsto \frac{e^+_{i,i+1}(zv^i)-e^-_{i,i+1}(zv^i)}{v-v^{-1}}, \quad
    x^-_i(z) \mapsto \frac{f^+_{i+1,i}(zv^i)-f^-_{i+1,i}(zv^i)}{v-v^{-1}}, \\
  & \quad \psi_i(z)\mapsto h^-_{i+1}(zv^i)h^-_{i}(zv^i)^{-1},\qquad
    \varphi_i(z)\mapsto h^+_{i+1}(zv^i)h^+_{i}(zv^i)^{-1}
\end{split}
\end{equation}
for $1\leq i<n$ and
\begin{equation}
\label{JLM formulas C2}
\begin{split}
  & x^+_n(z)\mapsto \frac{e^+_{n,n+1}(zv^{n+1})-e^-_{n,n+1}(zv^{n+1})}{v^2-v^{-2}}, \ \
    x^-_n(z) \mapsto \frac{f^+_{n+1,n}(zv^{n+1})-f^-_{n+1,n}(zv^{n+1})}{v^2-v^{-2}}, \\
  & \psi_n(z)\mapsto h^-_{n+1}(zv^{n+1})h^-_{n}(zv^{n+1})^{-1},\quad
    \varphi_n(z)\mapsto h^+_{n+1}(zv^{n+1})h^+_{n}(zv^{n+1})^{-1}.
\end{split}
\end{equation}
\end{Thm}


\subsection{The RTT realization of $\integralc$}

Let $\rtU^{\rtt,>}_v(L\mathfrak{sp}_{N})$ be the $\BZ[v,v^{-1}]$-subalgebra of $\rtU^\rtt_v(L\mathfrak{sp}_{N})$
generated by the coefficients of $\{e^\pm_{ij}(z)\}_{1\leq i<j\leq N}$, the matrix coefficients of $E^\pm(z)$. The key goal
of this Appendix is to highlight the natural origin of the integral form $\integralc$ introduced in Definition~\ref{def:rttintegral}
and its specific quantum root vectors (a special case of~\eqref{eq:rtt-vectors-C})
\begin{equation}
\label{eq:rtt-root-vectors-C}
\begin{split}
  & \tilde{\mathcal{E}}^{\rtt}_{[n],s}:= \langle 2\rangle_{v}\cdot  e_{n,s},\\
  & \tilde{\mathcal{E}}^{\rtt}_{[i,j],s}:= \langle 1\rangle_{v}\cdot
    [\cdots[[e_{i,s},e_{i+1,0}]_{v},e_{i+2,0}]_{v},\cdots,e_{j,0}]_{v},\\
  & \tilde{\mathcal{E}}^{\rtt}_{[i,n],s}:= \langle 1\rangle_{v}\cdot
    [[\cdots[e_{i,s},e_{i+1,0}]_{v},,\cdots,e_{n-1,0}]_{v},e_{n,0}]_{v^2},\\
  & \tilde{\mathcal{E}}^{\rtt}_{[i,n,j],s}:= \langle 1\rangle_{v}\cdot
    [\cdots[[[\cdots[e_{i,s},e_{i+1,0}]_{v},\cdots, e_{n-1,0}]_{v}, e_{n,0}]_{v^2},e_{n-1,0}]_{v},\cdots,e_{j,0}]_{v},
\end{split}
\end{equation}
for any $1\leq i<j<n$, while the root generators $\tilde{\mathcal{E}}^{\rtt}_{[i,n,i],s}$ are defined slightly differently via:
\begin{equation}
\label{eq:rtt-root-vectors-C-long}
  \tilde{\mathcal{E}}^{\rtt}_{[i,n,i],s} :=
  \frac{-1}{1-v^2}[\tilde{\mathcal{E}}^{\rtt}_{[i,n-1],s},\tilde{\mathcal{E}}^{\rtt}_{[i,n],0}]-
  \sum^{\mathrm{same\ sign}}_{a+b=s}
  \tilde{\mathcal{E}}^{\rtt}_{[i,n],a} \tilde{\mathcal{E}}^{\rtt}_{[i,n-1],b},
\end{equation}
where the condition ``$\mathrm{same\ sign}$'' in the sum means that $a,b\leq 0$ if $s\leq 0$, and $a,b>0$ if $s>0$.

We also express the matrix coefficients of $E^\pm(z)$ as series in $z^{\pm 1}$:
\begin{equation*}
  e^+_{ij}(z)=\sum_{r>0} e^{(-r)}_{ij} z^r,\qquad  e^-_{ij}(z)=\sum_{r\geq 0} e^{(r)}_{ij} z^{-r}
  \qquad \forall\, 1\leq i<j\leq N,
\end{equation*}
and define $e_{ij}(z):=e^+_{ij}(z)-e^-_{ij}(z)$. The key technical result of this subsection is:

\begin{Prop}\label{prop:E-explicitly-C}
(a) For any $1\leq i<j<n$, we have:
\begin{equation}
\label{eq:iterated-1-C}
  e_{i,j+1}(z)=(1-v^2)^{i-j}\cdot
  [\cdots[[e_{i,i+1}(z),e^{(0)}_{i+1,i+2}]_{v},e^{(0)}_{i+2,i+3}]_{v},\cdots,e^{(0)}_{j,j+1}]_{v}.
\end{equation}

\noindent
(b) For any $1\leq i<n$, we have:
\begin{equation}
\label{eq:iterated-1.2-C}
  e_{i,n+1}(z)=(1-v^4)^{-1}(1-v^2)^{i-n+1}\cdot
  [[\cdots[e_{i,i+1}(z),e^{(0)}_{i+1,i+2}]_{v},\cdots, e^{(0)}_{n-1,n}]_{v}, e^{(0)}_{n,n+1}]_{v^2}.
\end{equation}

\noindent
(c) For any $1\leq i<j<n$, we have:
\begin{multline}
\label{eq:iterated-2-C}
  e_{i,j'}(z)=(1-v^4)^{-1}(1-v^2)^{i+j-2n+1}(-1)^{j-n}\times \\
  [\cdots[[[\cdots[e_{i,i+1}(z),e^{(0)}_{i+1,i+2}]_{v},\cdots, e^{(0)}_{n-1,n}]_{v}, e^{(0)}_{n,n+1}]_{v^2},
  e^{(0)}_{n-1,n}]_{v},\cdots,e^{(0)}_{j,j+1}]_{v}.
\end{multline}

\noindent
(d) For any $1\leq i<n$, we have:
\begin{equation}
\label{eq:iterated-3-C}
  e^\pm_{i,i'}(z)=\frac{-1}{1-v^2}[e^\pm_{in}(z),e^{(0)}_{i,n+1}]-e^\pm_{i,n+1}(z)e^\pm_{in}(z).
\end{equation}
\end{Prop}

\begin{proof}
Due to the ``rank reduction'' embeddings of~\cite[\S3.3, Proposition 4.2]{JLM21}, it suffices to prove
formulas~\eqref{eq:iterated-1-C}--\eqref{eq:iterated-3-C} for $i=1$. In fact,
\eqref{eq:iterated-1-C}--\eqref{eq:iterated-2-C} for $i=1$ are proved completely analogously
to~\cite[(A.13, A.14)]{HT24}. Thus, we shall only provide details for $i=1$ case of~\eqref{eq:iterated-3-C}.

Comparing matrix coefficients $\langle v_1\otimes v_{n}|\cdots|v_{n}\otimes v_{1'}\rangle$ of both sides of the
RTT relation $\bar{R}_{\trig}(z/w)\Lf^-_1(z)\Lf^-_2(w)=\Lf^-_2(w)\Lf^-_1(z)\bar{R}_\trig(z/w)$, we get:
\begin{multline*}
   \frac{z-w}{vz-v^{-1}w} \ell^-_{1,n}(z)\ell^-_{n,1'}(w)+\frac{(v-v^{-1})z}{vz-v^{-1}w}\ell^-_{n,n}(z)\ell^-_{1,1'}(w)=\\
   \frac{z-w}{vz-v^{-1}w} \ell^-_{n,1'}(w)\ell^-_{1,n}(z)+\frac{(v-v^{-1})w}{vz-v^{-1}w}\ell^-_{n,n}(w)\ell^-_{1,1'}(z).
\end{multline*}
Expanding all rational factors as series in $z/w$ and evaluating the $[w^0]$-coefficients, we obtain:
\begin{equation}
\label{eq:rtt-2-C}
  v \ell^-_{1,n}(z)\ell^-_{n,1'}[0]=v\ell^-_{n,1'}[0]\ell^-_{1,n}(z)+(1-v^2)\ell^-_{n,n}[0]\ell^-_{1,1'}(z).
\end{equation}
Comparing matrix coefficients $\langle v_1\otimes v_{n}|\cdots|v_{n}\otimes v_{n}\rangle$ of the same RTT relation, we get:
\begin{equation*}
   \frac{z-w}{vz-v^{-1}w} \ell^-_{1,n}(z)\ell^-_{n,n}(w)+\frac{(v-v^{-1})z}{vz-v^{-1}w}\ell^-_{n,n}(z)\ell^-_{1,n}(w)=
   \ell^-_{n,n}(w)\ell^-_{1,n}(z).
\end{equation*}
Expanding both rational factors as series in $z/w$ and evaluating the $[w^0]$-coefficients, we obtain:
\begin{equation*}
   {\ell^-_{n,n}[0]}^{-1} \ell^-_{1,n}(z)=v^{-1} \ell^-_{1,n}(z){\ell^-_{n,n}[0]}^{-1},
\end{equation*}
which after left multiplication by $(\ell_{1,1}(z))^{-1}=(h^-_1(z))^{-1}$ yields:
\begin{equation}
\label{eq:rtt-4-C}
   {\ell^-_{n,n}[0]}^{-1} e^-_{1,n}(z)=v^{-1} e^-_{1,n}(z){\ell^-_{n,n}[0]}^{-1}.
\end{equation}
Comparing matrix coefficients $\langle v_1\otimes v_{1}|\cdots|v_{1}\otimes v_{n+1}\rangle$ of the same RTT relation, we get:
\begin{equation*}
   \ell^-_{1,1}(z)\ell^-_{1,n+1}(w)=
   \frac{z-w}{vz-v^{-1}w} \ell^-_{1,n+1}(w)\ell^-_{1,1}(z)+\frac{(v-v^{-1})w}{vz-v^{-1}w}\ell^-_{1,1}(w)\ell^-_{1,n+1}(z).
\end{equation*}
Expanding both rational factors as series in $z/w$ and evaluating the $[w^0]$-coefficients, we obtain:
\begin{equation*}
   v  \ell^-_{1,n+1}[0] \ell^-_{1,1}(z)=\ell^-_{1,1}(z) \ell^-_{1,n+1}[0] - (1-v^2)\ell^-_{1,1}[0] \ell^-_{1,n+1}(z),
\end{equation*}
which after left multiplication by $(\ell_{1,1}[0])^{-1}$ and evoking $\ell^-_{1,1}(z)=h^-_1(z)$ yields:
\begin{equation}
\label{eq:rtt-4.2-C}
   v e^{(0)}_{1,n+1} h^-_1(z)=h^-_1(z)\left(e^{(0)}_{1,n+1}-(1-v^2)e^-_{1,n+1}(z)\right).
\end{equation}
Plugging~(\ref{eq:rtt-4-C},~\ref{eq:rtt-4.2-C}) into~\eqref{eq:rtt-2-C} and evoking $e^{(0)}_{n,1'}=-e^{(0)}_{1,n+1}$,
we obtain the desired formula:
\begin{equation*}
  e^-_{1,1'}(z)=\frac{-1}{1-v^2}[e^-_{1,n}(z),e^{(0)}_{1,n+1}]-e^-_{1,n+1}(z)e^-_{1,n}(z).
\end{equation*}
Arguing in the same way, but using
$\bar{R}_{\trig}(z/w)\Lf^+_1(z)\Lf^-_2(w)=\Lf^-_2(w)\Lf^+_1(z)\bar{R}_\trig(z/w)$ instead, we also obtain
a similar formula for $e^+_{1,1'}(z)$. This completes our proof of \eqref{eq:iterated-3-C} for $i=1$.
\end{proof}

Combining Proposition~\ref{prop:E-explicitly-C} with~(\ref{JLM formulas C1},~\ref{JLM formulas C2})
and~(\ref{eq:rtt-root-vectors-C},~\ref{eq:rtt-root-vectors-C-long}), we get:

\begin{Cor}
For any $1\leq i < j< n$ and $s\in \BZ$, we have:
\begin{equation*}
  \varrho(\tilde{\mathcal{E}}^{\rtt}_{[i,j],s}) \doteq e^{(s)}_{i,j+1} ,\quad
  \varrho(\tilde{\mathcal{E}}^{\rtt}_{[i,n],s}) \doteq e^{(s)}_{i,n+1} ,\quad
  \varrho(\tilde{\mathcal{E}}^{\rtt}_{[i,n,j],s}) \doteq e^{(s)}_{i,j'} ,\quad
  \varrho(\tilde{\mathcal{E}}^{\rtt}_{[i,n,i],s}) \doteq e^{(s)}_{i,i'}.
\end{equation*}
\end{Cor}

The following result explains why we called $\integralc$ the RTT integral form of $\qlc$:

\begin{Prop}
$\varrho(\integralc)=\rtU^{\rtt,>}_v(L\mathfrak{sp}_{2n})$.
\end{Prop}

\begin{proof}
We note that $\tilde{\mathcal{E}}^{\rtt}_{\beta,s}$ of~\eqref{eq:rtt-root-vectors-C} coincide with
$\tilde{\mathcal{E}}^{+}_{\beta,s}$ from~\eqref{eq:rtt-vectors-C} corresponding to $s_i=s$ and $s_{\ne i}=0$
in the formulas~\eqref{rvc1}--\eqref{rvc3}, for all roots except $\beta=[i,n,i]\ (1\leq i<n)$. While
$\tilde{\mathcal{E}}^{\rtt}_{[i,n,i],s}$ and $\tilde{\mathcal{E}}^{+}_{[i,n,i],s}$ differ, we claim
that they generate the same $\BZ[v,v^{-1}]$-subalgebra together with the elements above. To this end,
it is convenient to replace $\tilde{\mathcal{E}}^{\rtt}_{[i,n,i],s}$ rather with
\[
  \tilde{\mathcal{E}}'_{[i,n,i],s} :=
  \frac{-1}{1-v^2}[\tilde{\mathcal{E}}^{\rtt}_{[i,n-1],s},\tilde{\mathcal{E}}^{\rtt}_{[i,n],0}],
\]
as the elements $\tilde{\mathcal{E}}^{\rtt}_{[i,n],a}$, $\tilde{\mathcal{E}}^{\rtt}_{[i,n-1],b}$ featured
in \eqref{eq:rtt-root-vectors-C-long} belong to $\integralc$ for any $a,b\in\BZ$.

First, let us show that $\tilde{\mathcal{E}}'_{[i,n,i],s}$ belongs to $\integralc$, or equivalently that
$\Psi(\tilde{\mathcal{E}}'_{[i,n,i],s})$ belongs to $\mathcal{S}$ of Subsection \ref{rttc}, due to Theorem~\ref{rttthm-C}(a).
To this end, we set
\begin{align*}
  & \sA= \prod_{\ell=i}^{n-2} \left\{\zeta\left(\frac{x_{\ell,1}}{x_{\ell,2}}\right)
    \zeta\left(\frac{x_{\ell,1}}{x_{\ell+1,2}}\right)\zeta\left(\frac{x_{\ell+1,1}}{x_{\ell,2}}\right)\right\}\cdot
    \zeta\left(\frac{x_{n-1,1}}{x_{n-1,2}}\right)\zeta\left(\frac{x_{n-1,1}}{x_{n,1}}\right) , \\
  & \sB= \prod_{\ell=i}^{n-2}\left\{\zeta\left(\frac{x_{\ell,2}}{x_{\ell,1}}\right)
    \zeta\left(\frac{x_{\ell+1,2}}{x_{\ell,1}}\right)\zeta\left(\frac{x_{\ell,2}}{x_{\ell+1,1}}\right)\right\}\cdot
    \zeta\left(\frac{x_{n-1,2}}{x_{n-1,1}}\right)\zeta\left(\frac{x_{n,1}}{x_{n-1,1}}\right),
\end{align*}
so that
\[
  \Psi(\tilde{\mathcal{E}}'_{[i,n,i],s})\doteq
  \langle 1\rangle^{2n-2i-1}_{v} \langle 2\rangle_{v}\cdot
  \mathop{Sym}\left(\frac{x^{s}_{i,1}x^{-1}_{n-1,1} \cdot \prod_{k=i}^{n-1}x_{k,1}x_{k,2}\cdot (\sA-\sB)}
    {\mathrm{denom}_{[i,n-1]}(\{x_{k,1}\}_{k=i}^{n-1})\cdot \mathrm{denom}_{[i,n]}(\{x_{k,2}\}_{k=i}^{n-1},x_{n,1})} \right),
\]
where $\mathop{Sym}$ denotes symmetrization with respect to all pairs $\{x_{k,1},x_{k,2}\}_{k=i}^{n-1}$. Since
\begin{align*}
  & \zeta\left(\frac{x_{n-1,1}}{x_{n,1}}\right) - \zeta\left(\frac{x_{n,1}}{x_{n-1,1}}\right),\qquad
    \zeta\left(\frac{x_{\ell,1}}{x_{\ell,2}}\right) - \zeta\left(\frac{x_{\ell,2}}{x_{\ell,1}}\right) \ \  (i\leq \ell \leq n-1), \\
  & \zeta\left(\frac{x_{\ell,1}}{x_{\ell+1,2}}\right)\zeta\left(\frac{x_{\ell+1,1}}{x_{\ell,2}}\right) -
    \zeta\left(\frac{x_{\ell+1,2}}{x_{\ell,1}}\right)\zeta\left(\frac{x_{\ell,2}}{x_{\ell+1,1}}\right)\ \ (i\leq \ell \leq n-2)
\end{align*}
are all divisible by $\langle 1\rangle_{v}$, we see that so is $\sA-\sB$. Hence, $\Psi(\tilde{\mathcal{E}}'_{[i,n,i],s})$
satisfies the condition~\eqref{rttconstant1-C}.

Next, we show that for any $\unl{d}\in\text{KP}(\unl{k})$ with $\unl{k}=2\alpha_{i}+\cdots+2\alpha_{n-1}+\alpha_{n}$,
the specialization $\phi_{\unl{d}}(\Psi(\tilde{\mathcal{E}}'_{[i,n,i],s}))$ is divisible by $A_{\unl{d}}$ of \eqref{rttconstant2-C}.
If $\unl{d}=\unl{d}_{0}=\{d_{[i,n,i]}=1,\, d_{\gamma}=0\ \text{for other}\ \gamma\}$, then
\begin{equation}
\label{eq:long-spec-rtt}
  \phi_{\unl{d}_{0}}(\Psi(\tilde{\mathcal{E}}'_{[i,n,i],s}))\doteq
    \langle 1\rangle^{2n-2i-1}_{v} \langle 2\rangle^{2}_{v}\cdot w_{\beta,1}^{s+2n-2i}
\end{equation}
by Lemma \ref{phirv-C}, so that $\Psi(\tilde{\mathcal{E}}'_{[i,n,i],s})$ is non-zero and
$\phi_{\unl{d}_{0}}(\Psi(\tilde{\mathcal{E}}'_{[i,n,i],s}))$ satisfies the condition~\eqref{rttconstant2-C}.
For any $\unl{d}>\unl{d}_{0}$, arguing as in the proof of Proposition~\ref{goodC}, we see that the $\zeta$-factors
arising from the variables $x^{(\beta,s)}_{*,*}$ with $\beta=[i,n,j]$ and $d_{\beta}>0$ contribute $A_{\unl{d}}$
in the $\phi_{\unl{d}}$-specialization (since $o(x^{(\beta,s)}_{\ell,1})\ne o(x^{(\beta,s)}_{\ell,2})$ in the present setup
of $\Psi(\tilde{\mathcal{E}}^{\rtt}_{[i,n-1],s} \tilde{\mathcal{E}}^{\rtt}_{[i,n],0})$ and
$\Psi(\tilde{\mathcal{E}}^{\rtt}_{[i,n],0} \tilde{\mathcal{E}}^{\rtt}_{[i,n-1],s})$, we actually never have to reserve to
the $Q$-factors of~\eqref{eq:Qform-C} or the factors~\eqref{eq:otherfactors} that were utilized a few times in the proof
of Proposition~\ref{goodC}, and thus the overall contribution of $A_{\unl{d}}$ arises precisely from the same $\zeta$-factors
as used in the proof of Proposition~\ref{goodC}).

Finally, if we expand $\tilde{\mathcal{E}}'_{[i,n,i],s}$ as a linear combination of monomials
$\prod_{\ell=1}^{k}e_{i_\ell,s_{\ell}}$ with coefficients in $\BZ[v,v^{-1}]$, then as in the proof of
Proposition~\ref{goodC} we also see that $\Psi(\tilde{\mathcal{E}}'_{[i,n,i],s})$ is integral. Thus
$\Psi(\tilde{\mathcal{E}}'_{[i,n,i],s})\in \mathcal{S}$, so that $\tilde{\mathcal{E}}'_{[i,n,i],s}\in\integralc$
by Theorem~\ref{rttthm-C}(a). On the other hand, combining~\eqref{eq:long-spec-rtt} with Lemma~\ref{vanishC} and
Theorem~\ref{rttthm-C}(b), we see that $\tilde{\mathcal{E}}'_{[i,n,i],s}-a\cdot \tilde{\mathcal{E}}^{+}_{[i,n,i],s}$
is a polynomial in $\tilde{\mathcal{E}}^{+}_{\beta,s}\ (|\beta|\leq 2n-2i)$ with coefficients in $\BZ[v,v^{-1}]$
for some $a\in \BQ^\times \cdot v^{\BZ}$.

This proves that the quantum root vectors $\{\tilde{\mathcal{E}}^{\rtt}_{\beta,s}\}^{s\in\BZ}_{\beta\in\Delta^{+}}$ indeed generate $\integralc$.
\end{proof}


\end{document}